\documentclass[11pt,a4paper]{article}

\usepackage[english]{babel}
\usepackage[utf8]{inputenc}
\usepackage{authblk}

\usepackage[T1]{fontenc}

\usepackage[absolute]{textpos}

\usepackage[vlined,ruled]{algorithm2e}
\usepackage[left=3cm,right=3cm,top=2cm,bottom=2cm]{geometry}
\SetKwInOut{Input}{input}
\SetKwInOut{Output}{output}
\SetKwComment{Comment}{}{}

\SetCommentSty{mycmtsty}

\usepackage[vlined,ruled]{algorithm2e}
\usepackage{graphicx}
\usepackage{subfigure}
\usepackage{amsmath, amsthm}
\usepackage{hyperref}
\usepackage{amsfonts}
\usepackage{float}
\usepackage{pstricks,pst-plot,pst-text,pst-tree,pst-eps,pst-fill,pst-node,pst-math}
\usepackage[justification = centering]{caption}
\usepackage{tikz}
\usepackage{setspace}
\usepackage{enumitem}

\usepackage{titlesec}
\titleformat{\chapter}[display]   
{\normalfont\huge\bfseries}{\chaptertitlename\ \thechapter}{20pt}{\Huge}   
\titlespacing*{\chapter}{0pt}{-30pt}{30pt}

\usepackage{fancyvrb}
\VerbatimFootnotes 

\usepackage{amsthm}
\usepackage{subfigure}
\usepackage{appendix}

\usepackage{mathtools}
\mathtoolsset{showonlyrefs,showmanualtags}

\newtheorem{theorem}{Theorem}
\newtheorem{dftn}{Definition}
\newtheorem{lemme}{Lemma}
\newtheorem{remark}{Remark}
\newtheorem{prop}{Proposition}
\newtheorem{corol}{Corollary}

\begin{document}

\begin{spacing}{1}

\title{\Huge A periodic homogenization problem with defects rare at infinity}

\author{\huge R\'emi Goudey}

\affil{ \large CERMICS, Ecole des Ponts and MATHERIALS Team, INRIA, \\ 6 \& 8, avenue Blaise Pascal, 77455 Marne-La-Vall\'ee Cedex 2, FRANCE.\\ \textbf{remi.goudey@enpc.fr}}

\date{}

\maketitle

\pagestyle{plain}
\bibliographystyle{abbrv}

\begin{abstract}
We consider a homogenization problem for the diffusion equation $-\operatorname{div}\left(a_{\varepsilon} \nabla u_{\varepsilon} \right) = f$ when the  coefficient $a_{\varepsilon}$ is a non-local perturbation of a periodic coefficient. The perturbation does not vanish but becomes rare at infinity in a sense made precise in the text. We prove the existence of a corrector, identify the homogenized limit and study the convergence rates of $u_{\varepsilon}$ to its homogenized limit. 
\end{abstract}




\section{Introduction}

\subsection{Motivation}

The purpose of this paper is to address the homogenization problem for a second order elliptic equation in divergence form with a certain class of oscillating coefficients : 
\begin{equation}
\label{equationepsilon}
\left\{
\begin{array}{cc}
   -\operatorname{div}(a(x/\varepsilon) \nabla u ^{\varepsilon}) = f   & \text{in } \Omega, \\
    u^{\varepsilon}(x) = 0 & \text{in } \partial \Omega,
\end{array}
\right.
\end{equation}
where $\Omega$ is a bounded domain of $\mathbb{R}^d$ ($d \geq 1$) sufficiently regular (the regularity will be made precise later on) and $f$ is a function in $L^2(\Omega)$. The class of (matrix-valued) coefficients $a$ considered is that of the form 
\begin{equation}
\label{coefficientform}
a_{per} + \Tilde{a}, 
\end{equation}
which describes a periodic geometry encoded in the coefficient $a_{per}$ and perturbed by a coefficient $\Tilde{a}$ that represents a non-local perturbation (a "defect") that, although it does not vanish at infinity, becomes rare at infinity. More specifically, we consider coefficients~$\Tilde{a}$ that locally behave like $ L^2(\mathbb{R}^d)$ functions in the neighborhood of a set of points localized at an exponentially increasing distance from the origin. Formally, the coefficient~$\Tilde{a}$ is an infinite sum of localized perturbations, increasingly distant from one another. A prototypical one-dimensional example of such a defect reads as $\displaystyle\sum_{k \in \mathbb{Z}} \phi(x - \operatorname{sign}(k)2^{|k|})$ for some fixed $\phi \in \mathcal{D}(\mathbb{R})$, where $|k|$ denotes the absolute value of $k$ and $\operatorname{sign}(k)$ denotes its sign. It is depicted in Figure \ref{figf1} \\


\begin{figure}[h!]
\centering
\includegraphics[scale=0.36]{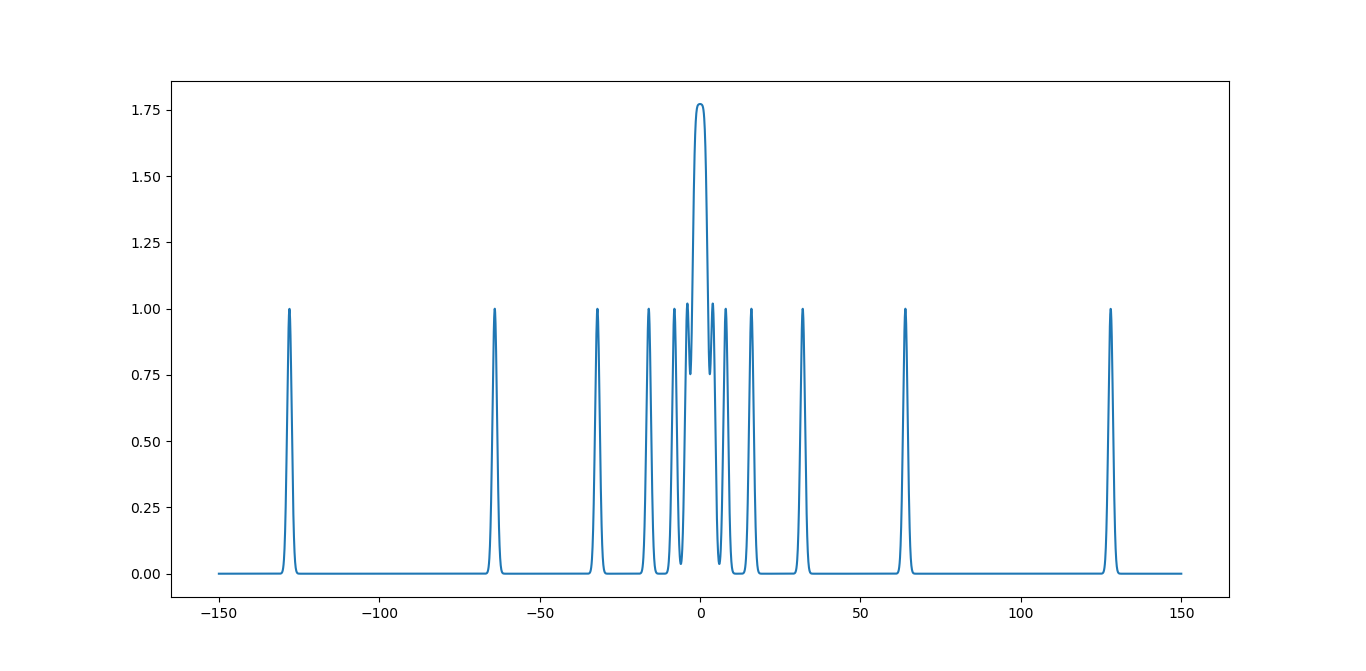}
\caption
{Prototype perturbation in dimension $d=1$.}
\label{figf1}
\end{figure}


Homogenization theory for the unperturbed periodic problem  \eqref{equationepsilon}-\eqref{coefficientform} when $\Tilde{a}=0$ is well-known (see for instance \cite{bensoussan2011asymptotic, jikov2012homogenization}). The solution $u^{\varepsilon}$ converges  strongly in $L^2(\Omega)$ and weakly in $H^1(\Omega)$ to $u^*$, solution to the homogenized problem : 
\begin{equation}
\label{homog}
\left\{
\begin{array}{cc}
  -\operatorname{div}(a^* \nabla u^*) = f   & \text{in } \Omega, \\
    u^{*}(x) = 0 & \text{in } \partial \Omega,
\end{array}
\right.
\end{equation}
where $a^*$ is a constant  matrix. The convergence in the $H^1(\Omega)$ norm is obtained upon introducing a corrector $w_{per,p}$ defined for all $p$ in $\mathbb{R}^d$ as the periodic solution (unique up to the addition of a constant) to :
\begin{equation}
\label{problemeperiodique}
    -\operatorname{div}(a_{per}(\nabla w_{per,p} + p)) = 0 \quad \text{in } \mathbb{R}^d. 
\end{equation}
This corrector allows to both make explicit the homogenized coefficient  
\begin{equation}
\label{homogenizedcoeff}   
    (a^*)_{i,j} = \int_{Q} e_i^T a_{per}(y) \left( e_j + \nabla w_{per,e_j} \right) dy, 
\end{equation}
(where $Q$ denotes the d-dimensional unit cube, $(e_i)$ the canonical basis of $\mathbb{R}^d$) and define the approximation 
\begin{equation}
\label{approximatesequence}
u^{\varepsilon,1} = u^*(.) + \varepsilon \displaystyle \sum_{i=1}^d \partial_{i}u^*(.) w_{per,e_i}(./ \varepsilon),
\end{equation}
such that $u^{\varepsilon,1} - u^{\varepsilon}$ strongly converges to $0$ in $H^1(\Omega)$ (see \cite{allaire1992homogenization} for more details). In addition, convergence rates can be made precise, with in particular :
\begin{align}
    & \|\nabla u^{\varepsilon} - \nabla u^{\varepsilon,1}\|_{L^2(\Omega)} \leq C \sqrt{\varepsilon}\|f\|_{L^2(\Omega)}, \\
    & \|\nabla u^{\varepsilon} - \nabla u^{\varepsilon,1}\|_{L^2(\Omega_1)} \leq C \varepsilon\|f\|_{L^2(\Omega)} \quad \text{for every } \Omega_1 \subset \subset \Omega,
\end{align}
for some constants independent of $f$. 

Our purpose here is to extend the above results to the setting of the \emph{perturbed} problem~\eqref{equationepsilon}-\eqref{coefficientform}. The main difficulty is that the corrector equation 
$$- \operatorname{div}\left( \left(a_{per} + \Tilde{a}\right) \left(\nabla w_p +p \right) \right)=0,$$ (formally obtained by a two-scale expansion (see again \cite{allaire1992homogenization} for the details) and analogous to \eqref{problemeperiodique} in the periodic case) is defined on the whole space $\mathbb{R}^d$ and cannot be reduced to an equation posed on a bounded domain, as is the case in periodic context in particular. This prevents us from using classical techniques. 
The present work follows up on some previous works \cite{blanc2018precised, blanc2018correctors, blanc2015local, blanc2012possible} where the authors have developed an homogenization theory in the case where $\Tilde{a} \in L^p(\mathbb{R}^d)$ for $p \in ]1, \infty[$. The existence and uniqueness (again up to an additive constant) of a corrector, the gradient of which shares the same structure
"periodic + $L^p$" as the coefficient $a$, is established. Convergence rates are also made precise. Similarly to \cite{blanc2018precised, blanc2018correctors, blanc2015local, blanc2012possible}, we aim to show here, in a context of a perturbation rare at infinity, there also exists a corrector (unique up to the addition of a constant), and such that its gradient has the structure \eqref{coefficientform} of the diffusion coefficient : it can be decomposed as a sum of the gradient of a periodic corrector and a gradient that becomes rare at infinity (in a sense similar to that for $\Tilde{a}$, and made precise below).

\subsection{Functional setting}

We introduce here a suitable functional setting to describe the class of defects we consider. 

In order to formalize our mathematical setting, we first  define a generic infinite discrete set of points denoted by $\mathcal{G} = \left\{x_p \right\}_{p \in \mathbb{Z}^d}$. In the sequel, each point $x_p$ actually models the presence of a defect in the periodic background modeled by $a_{per}$ and our aim is to ensure these defects are sufficiently rare at infinity.

We next introduce the Voronoi diagram associated with our set of points. For $x_p \in \mathcal{G}$, we denote by $V_{x_p}$ the Voronoi cell containing the point $x_p$ and defined by 
\begin{equation}
\label{VoronoiDEF}
    V_{x_p} =  \bigcap_{x_q \in \mathcal{G}\setminus{\{x_p\}}} \left\{ x \in \mathbb{R}^d  \middle| |x-x_p| \leq |x - x_q| \right\}. 
\end{equation}
We now consider three geometric assumptions that ensure an appropriate distribution of the points in the space. The set $\mathcal{G}$ is required to satisfy the following three conditions~: 

\begin{equation}
\label{H1}
\tag{H1} 
\forall x_p \in \mathcal{G}, \quad \left|V_{x_p} \right| < \infty,
\end{equation}

\begin{equation}
\label{H2}
\tag{H2} 
\exists C_1>0, \ C_2>0, \ \forall x_p \in \mathcal{G}, \quad C_1 \leq \frac{1 + \left|x_p \right|}{D\left(x_p, \mathcal{G} \setminus \left\{x_p \right\}\right)} \leq C_2,
\end{equation}

\begin{equation}
\label{H3}
\tag{H3} 
\exists C_3>0, \forall x_p \in \mathcal{G}, \quad \frac{Diam\left(V_{x_p}\right)}{D\left(x_p, \mathcal{G} \setminus \left\{x_p \right\}\right)} \leq C_3,
\end{equation}

where $|A|$ denotes the volume of a subset $A \subset \mathbb{R}^d$, $Diam(A)$ the diameter of $A$ and $D(.,.)$ the euclidean distance.

Assumption \eqref{H2} is the most significant assumption in our case since it implies that the points are increasingly distant from one another far from the origin. It in particular implies 
\begin{equation}
 \lim_{x_p \in \mathcal{G}, \ |x_p| \rightarrow \infty} D\left(x_p, \mathcal{G} \setminus \left\{x_p \right\}\right) = + \infty.
\end{equation}
More precisely, it ensures the distance between a point $x_p$ and the others has the same growth as the norm $|x_p|$ and, therefore, requires the Voronoi cell $V_{x_p}$ (which contains a ball of radius $\displaystyle \frac{D\left(x_p, \mathcal{G}\setminus{\{x_p\}}\right)}{2}$ as a consequence of its definition) to be sufficiently large. In particular, this assumption ensures that the defects modeled by the points $x_p$ are sufficiently rare at infinity. In particular, we show in Section \ref{Section2} that Assumption \eqref{H2} ensures that the number of points $x_p$ contained in a ball $B_R$ of radius $R>0$ is bounded by the logarithm of $R$. This property is an essential element for the methods used in the proof of this article.

In contrast to \eqref{H2}, Assumptions \eqref{H1} and \eqref{H3} are only technical and not very restrictive. They limit the size of the Voronoi cells. In the case where these assumptions are not satisfied, our main results of Theorems \ref{theoreme1} and \ref{theoreme3} stated below still hold. Their proofs have to be adapted, upon splitting the Voronoi cells in several subsets such that each subset satisfies geometric constraints similar to \eqref{H1}, \eqref{H2} and \eqref{H3}. To some extent, our assumptions \eqref{H1} and \eqref{H3} ensure we consider the worst case scenario, where the set $\mathcal{G}$ contains as many points as possible while satisfying \eqref{H2}.

In addition, although we establish in Section \ref{Section2} all the geometric properties satisfied by the Voronoi cells $V_{x_p}$ which are required in our approach to study the homogenization problem \eqref{equationepsilon}
with the whole generality of Assumptions \eqref{H1}, \eqref{H2} and \eqref{H3}, we choose, for the sake of illustration and for pedagogic purposes, to work with a particular set of points (for which the coordinates are powers of 2) and to establish our main results of homogenization in this specific setting. There are, of course, many alternative sets that satisfy \eqref{H1}, \eqref{H2} and \eqref{H3} but our specific choice is convenient. To define our specific set of points, we first introduce a constant~$C_0>1$ and a set of indices $\mathcal{P}_{C_0}$ defined by :
\begin{equation}
\label{defPC0}
\mathcal{P}_{C_0} = \left\{p \in \mathbb{Z}^d \ \middle| \  \displaystyle \max_{p_i \neq 0 }\left\{|p_i|\right\}\leq C_0 + \displaystyle \min_{p_i \neq 0}\left\{|p_i|\right\} \right\}.
\end{equation}
Our specific set of points (see Figure \ref{figf2}) is then defined by : 
\begin{equation}
    \label{defG}
   \mathcal{G}_{C_0} = \left\{x_p = \left(\operatorname{sign}(p_i)2^{|p_i|}\right)_{i \in \left\{1,...d\right\}} \ \middle| \ \left(p_1,...,p_d\right)\in \mathcal{P}_{C_0} \right\}.
\end{equation}
We use here the convention $\operatorname{sign}(0) = 0$. The set of indices \eqref{defPC0} contains only the points with integer coordinates on the axes $\operatorname{Span}\left(e_i\right)$ and the points close to each diagonal of the form $\operatorname{Span}\left( e_{i_1} + ... + e_{i_k} \right)$ for $k \in \left\{2,..,d \right\}$ and $\left(i_1,...,i_k \right) \in \left\{1,...,d\right\}^k$. In this way, the points of $\mathcal{G}_{C_0}$ are exponentially distant from each other with respect to the norm of $p$.  In Section \ref{Section2}, we show that the set $\mathcal{G}_{C_0}$ defined by \eqref{defG} indeed satisfies Assumptions \eqref{H1}, \eqref{H2} and \eqref{H3}. 


\begin{figure}[h!]
\centering
\includegraphics[scale=0.39]{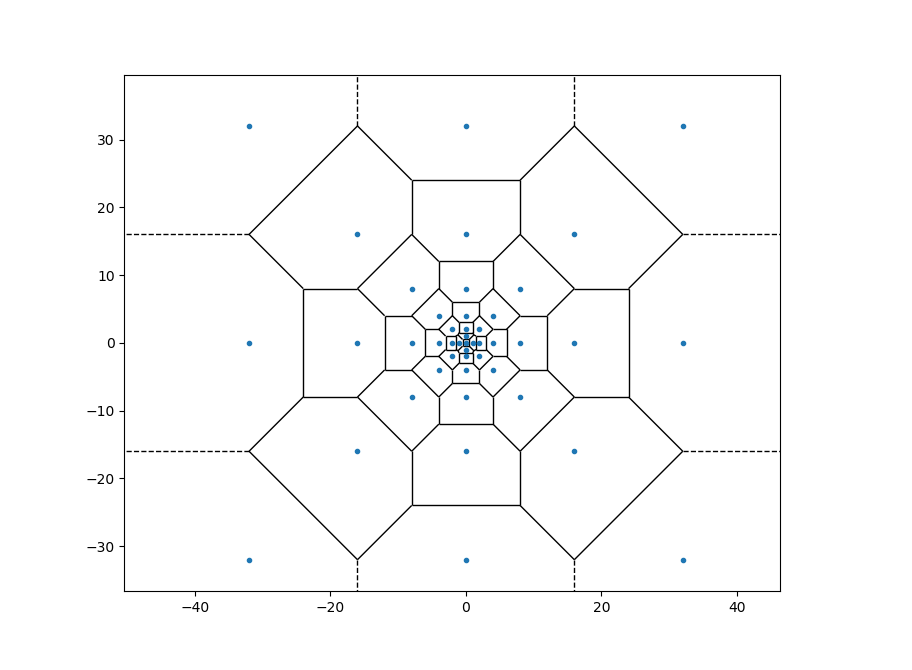}
\caption
{Example of points in ambient dimension 2 that satisfy our assumptions along with their associated Voronoi diagram.}
\label{figf2}
\end{figure}


In the sequel, we use the following notation : 
\begin{itemize}
    \item[$\bullet$] $B_R$ : the ball of radius $R>0$ centered at the origin ; $B_R(x)$ : the ball of radius $R>0$ and center $x\in \mathbb{R}^d$ ; $\displaystyle A_{R,R'}$ : the set $B_R \setminus{B_{R'}}$ for $R>R'>0$.
    \item[$\bullet$] $\displaystyle Q_R(x)$ : the set $\displaystyle \left\{ y \in \mathbb{R}^d \ \middle| \ \max_i |y_i - x_i| \leq R \right\}$  for $R>0$ and $x\in \mathbb{R}^d$ ; $\displaystyle Q_R$ : the set $Q_R(0)$.
    \item[$\bullet$] $\sharp B$ : the cardinality of a discrete set $B$. 
    \item[$\bullet$] $2^p$ : the point $x_p \in \mathcal{G}_{C_0}$ for $p\in \mathcal{P}_{C_0}$ ; $\tau_p$ : the translation $\tau_{2^p}$ where $\tau_x f = f(.+x)$~; $V_p$ : the Voronoi cell $V_{2^p}$.
    \item[$\bullet$] $|p|$ : the norm defined by  $\displaystyle \max_{i \in \{1,...,d\}} |p_i|$.
\end{itemize}
In addition, for a normed vector space $(X,\|.\|_X)$ and a matrix-valued function $f \in X^n$, $n\in \mathbb{N}$, we use the notation $\|f\|_{X} \equiv \|f \|_{X^n}$ when the context is clear. 

We associate to \eqref{defPC0}-\eqref{defG} the following functional space : 

\begin{equation}
\label{spacedef}
 \mathcal{B}^2(\mathbb{R}^d) = \left\{f \in L^2_{unif}(\mathbb{R}^d) \ \middle| \ \exists f_{\infty} \in L^2(\mathbb{R}^d),  \lim_{|p| \rightarrow \infty}  \int_{\textit{V}_p} |f(x) - \tau_{-p} f_{\infty}(x)|^2 dx = 0\right\},  
\end{equation}
equipped with the norm 
\begin{equation}
\label{normdef}
\|f \|_{\mathcal{B}^2(\mathbb{R}^d)} = \|f_{\infty}\| _{ L^2(\mathbb{R}^d)} + \|f\|_{L^2_{unif}(\mathbb{R}^d)}+ \sup_{p\in \mathcal{P}_{C_0}} \|f-\tau_{-p}f_{\infty}\| _{ L^2(\textit{V}_p)}.   
\end{equation}
In \eqref{spacedef}, \eqref{normdef} we have denoted by : 
\begin{equation}
L^2_{unif}(\mathbb{R}^d) = \left\{f \in L^2_{loc}(\mathbb{R}^d), \ \sup_{x \in \mathbb{R}^d}  \|f\|_{L^2(B_1(x))} < \infty\right\},
\end{equation}
and 
\begin{equation}
\|f \| _{ L^2_{unif}(\mathbb{R}^d)} = \sup_{x \in \mathbb{R}^d} \|f\|_{L^2(B_1(x))}.
\end{equation}

Intuitively, a function in $ \mathcal{B}^2(\mathbb{R}^d)$ behaves, locally at the "vicinity" of each point $x_p$, as a fixed $L^2$ function truncated over the domain $V_p$. We show several properties of the functional space $\mathcal{B}^2(\mathbb{R}^d)$ in Section~\ref{Section3}.

\subsection{Main results}

We henceforth assume that the ambient dimension $d$ is equal to or larger than 3. The one-dimensional and two-dimensional contexts are specific. Some results or proofs must be adapted in these particular cases but we will not proceed in that direction in all details. This is due to the asymptotic behavior of the Green function of the Laplacian operator in these two dimensions. In these two particular cases, we claim that it is still possible to show the existence of the corrector defined by Theorem \ref{theoreme1} below. However, the method used in Lemmas \ref{lemmeexistenceperiodique1} and \ref{existenceper}, both useful for the proof of Theorem \ref{theoreme1}, need to be adapted. The one-dimensional context can be addressed easily because the solution to \eqref{correcteur} is explicit. The two-dimensional case requires more work. We explain how to adapt our proof in Remark \ref{remarkcorrectordim2}. In contrast, in dimensions $d=1$ and $d=2$, the convergence rates of Theorem~\ref{theoreme3} no longer hold.  Indeed, the corrector $w_p$ is then not necessarily bounded (see Lemma \ref{lemmeborne} for details). We are only able to prove weaker results in these cases. Additional details about these cases may be found in Remarks \ref{remarkcorrectordim2}, \ref{remarkbornedimension2}, and \ref{dim1}.

For $\alpha \in ]0,1[$, we denote by $ \mathcal{C}^{0,\alpha}(\mathbb{R}^d)$ the space of \emph{uniformly} H\"older continuous and bounded functions with exponent $\alpha$, that is : 
$$ \mathcal{C}^{0,\alpha}(\mathbb{R}^d) = \left\{f \in L^{1}_{loc}(\mathbb{R}^d) \ \middle| \ \|f\|_{\mathcal{C}^{0,\alpha}(\mathbb{R}^d)} < \infty \right\},$$ where 
$$\|f\|_{\mathcal{C}^{0,\alpha}(\mathbb{R}^d)} = \|f\|_{L^{\infty}(\mathbb{R}^d)} + \sup_{x,y \in \mathbb{R}^d,\ x\neq y}\frac{|f(x) - f(y)|}{|x-y|^{\alpha}}.$$

We consider a matrix-valued coefficient of the form \eqref{coefficientform} with $a_{per} \in  L^2_{per}(\mathbb{R}^d)^{d \times d}$ and $\Tilde{a} \in \mathcal{B}^2(\mathbb{R}^d)^{d \times d}$. We denote by $\Tilde{a}_{\infty}$ the matrix-valued limit $L^2$-function associated with $\Tilde{a}$, where each coefficient $(\Tilde{a}_{\infty})_{i,j}$ is the limit $L^2$-function associated with $(\Tilde{a})_{i,j}\in \mathcal{B}^2(\mathbb{R}^d)$ and defined in~\eqref{spacedef}. We assume that $a_{per}$, $\Tilde{a}$ and $\Tilde{a}_{\infty}$ satisfy : 
\begin{equation}
\label{hypothèses1}
\exists \lambda> 0 \text{ such that for all $x$, $\xi \in \mathbb{R}^d$} \quad  \lambda |\xi|^2 \leq \langle a(x)\xi, \xi\rangle , \quad  \lambda |\xi|^2 \leq  \langle a_{per}(x)\xi, \xi\rangle,\\    
\end{equation}
and
\begin{equation}
\label{hypothèses2}
   a_{per}, \ \Tilde{a}, \ \Tilde{a}_{\infty}  \in \mathcal{C}^{0,\alpha}(\mathbb{R}^d)^{d \times d}, \qquad \text{ $\alpha \in ]0,1[$}.
\end{equation}

The coercivity \eqref{hypothèses1} and the $L^{\infty}$ bound on $a$ ensure that the sequence of solutions $\left(u^{\varepsilon}\right)_{\varepsilon>0}$ to \eqref{equationepsilon} converges in $weak-H^1(\Omega)$ and $strong-L^2(\Omega)$ up to an extraction when $\varepsilon \rightarrow 0$. Classical results of homogenization show the limit $u^*$ is a solution to a diffusion equation of the form \eqref{homog} for some matrix-valued coefficient $a^*$ to be determined. The questions that we examine in this paper are : What is the diffusion coefficient $a^*$ of the homogenized equation ? Is it possible to define an approximate sequence of solutions $u^{\varepsilon,1}$  as in \eqref{approximatesequence} ? For which topologies does this approximation correctly describe the behavior of $u^{\varepsilon}$ ? What is the convergence rate ? 

In answer to our first question, we prove in Proposition \ref{proplimithomogenization} that the homogenized coefficient $a^*$ is constant and is the same as in the periodic case. This result is a direct consequence of Proposition \ref{moyenne} which ensures that the perturbations of $\mathcal{B}^2(\mathbb{R}^d)$ have a zero average in a strong sens. Consequently, our perturbations are "small" at the macroscopic scale and do not affect the homogenization that occurs in the periodic case associated with the periodic coefficient $a_{per}$. In reply to the other questions, our main results are contained in the following two theorems :

\begin{theorem}
\label{theoreme1}
For every $p \in \mathbb{R}^d$, there exists a unique (up to an additive constant)
function $w_p \in H^1_{loc}(\mathbb{R}^d)$ such that $\nabla w_p \in \left(L^2_{\textit{per}}(\mathbb{R}^d) + \mathcal{B}^2(\mathbb{R}^d)\right)^d\cap \mathcal{C}^{0,\alpha}(\mathbb{R}^d)^d$, solution to : 
\begin{equation}
\label{correcteur}
\left\{
\begin{array}{cc}
  -\operatorname{div}((a_{\textit{per}} + \Tilde{a})(p+\nabla w_p)) = 0   &  \text{in $\mathbb{R}^d$},\\[0.2cm]
  \displaystyle \lim_{|x| \rightarrow \infty} \dfrac{|w_p(x)|}{1 + |x|} = 0. & 
\end{array}
\right.    
\end{equation}
\end{theorem}

\begin{theorem}
\label{theoreme3}
Assume $\Omega$ is a ${C}^{2,1}$-bounded domain. Let $\displaystyle\Omega_1 \subset \subset \Omega$. We define $ \displaystyle u^{\varepsilon,1} =  u^* + \varepsilon \sum_{i=1}^{d}\partial_i u^{*} w_{e_i}(./\varepsilon)$ where $w_{e_i}$ is defined by Theorem \ref{theoreme1} for $p=e_i$ and $u^*$ is the solution to \eqref{homog}. Then $R^{\varepsilon} = u^{\varepsilon} - u^{\varepsilon,1}$ satisfies the following estimates : 
\begin{equation}
\label{estimate1}
\|R^{\varepsilon}\|_{L^2(\Omega)} \leq C_1 \varepsilon \|f\|_{L^2(\Omega)},
\end{equation}
\begin{equation}
\label{estimate2}
\|\nabla R^{\varepsilon}\|_{L^2(\Omega_1)} \leq C_2 \varepsilon \|f\|_{L^2(\Omega)},
\end{equation}
where $C_1$ and $C_2$ are two positive constants independent of $f$ and $\varepsilon$. 
\end{theorem}


Our article is organized as follows. In Section \ref{Section2} we prove some geometric properties satisfied by our set of points $\mathcal{G}_{C_0}$, in particular we show that it satisfies Assumptions \eqref{H1}, \eqref{H2} and \eqref{H3}. In section \ref{Section3}  we study the properties of $\mathcal{B}^2(\mathbb{R}^d)$ and its elements. In Section~\ref{Section4} we prove Theorem \ref{theoreme1}. Finally, in Section~\ref{Section5} we obtain the expected homogenization convergences stated in Theorem  \ref{theoreme3}. We conclude this introduction section with some comments.

\subsection{Extensions and perspectives}

 A first possible extension of the above results, which we study in  Appendix \ref{AnnexeA}, consists in considering the functional spaces  $\mathcal{B}^r$  for $r\neq 2$, $1<r<\infty$, defined similarly to $\mathcal{B}^2$, but using the $L^r$ topology. As in the study of the $L^r(\mathbb{R}^d)$ defects in \cite{blanc2018correctors, blanc2015local}, we show some modifications for the convergence rates of Theorem~\ref{theoreme3} depending upon the value of $r$ and the ambient dimension $d$. Indeed, in this case, some results related to the strict sub-linearity of the corrector allow us to show that the convergence rates of $R^{\varepsilon}$ is $\varepsilon^{\frac{d}{r}}|\log(\varepsilon)|^{\frac{1}{r}}$ if $r\leq d$ and $\varepsilon$ else.

In addition, although we have not pursued in these directions, we believe it is possible to extend the above results in several other manners.

\begin{itemize}
    \item[1)] First, under additional assumptions satisfied by the function $f$, we expect the estimates of Theorem~\ref{theoreme3} to hold, with possibly different rates, in other norms than $L^2$ such as $L^q$, for $1<q< \infty$ or $\mathcal{C}^{0,\alpha}$, for $\alpha \in ]0,1[$. It seems that such questions could be addressed by adapting the proofs of Section \ref{Section5} and consider the methods employed in \cite{blanc2018precised} using the behavior of the Green function associated with problem \eqref{equationepsilon}.

    \item[2)] We also believe that it is possible to show results analogous to that of Theorems \ref{theoreme1} and \ref{theoreme3} in the case of equations not in divergence form, instead of \eqref{equationepsilon}, 
    \begin{equation}
    \label{nondivergenceform}
    -a_{ij}\partial_{ij} u = f,     
    \end{equation}
    where $a$ is a periodic coefficients perturbed by a defect in $\mathcal{B}^2(\mathbb{R}^d)$ of the form \eqref{coefficientform}. One way to address this question could be to adapt the methods of \cite[Section 3]{blanc2018correctors} the the case of local perturbations, that is, to show the existence of an invariant measure $m = m_{per} + \Tilde{m}$ in $L
   ^2_{per}+ \mathcal{B}
   ^2(\mathbb{R}^d)$ solution to : 
   \begin{equation}
   \label{equation_invariant_measure}
     -\partial_{i,j}\left(a_{i,j} m_{i,j} \right) = 0 \quad \text{in } \mathbb{R}^d,  
   \end{equation}
   such that $\inf m >0$. Indeed, using the method presented in \cite{avellaneda1989compactness}, this study could be then reduced to a problem of divergence form operator as soon as such a measure $m$ exists and the results established in this article could allow to conclude.

    \item[3)] In the same way, another possible generalization concerns advection-diffusion equation in the form : 
    \begin{equation}
        -a_{ij}\partial_{ij}u + b_j \partial_j u = f \quad \text{in } \mathbb{R}^d,
    \end{equation}
    where $a$ and $b$ are two periodic coefficients perturbed by a defect in $\mathcal{B}^2(\mathbb{R}^d)$.
    The method \cite{blanc2019correctors} is likely to be adapted to this case, showing the existence of an invariant measure $m$ in $L
   ^2_{per}+ \mathcal{B}
   ^2(\mathbb{R}^d)$ solution to
      \begin{equation}
   \label{equation_invariant_measure2}
     -\partial_{i}\left(\partial_j\left(a_{i,j} m_{i,j} \right) + b_i m_{i,j}\right) = 0 \quad \text{in } \mathbb{R}^d.  
   \end{equation}
   
\end{itemize}

\section{Geometric properties of the Voronoi cells}
\label{Section2}

We start by studying the geometric properties of the Voronoi cells associated to every sets of points $\mathcal{G}$ satisfying the general Assumptions \eqref{H1}, \eqref{H2} and \eqref{H3}. In particular, we show these assumptions ensure the rarity of the points $x_{p}$ in the space proving, in Proposition \ref{numberpointAn} and Corollary \ref{corollogboundVP}, that the number of points of $\mathcal{G}$ contained in a ball of radius $R>0$ is bounded by the logarithm of $R$. In Propositions \ref{tailledesVP} and \ref{openconstruction}, we also show two technical properties regarding the size and the structure of the cells. All these properties are actually fundamental for the rest of our work since they allow us to prove several results regarding the existence and uniqueness of solutions to the class \eqref{equationref} of diffusion equations $-\operatorname{div}(a\nabla u) = \operatorname{div}(f)$ studied in Section \ref{Section4}. In particular, as we shall see in the proof of Lemma \ref{lemmeexistenceperiodique1}, we use these geometric properties to bound several integrals in order to define a solution to equation \eqref{eqper1},  that is \eqref{equationref} with $a=a_{per}$, using the associated Green function. 
To conclude this section, we also show that our specific set of points $\mathcal{G}_{C_0}$, defined by \eqref{defG}, satisfies \eqref{H1}, \eqref{H2} and \eqref{H3}.

\subsection{General properties}

In this subsection only, we proceed with the whole generality of Assumptions \eqref{H1}, \eqref{H2} and \eqref{H3} and we introduce several useful geometric properties satisfied by every sets of points $\mathcal{G}$ satisfying these assumptions. These properties relate to the size of the Voronoi cells, their volume and their distribution in the space $\mathbb{R}^d$.  

To start with, we show two properties regarding the volume of the Voronoi cells.

\begin{prop}
\label{volumeVx}
There exist $C_1>0$ and $C_2>0$ such that for every $x \in \mathcal{G}$, we have the following bounds : 
\begin{equation}
    C_1 |x|^d \leq \left|V_x\right| \leq C_2 |x|^d.
\end{equation}
\end{prop}

\begin{proof}
For every $x\in \mathcal{G}$, using the definition of the Voronoï diagram, we have the following inclusion : 
$$ B_{D(x, \mathcal{G}\setminus{\{x\}})/2}(x) \subset{V_x}.$$ 
Therefore, there exists a constant $C(d)>0$ such that :   
$$C(d) D(x, \mathcal{G}\setminus{\{x\}})^d = \left|B_{D(x, \mathcal{G}\setminus{\{x\}})/2}(x)\right| \leq \left|V_x \right| \leq Diam(V_x)^d.$$
We conclude using \eqref{H2} and \eqref{H3}. 
\end{proof}

\begin{prop}
\label{tailledesVP}
There exists a sequence $(x_{n})_{n\in \mathbb{N}} \in \mathcal{G}^{\mathbb{N}}$ such that $\left(V_{x_n} - x_n\right)$ is an increasing sequence of sets and : 
\begin{equation}
    \bigcup_{n \in \mathbb{N}} \left(V_{x_{n}} - x_n \right) = \mathbb{R}^d.
\end{equation}
\end{prop}

\begin{proof}
We consider a sequence $(x_{n})_{n\in \mathbb{N}} \in \mathcal{G}^{\mathbb{N}}$ such that the sequence $|x_n|$ is increasing and $\displaystyle \lim_{n \rightarrow \infty}|x_n| = \infty$ (such a choice is always possible according to Assumptions \eqref{H1} and \eqref{H2}). Since we have assumed that $\mathcal{G}$ satisfies \eqref{H2}, there exists $C>0$ such that for all $n \in \mathbb{N}$ : 
$$D(x_n, \mathcal{G}\setminus{\{x_n\}}) \geq C |x_n|.$$
Therefore, as a consequence of the definition of the Voronoi cells, the ball $B_{C|x_n|/2}(x_n)$ is included in $V_{x_n}$ and, by translation, the ball  $B_{C|x_n|/2}$ is included in $V_{x_n}-x_n$. Since $(x_n)_{n\in \mathbb{N}}$ is an increasing sequence such that $\displaystyle \lim_{n \rightarrow \infty}|x_n| = \infty$, we use \eqref{H1} and we obtain, up to an extraction, that $V_{x_n}$ is included in $B_{C|x_{n+1}|/2}(x_{n})$. Thus 
$$\forall n \in \mathbb{N}, \quad V_{x_n} - x_n \subset B_{C|x_{n+1}|/2} \subset V_{x_{n+1}} - x_{n+1}.$$
The sequence $(V_{x_n} - x_n)$ is therefore an increasing sequence of sets and, in addition, 
$$ \mathbb{R}^d = \bigcup_{n \in \mathbb{N}} B_{C|x_n|/2} \subset  \bigcup_{n \in \mathbb{N}} \left(V_{x_{n}} - x_n \right).$$
We directly deduce that $\displaystyle \mathbb{R}^d = \bigcup_{n \in \mathbb{N}} \left(V_{x_{n}} - x_n \right)$. 
\end{proof}

The next results ensure a certain distribution of the Voronoi cells in the space. In particular, we prove that the number of cells contained in a ball of radius $R>0$ increases at most as the logarithm of this radius. This property reflects the rarity of our points far from the origin and is essential in our approach. 

\begin{prop}
\label{numberpointAn}
There exists a constant $C(d)>0$ that depends only of the ambient dimension $d$ such that : 
\begin{equation}
    \sharp \left\{x \in \mathcal{G} \middle| x \in A_{2^n,2^{n+1}} \right\} \leq C(d).
\end{equation}
\end{prop}

\begin{proof}
Let $x\in \mathcal{G}$ such that $x \in A_{2^n,2^{n+1}}$. The definition of the Voronoi cells ensures that the distance $D(x,\partial V_x)$ is equal to $\frac{D\left(x, \mathcal{G} \setminus \left\{x \right\}\right)}{2}$. Property \eqref{H2} gives the existence of a constant $C_1>0$ independent of $x$ such that : 
$$\frac{D\left(x, \mathcal{G} \setminus \left\{x \right\}\right)}{2} \geq C_1\frac{|x|}{2} \geq C_1 2^{n-1}.$$ 
Then, the ball $B_{C_12^{n-1}}(x)$ is contained in $V_x$, that is $x$ is the only element of $\mathcal{G}$ in this ball. In addition, since $|x| \leq 2^{n+1}$, we obtain the following inclusion using a triangle inequality~: 
\begin{equation*}
\label{ballinclusion}
B_{C_12^{n-1}}(x) \subset B_{(C_1 + 4)2^{n-1}}.  
\end{equation*}
Since this inclusion is valid for every $x \in \mathcal{G}\cap A_{2^n,2^{n+1}}$ we obtain : 
$$\bigcup_{x \in \mathcal{G}\cap A_{2^n,2^{n+1}}} B_{C_1 2^{n-1}}(x) \subset B_{(C_1 + 4)2^{n-1}}. $$
Therefore, there exists $C_2(d)>0$ such that : 
\begin{equation}
\label{volumeunion1}
\left|\bigcup_{x \in \mathcal{G}\cap A_{2^n,2^{n+1}}} B_{C_1 2^{n-1}}(x)\right| \leq \left|B_{(C_1 + 4)2^{n-1}}\right| \leq C_2(d) 2^{d(n-1)}.
\end{equation}
Next, we know that the Voronoi cells are disjoint and, therefore, the collection of balls $\left(B_{C_1 2^{n-1}}(x)\right)_{x \in \mathcal{G}\cap A_{2^n,2^{n+1}}}$ is also disjoint. Thus, there exists $C_3(d)>0$ such that :
\begin{align}
\left|\bigcup_{x \in \mathcal{G}\cap A_{2^n,2^{n+1}}} B_{C_1 2^{n-1}}(x)\right| & = \sharp \left\{x \in \mathcal{G} \middle| x \in A_{2^n,2^{n+1}} \right\} \left|B_{C_1 2^{n-1}}\right|\\
\label{volumeunion2}
&= \sharp \left\{x \in \mathcal{G} \middle| x \in A_{2^n,2^{n+1}} \right\} C_3(d) 2^{d(n-1)}.  
\end{align}
With \eqref{volumeunion1} and \eqref{volumeunion2}, we conclude that : 
$$\sharp \left\{x \in \mathcal{G} \middle| x \in A_{2^n,2^{n+1}} \right\} \leq \frac{C_2(d)}{C_3(d)}.$$
\end{proof}

\begin{corol}
\label{corollogboundVP}
There exists $C>0$ such that for every $R>0$ and $x_0 \in \mathbb{R}^d$ : 
\begin{equation}
\label{logboundVP}
    \sharp \left\{x \in \mathcal{G} \middle| V_x \cap B_R(x_0) \neq \emptyset \right\} < C \log(R).
\end{equation}
\end{corol}

\begin{proof}
We start by proving the result if $R = 2^n$ for $n \in \mathbb{N}^*$.
Without loss of generality, we can assume that $n$ is sufficiently large to ensure there exists $x$ in $\mathcal{G}\cap B_{2^n}(x_0)$.
Using a triangle inequality, we remark that if $y\in B_{2^n}(x_0)$ we have : 
$$\left|x - y\right| \leq |x-x_0| + |y-x_0| \leq 2^{n+1} \leq D\left(y, \mathbb{R}^d\setminus{B_{2^{n+3}}(x_0)}\right).$$
That is, if $\Tilde{x} \in \mathcal{G}$ is such that $\Tilde{x} \notin B_{2^{n+3}}(x_0)$, every point $y \in B_{2^n}(x_0)$ is closer to $x$ than to $\Tilde{x}$, that is $V_{\Tilde{x}} \cap B_{2^n}(x_0) = \emptyset$. Therefore, we have $$\sharp \left\{x \in \mathcal{G} \middle| V_x \cap B_{2^n}(x_0)\neq \emptyset \right\}  \leq \sharp \left\{x \in \mathcal{G} \middle|  x \in B_{2^{n+3}}(x_0) \right\}.$$
Next, if $|x_0|\leq 2^{n+4}$, we have $B_{2^{n+3}(x_0)} \subset B_{2^{2n+7}}$
and we use Proposition \ref{numberpointAn} to obtain the existence of a constant $C>0$ independent of $n$ such that :
\begin{align*}
\sharp \left\{x \in \mathcal{G} \middle|  x \in B_{2^{n+3}}(x_0) \right\} & \leq \sharp \left\{x \in \mathcal{G} \middle|  x \in B_{2^{2n+7}} \right\} \\
& = \sum_{k = 0}^{2n+6} \sharp \left\{x \in \mathcal{G} \middle| x \in A_{2^k,2^{k+1}} \right\} + \sharp \left\{x \in \mathcal{G} \middle| x \in B_1\right\} \\
& \leq C n.
\end{align*}
If $|x_0|>2^{n+4}$, we denote by $m\geq n+4$ the unique integer such that $2^m~<~|x_0|$ and $|x_0|\leq~2^{m+1}$. In this case, we use a triangle inequality and we have
$$B_{2^{n+3}(x_0)} \subset A_{|x_0|+2^{n+3},|x_0|-2^{n+3}}\subset A_{2^{m+2},2^{m-1}}.$$
Proposition \ref{numberpointAn} gives the existence of $C>0$ independent of $x_0$ and $n$ such that :
\begin{align*}
\sharp \left\{x \in \mathcal{G} \middle|  x \in B_{2^{n+3}}(x_0) \right\} & \leq \sharp \left\{x \in \mathcal{G} \middle|  x \in A_{2^{m+2},2^{m-1}} \right\} \\
& = \sum_{k = -1}^{1} \sharp \left\{x \in \mathcal{G} \middle| x \in A_{2^{m+k},2^{m+k+1}} \right\} \leq C.
\end{align*}
Finally, we have estimate \eqref{logboundVP} in the particular case $R=2^n$.

Next, for any $R>0$, we have :
$$R = 2^{\log_2(R)} \leq 2^{[\log_2(R)]+1},$$
where $[.]$ denotes the integer part. Thus, we obtain the following upper bound : 
$$\sharp \left\{x \in \mathcal{G} \middle| V_x \cap B_R(x_0) \neq \emptyset \right\} \leq \sharp \left\{x \in \mathcal{G} \middle| V_x \cap B_{2^{[\log_2(R)]+1}}(x_0) \neq \emptyset \right\} \leq C\left([\log_2(R)]+1\right),$$
and we can conclude.
\end{proof}

To conclude this section, we now introduce a particular set (denoted by $W_x$ in the proposition below) containing a point $x\in \mathcal{G}$ which is both bigger than the cell $V_x$ and far from all the others points of $\mathcal{G}$. As we shall see in Lemmas \ref{lemmeexistenceperiodique1}  and \ref{existenceper}, this set is actually a technical tool that allows us to show the existence of the corrector stated in Theorem~\ref{theoreme1}. 

\begin{prop}
\label{openconstruction}
For every $x \in \mathcal{G}$, there exists a convex open set $W_x$ of $\mathbb{R}^d$ and $C_1$, $C_2$, $C_3$, $C_4$ and $C_5$ five positive constants independent of $x$ such that :
\begin{enumerate}[label=(\roman*)]
\item \label{Wi} $\displaystyle \textit{V}_x \subset{W}_x,$
\item \label{Wii}$Diam(W_x) \leq C_1 |x|$ and $D(\textit{V}_x,\partial W_x) \geq C_2 |x|,$
\item \label{Wiii} $\forall y \in \mathcal{G}\setminus{\{x\}}$, $D(y, W_x) \geq C_3 |x|,$ 
\item \label{Wiv} $\sharp \left \{ y \in \mathcal{G} \middle| V_y \cap W_x \neq \emptyset \right\} \leq C_4,$
\item \label{Wv}  $\forall y \in \mathcal{G}\setminus{\{x\}}$, $D(V_y \setminus{W_x}, V_x) \geq C_5|y|.$ 
\end{enumerate}
\end{prop}

\begin{proof}
Let $x$ be in $\mathcal{G}$. In the sequel, we denote $I_{x,y} = \left\{z \in \mathbb{R}^d \ \middle| \ |z-x| \leq |z - y| \right\}$ and  $\varphi_x$ the homothety of center $x$ and ratio $\frac{3}{2}$.  For $y \in \mathcal{G}\setminus{\{x\}}$, we denote by $H_{x,y}$ the set defined by : 
$$ H_{x,y} = \varphi_x\left( I_{x,y}\right).$$
The set $H_{x,y}$ can be easily determined, it is the half-space defined by :
$$H_{x,y} = \left\{z \in \mathbb{R}^d \ \middle| \ |z-x| \leq |z - y| \right\} + \frac{1}{4}\overrightarrow{xy} = I_{x,y} + \frac{1}{4}\overrightarrow{xy}. $$ 
We finally consider : 
$$W_x = \bigcap_{y \in \mathcal{G}\setminus{\{x\}}} H_{x,y},$$ 
which is actually the image of the cell $V_x$ by the homothety $\varphi_x$ (see figure \ref{ouvertvoronoi}). 

\begin{figure}[h!]
\centering
\includegraphics[width=0.55\linewidth]{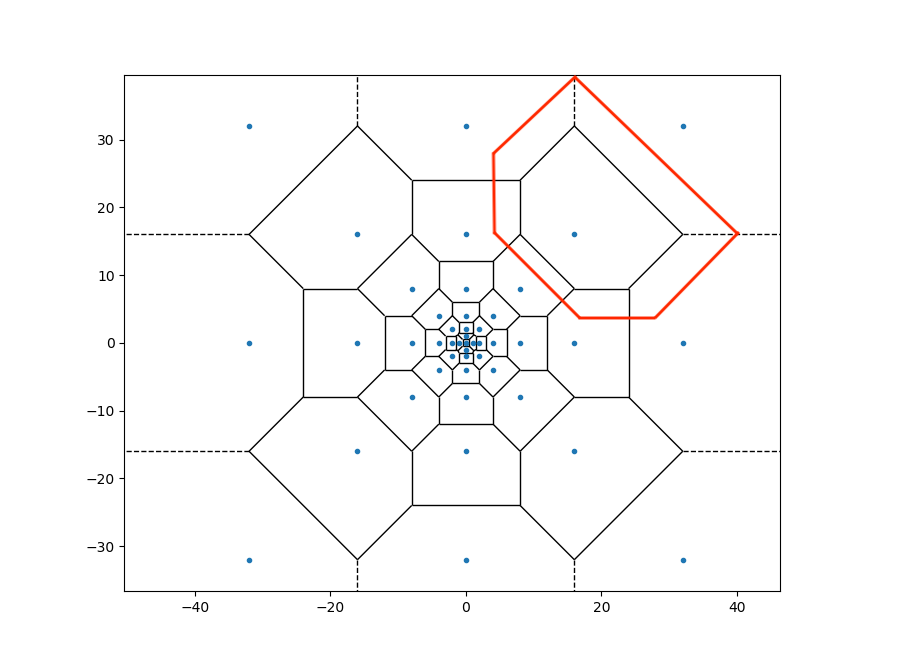}
\caption{Example for the choice of the open subset $W_x$ (in red) when $d=2$.}
\label{ouvertvoronoi}
\end{figure}

We next prove that $W_x$ satisfies \ref{Wi}, \ref{Wii}, \ref{Wiii}, \ref{Wiv} and \ref{Wv}. 
\begin{itemize}
    \item[\ref{Wi} :] For every $y \in \mathcal{G}\setminus{\{x\}}$ we have $I_{x,y} \subset H_{x,y}$  and therefore, we obtain using definition \eqref{VoronoiDEF} of $V_x$  : 
    $$V_x = \bigcap_{y \in \mathcal{G}\setminus{\{x\}}} I_{x,y} \subset \bigcap_{y \in \mathcal{G}\setminus{\{x\}}} H_{x,y} = W_x,$$
    and we have the first inclusion. 
    \item[\ref{Wii} :] $W_x$ is a $\frac{3}{2}$-dilation of $V_x$, thus we have $Diam(W_x) = \frac{3}{2}Diam(V_x)$. We use \eqref{H2} and \eqref{H3} to obtain the first estimate. Next, the definitions of the sets $H_{x,y}$ and $W_x$ give : 
    $$D(V_x, \partial W_x) = \frac{1}{4}\inf_{y \in \mathcal{G}\setminus{\{x\}}} |x-y| = \dfrac{1}{4} D\left(x,\mathcal{G}\setminus{\{x\}}\right).$$   We conclude using \eqref{H2}.
    \item[\ref{Wiii} :] Let $y$ be in $\mathcal{G}\setminus{\{x\}}$. By definition, for every $v\in W_x$, there exists $u \in I_{x,y}$ such that $v = u + \frac{1}{4}\overrightarrow{xy}$. Therefore, we use the triangle inequality and we have : 
    $$|v-y| \geq D\left( y, I_{x,y}\right) - \frac{1}{4}|x-y| = \frac{1}{2}|x-y| - \frac{1}{4}|x-y| = \frac{1}{4}|x-y|.$$
    Taking the infimum over all $v \in W_x$ in the above inequality and using \eqref{H2}, we finally obtain :
    $$D(y, W_x) \geq \frac{1}{4}|x-y| \geq \frac{1}{4} D \left(x, \mathcal{G}\setminus{\{x\}}\right) \geq C \frac{1}{4} |x|,$$
    where $C>0$ is independent of $x$ and $y$. 
     \item[\ref{Wiv} :] First, we have proved there exists a constant $C_1\geq1$ independent of $x$ such that $Diam\left(W_x\right)\leq C_1 |x|$. Second, using Assumption \eqref{H2}, we know there exists a constant $C_2>0$ such that for every $y \in \mathcal{G}$ we have $D\left(y, \mathcal{G}\setminus{\{y\}}\right)\geq C_2 |y|$. Let $k>2$ be an integer such that : 
     \begin{equation}
         \label{defk}
         C_22^{k-2} -1 > 4 C_1.
     \end{equation}
     We denote $n \in \mathbb{N}$, the unique integer such that $x \in A_{2^n,2^{n+1}}$. Here, it is sufficient to establish a bound for $x$ sufficiently large, thus without loss of generality, we can assume that $n>k$. We next show that if $y \in \mathcal{G}$ satisfies $|y|\leq 2^{-k-1}|x| \leq 2^{n-k}$ or $|y|\geq 2^k |x|  \geq 2^{n+k}$, then $W_x \cap V_y = \emptyset$. 
     
    We start by assuming that $y \in \mathcal{G}\cap \left(\mathbb{R}^d \setminus{B_{2^{n+k}}}\right)$. Since 
    $$Diam\left(W_x\right)\leq C_1 |x| \leq~C_1 2^{n+1},$$ we have $W_x \subset B_{C_12^{n+1}}(x)$. Therefore, using a triangle inequality we obtain $W_x \subset B_{C_12^{n+2}}$. Our aim here is to prove that $I_{y,x} \cap B_{C_12^{n+2}} = \emptyset $ in order to deduce $I_{y,x} \cap W_x = \emptyset $.  For every $z \in I_{y,x}$  : 
    \begin{align*}
     |z| & \geq |z-x| - |x| \geq D \left(x,I_{x,y} \right) - |x|.
     \end{align*}
     In addition, for every $y \in \mathcal{G} \setminus{\{y\}}$, we have $ \displaystyle D \left(x,I_{x,y} \right) = \frac{1}{2}|x-y|\geq \frac{1}{2} D\left(y,\mathcal{G}\setminus{\{y\}}\right)$ and we deduce that :  
     \begin{align*}
      |z| & \geq D\left(y,\mathcal{G}\setminus{\{y\}}\right) - |x| \\
      &\geq \frac{C_2}{2} |y| - |x| \\
      & \geq C_2 2^{n+k-1} - 2^{n+1} \\
     & \geq 2^{n+1}\left(C_2 2^{k-2}-1 \right) \geq C_12^{n+3}.
    \end{align*}
     Therefore, $I_{x,y} \subset \left(\mathbb{R}^d \setminus{ B_{C_1 2^{n+3}}}\right)$ and we obtain $W_x \cap I_{x,y} = \emptyset
     $. Since, $ \displaystyle V_y = \bigcap_{z \in \mathcal{G}\setminus{\{y\}}} I_{z,y}$, we deduce that $V_y \cap W_x = \emptyset$.
     
     Next we assume that $y \in B_{2^{n-k}}$ and we want to prove that $V_y \cap H_{x,y} = \emptyset$. As above, we can show that $V_y \subset B_{C_1 2^{n-k+1}}$ and for every $z \in H_{x,y}$ : 
     \begin{align*}
         |z| & \geq \frac{1}{4}|x-y| - |y| \\
         & \geq \frac{1}{4}C_2 |x| - |y|\\
         & \geq 2^{n-k} \left(C_2 2^{k-2} - 1\right) \geq C_1 2^{n-k +2}.
     \end{align*}
Therefore $H_{x,y} \subset{B_{C_1 2^{n-k+2}}}$ and we have $V_y \cap H_{x,y} = \emptyset$. We deduce that $V_y \cap W_{x} = \emptyset$. 

To conclude, we use Proposition \ref{numberpointAn} and we obtain the existence of a constant $C_3>0$ independent of $n$ such that : 
$$\sharp \left\{x \in \mathcal{G} \middle| x \in A_{2^n,2^{n+1}} \right\} \leq C_3,$$
and therefore : 
\begin{align*}
\sharp \left \{ y \in \mathcal{G} \middle| V_y \cap W_x \neq \emptyset \right\} & \leq \sum_{m = -k}^{k} \left\{x \in \mathcal{G} \middle| x \in A_{2^m,2^{m+1}} \right\} \\ 
& \leq \sum_{m = n-k}^{n+k} C_3 = \left(2k +1\right)C_3.
\end{align*}
We have finally proved \ref{Wiv}. 
\item[\ref{Wv} :] Let $y$ be in $\mathcal{G}\setminus{\{x\}}$. We first assume that $2^{-k-1}|y|> |x|$, where $k$ is defined as in \eqref{defk} and is independent of $x$. In the proof of \ref{Wiv} above, we have shown that $W_y\cap V_x = \emptyset$. Therefore, using Property \ref{Wi} and \ref{Wii} of $W_y$ we easily obtain that there exists a constant $M_1>0$ independent of $x$ and $y$ such that $D(V_x,V_y)>M_1|y|$ and we can conclude. Next, we assume that $2^{-k-1}|y| \leq |x|$. Using again Property \ref{Wi} and \ref{Wii} of $W_x$, we obtain the existence of $M_2>0$ independent of $x$ and $y$ such that $D\left(V_x, V_y\setminus{W_x}\right)\geq M_2 |x| \geq M_2 2
^{-k-1} |y|$. Finally, we have proved \ref{Wv} with $C_5 =\min(M_1, M_22
^{-k-1})$.
\end{itemize}
\end{proof}

\subsection{The particular case of the "$2^p$"}

We next prove that the set $\mathcal{G}_{C_0}$ defined by \eqref{defG} satisfies Assumptions \eqref{H1}, \eqref{H2} and \eqref{H3}. In order to avoid many unnecessary technical details, we study here the Voronoï diagram only for $d=2$ (the case $d=1$ being obvious) and, in the sequel, we admit that these properties still hold in higher dimension. We also consider the cell $V_p$ only for $p=(p_1,p_2) \in \left(\mathbb{R}^{+*}\right)^2$. Since the distribution of the points $2^q$ is symmetric with respect to the origin, the other cases are similar and we omit them.

\begin{proof}[Proof of \eqref{H1}]

 Let $p=(p_1,p_2)$ be in $\mathcal{P}_{C_0}\cap \left(\mathbb{R}^{+*}\right)^2 $.  We first prove the following inclusion : 
  \begin{equation}
 \label{inclusionVP}
    V_{p} \subset{\left\{(x,y) \in \mathbb{R}^2    \middle| 2^{p_1-1}\leq x \leq 2^{|p|+1} \ \text{and } 2^{p_2-1}\leq y \leq 2^{|p|+1}\right\}}.
    \end{equation}
To this aim, we want to show that if $(x,y)\notin [2^{p_1-1},2^{|p|+1}]\times [2^{p_2-1}, 2^{|p|+1}]$, then there exists $x_q\in \mathcal{P}_{C_0}\setminus{\{x_p\}}$ such that the point $(x,y)$ is closer to $x_q$ than to $x_p$ and therefore $(x,y) \notin V_p$. We consider $(x,y)\in \left(\mathbb{R}^+\right)^2$ and we start by assuming that $x<2^{p_1-1}$. We have : 
    $$D((x,y) , x_p)^2 = |x - 2^{p_1}|^2 + |y - 2^{p_2}|^2, $$
    and : 
    $$D((x,y), (0,2^{p_2}))^2 = |x|^2 + |y - 2^{p_2}|^2. $$
    In addition, since $x<2^{p_1-1}$, we use a triangle inequality and : 
    $$|x - 2^{p_1}| > 2^{p_1} - 2^{p_1-1} = 2^{p_1-1} > |x|,$$
    we obtain that $D((x,y) , x_p)^2 > D((x,y), (0,2^{p_2}))^2$. That is, $(x,y)$ is closer to $(0,2^{p_2})\in \mathcal{G}_{C_0}$ than to $x_p$ and we deduce that $(x,y) \notin V_{p}$. Thus, we can first conclude that $V_p$ is included in $\left\{ (x,y) \in \mathbb{R}^2 \ \middle| \ 2^{p_1 - 1} \leq x \right\}$.
    
    Next, if $x>2^{|p|+1}$, we consider two cases : \\
    1. If $y\leq 2^{p_2}$, we have : 
        $$ D\left((x,y), \left(2^{|p|+1},0\right)\right)^2 =  \left|x - 2^{|p|+1}\right|^2 + |y|^2 \leq \left|x - 2^{|p|+1}\right|^2 + 2^{2p_2}.$$
        Since $|p|=\max(p_1,p_2)$, we have $|p|\geq p_1$ and : 
        \begin{align*}
           D((x,y) , x_p)^2 & \geq  |x-2^{p_1}|^2  = \left(x-2^{|p|+1} + 2^{|p|+1} - 2^{p_1}\right)^2\\
            & \geq  \left(x-2^{|p|+1} + 2^{|p|+1} - 2^{|p|} \right)^2 = \left(x-2^{|p|+1} + 2^{|p|} \right)^2.
        \end{align*}
        The inequalities above are valid since $x - 2^{|p|+1}>0$. Again, since $|p|\geq p_2$, we have :
        \begin{align*}
           \left(x-2^{|p|+1} + 2^{|p|} \right)^2 & \geq \left(x-2^{|p|+1} + 2^{p_2}\right)^2 \\ 
            & > \left|x-2^{|p|+1}\right|^2 + 2^{2p_2}.
        \end{align*}
        Finally, 
        $$D((x,y) , x_p)^2  > \left|x-2^{|p|+1}\right|^2 + 2^{2p_2}\geq D\left((x,y), \left(2^{|p|+1},0\right)\right).$$
        We conclude that $ D((x,y) , x_p)> D((x,y), (2^{p_1+1},0)) $ and finally $(x,y) \notin V_{p}$.
        \\
        
       2. If $y > 2^{p_2}$ : 
        $$ D((x,y), (2^{|p|+1},2^{p_2 +1}))^2 =  |x - 2^{|p|+1}|^2 + |y- 2^{p_2 +1}|^2.$$
        We have proved above that $|x-2^{p_1}|^2>\left|x-2^{|p|+1}\right|^2 + 2^{2p_2}$. In order to obtain a lower bound on $D((x,y) , x_p)$, we have to establish a similar inequality for the term $|y-2^{p_2}|^2$. We have : 
        \begin{align*}
            |y-2^{p_2}|^2 & = \left(y-2^{p_2+1} + 2^{p_2+1} - 2^{p_2}\right)^2\\
            & = \left(y-2^{p_2+1} + 2^{p_2}\right)^2\\
            & = \left|y-2^{p_2+1}\right|^2 + 2^{2p_2} + 2\left(y-2^{p_2+1}\right) 2^{p_2}\\
            & \geq \left|y-2^{p_2+1}\right|^2 + 2^{2p_2} - 2^{2p_2+1} = \left|y-2^{p_2+1}\right|^2 - 2^{2p_2}.
        \end{align*}
        Finally, we have : 
        $$|x-2^{p_1}|^2+|y-2^{p_2}|^2>\left|x-2^{|p|+1}\right|^2 + 2^{2p_2} + \left|y-2^{p_2+1}\right|^2 - 2^{2p_2} = \left|x-2^{|p|+1}\right|^2  + \left|y-2^{p_2+1}\right|^2.$$
        As above, we obtain that $D((x,y) , x_p)> D((x,y), (2^{|p|+1},2^{p_2+1}))$.  
        We just have to check that $\left(2^{|p| +1}, 2^{p_2+1}\right) \in \mathcal{G}_{C_0}$. First if $|p|= p_1$, since $\left(p_1, p_2 \right)\in \mathcal{P}_{C_0}$, we have :
        $$\max\left\{p_1, p_2\right\}\leq \min\left\{p_1, p_2\right\}+ C_0,$$
        and therefore : 
        $$\max\left\{p_1+1, p_2+1\right\}\leq \min\left\{p_1+1, p_2+1\right\}+ C_0.$$
        That is, $\left(|p|+1, p_2+1 \right) = \left(p_1+1, p_2+1 \right)\in \mathcal{P}_{C_0}$. Secondly, if $|p|=p_2$, we clearly have $\left(|p|+1, p_2+1 \right)= (p_2+1, p_2+1)\in \mathcal{P}_{C_0}$. In every cases, we can conclude that $(x,y) \notin V_{p}$ and finally we obtain that $V_p$ is included in $\left\{ (x,y) \in \mathbb{R}^2 \ \middle| \ x  \leq 2^{|p|} \right\}$..

    Using the symmetry of the distribution, we can use exactly the same argumentation to treat the cases $y<2^{p_2-1}$ and $y>2^{|p|+1}$. We finally have established the inclusion \eqref{inclusionVP}. We note the volume of the cube $\left[2^{p_1-1},2^{|p|+1}\right]\times \left[2^{p_2-1},2^{|p|+1}\right]$ is bounded by $4.2^{2|p|}$, and we can deduce that : 
    $$\left|V_{x_p}\right| \leq 4 .2^{2|p|}.$$ 
  \eqref{H1} is proved. 

\end{proof}

\begin{proof}[Proof of \eqref{H2}]
Let $p$ be in $\mathcal{P}_{C_0}\cap \left(\mathbb{R}^{+*}\right)^2$. We have : 
    $$D\left(x_p, \mathcal{G}_{C_0}\setminus{\{x_p\}}\right) \leq D(x_p, 0) = |x_p|,$$
    and therefore : 
    $$ 1 \leq \frac{1 + \left|x_p \right|}{D\left(x_p, \mathcal{G}_{C_0} \setminus \left\{x_p \right\}\right)}.$$
    To show the upper bound, we consider $x_q \in \mathcal{G}_{C_0}\setminus{\{x_p\}}$. Without loss of generality, we can assume $|p_1|= |p|$ and there are three cases :  
    \begin{itemize}
        \item If $|q_1|\neq |p_1|$, then : 
        \begin{align*}
    D(x_p, x_q) & \geq \left|\operatorname{sign}(p_1)2^{|p_1|} - \operatorname{sign}(q_1)2^{|q_1|}\right|\\
    & \geq \left|2^{|p_1|} - 2^{|q_1|}\right| = 2^{|p_1|}\left|1 - 2^{|q_1| - |p_1|}\right|\\
    & \geq  2^{|p_1|}\frac{1}{2} =  2^{|p|-1}.
    \end{align*}
    \item If $p_1 = q_1$, since $p\in \mathcal{P}_{C_0}$, we have $|p_2|\geq |p| - C_0$ and as above : 
    $$D(x_p, x_q) \geq 2^{|p_2|-1} \geq 2^{|p|-C_0-1}.$$
    \item Finally, if $p_1 = - q_1$, we have : 
    $$D(x_p, x_q) \geq \left|\operatorname{sign}(p_1)2^{|p_1|} - \operatorname{sign}(q_1)2^{|q_1|}\right|  = 2^{|p|+1}.$$
    \end{itemize}
    In the three cases we conclude there exists $C>0$ independent of $q$ such that $D(x_p - x_q) \geq C 2^{|p|}$. Finally, since $|x_p|= \left(2^{2p_1} + 2^{2p_2}\right)^{1/2}\leq \sqrt{2}.2^{|p|}$, we obtain the existence of a constant  $C_1>0$ independent of $p$ such that :  
    $$\frac{1 + \left|x_p \right|}{D\left(x_p, \mathcal{G}_{C_0} \setminus \left\{x_p \right\}\right)} \leq C_1.$$

\end{proof}

\begin{proof}[Proof of \eqref{H3}]
Let $p = (p_1,p_2)$ be in $\mathcal{P}_{C_0}$. We use \eqref{inclusionVP} to bound the diameter of $V_{p}$ by the diameter of the cube $[0, 2^{|p|+1}] \times [0, 2^{|p|+1}]$, that is : 
     $$Diam\left(V_{p}\right) \leq \sqrt{2}. 2^{|p|+1}.$$
     In addition, we have proved there exists $C>0$ such that for every $x_p \in \mathcal{G}$, we have : 
     $$D\left(x_p, \mathcal{G}_{C_0} \setminus \left\{x_p \right\}\right) \geq C |x_p| \geq C 2^{|p|}.$$
     We directly obtain \eqref{H3}.

\end{proof}

We finally conclude this section establishing an estimate regarding the norm of each element $x_p$ of $\mathcal{G}_{C_0}$. Using Proposition \ref{volumeVx}, the next property shall be useful to estimate the volume of the Voronoi cells in our particular case. 

\begin{prop}
\label{norm2p}
There exists $C_1>0$ and $C_2>0$ such that for every $p$ in $\mathcal{P}_{C_0}$, we have :
\begin{equation}
\label{normestimate}
    C_1 2^{|p|} \leq |x_p| \leq C_2 2^{|p|}.
\end{equation}
\end{prop}

\begin{proof}
For $p \in \mathcal{P}_{C_0}$, we have: 
$$|x_p| = \left(\sum_{i \in {1,...,d}} 2^{2|p_i|}\right)^{1/2}.$$
We first use the inequality $|p_i|\leq |p| $ to obtain the upper bound. That is : 
$$|x_p| \leq \left(\sum_{i \in {1,...,d}} 2^{2|p|}\right)^{1/2} \leq \sqrt{d}2^{|p|}.$$
For the lower bound, we denote $\displaystyle j = \operatorname{argmax}_{i \in \{1,...,d\}} |p_i|$ and we have :
$$|x_p| \geq 2^{|p_j|} = 2^{|p|}.$$
We have established the norm estimate \eqref{normestimate}.
\end{proof}

In the sequel of this work, we only consider the specific set $\mathcal{G}_{C_0}$, defined by \eqref{defG}, for a fixed arbitrary constant $C_0>1$. Therefore, for the sake of clarity and without loss of generality, we will denote $\mathcal{G}$ and $\mathcal{P}$ instead of $\mathcal{G}_{C_0}$ and $\mathcal{P}_{C_0}$.

\section{Properties of the functional space $\mathcal{B}^2(\mathbb{R}^d)$}
\label{Section3}

In this section we prove some properties satisfied by the functional space $\mathcal{B}^2(\mathbb{R}^d)$. The following results are heavily based upon the geometric distribution of the $x_p$. They are key for the understanding of the structure of $\mathcal{B}^2$ and to establish the homogenization of problem \eqref{equationepsilon}. 

To start with, we show the uniqueness of a limit $L^2$-function $f_{\infty}$ in $L^2(\mathbb{R}^d)$ defined in \eqref{spacedef} and characterizing each element of $\mathcal{B}^2(\mathbb{R}^d)$. This result ensures that the definition of the function space $\mathcal{B}^2(\mathbb{R}^d)$ is consistent.  

\begin{prop}
Let $f$ be a function of $\mathcal{B}^2(\mathbb{R}^d)$. Then, the limit function $f_{\infty} \in L^2(\mathbb{R}^d)$ defined in \eqref{spacedef} is unique. 
\end{prop}

\begin{proof}
We assume there exist two functions $f_{\infty}$ and $g_{\infty}$ in $L^2(\mathbb{R}^d)$ such that $$\displaystyle\lim_{|p| \rightarrow \infty}  \|f - \tau_{-p} f_{\infty}\|_{L^2(V_p)} =\displaystyle\lim_{|p| \rightarrow \infty}  \|f - \tau_{-p} g_{\infty}\|_{L^2(V_p)}= 0.$$
By a triangle inequality, we obtain for every $p \in \mathcal{P}$ : 
$$\|\tau_{-p}f_{\infty} - \tau_{-p}g_{\infty}\|_{L^2(V_p)} \leq \|f - \tau_{-p} f_{\infty}\|_{L^2(V_p)} + \|f - \tau_{-p} g_{\infty}\|_{L^2(V_p)} \underset{|p|\to+\infty}{\longrightarrow} 0.$$
In addition, we have $\|\tau_{-p}f_{\infty} - \tau_{-p}g_{\infty}\|_{L^2(V_p)} = \|f_{\infty} - g_{\infty}\|_{L^2(V_p-2^p)}$. According to Proposition \ref{tailledesVP}, we can find a sequence $(p_n)_{n \in \mathbb{N}} \in \mathcal{P}$ such that $\displaystyle \lim_{n \rightarrow \infty}|p_n| = \infty $ and : 
$$\bigcup_{n \in \mathbb{N}}\left(V_{p_n}-2^{p_n}\right) = \mathbb{R}^d.$$
We can finally conclude that $\|f_{\infty} - g_{\infty}\|_{L^2(\mathbb{R}^d)} = 0$, that is $f_{\infty} = g_{\infty}$ in $L^2(\mathbb{R}^d)$.
\end{proof}

We next study the structure of the space $\mathcal{B}^2(\mathbb{R}^d)$ showing two essential properties that shall allow us to establish the existence of the corrector in Section \ref{Section4}. In particular, we prove in Proposition \ref{Banach} that $\mathcal{B}^2(\mathbb{R}^d)$ is a Banach space. 

\begin{prop}
\label{Banach}
The space $\mathcal{B}^2(\mathbb{R}^d)$ equipped with the norm defined by \eqref{normdef}, is a Banach space.
\end{prop}

\begin{proof}
Let $(f_n)_{n \in \mathbb{N}}$ be a Cauchy sequence in $\mathcal{B}^2(\mathbb{R}^d)$. Definitions \eqref{spacedef} and \eqref{normdef} ensure the existence of a Cauchy sequence $f_{n, \infty}$ in $L^2(\mathbb{R}^d)$ such that for every $n \in \mathbb{N}$,  
$$\displaystyle\lim_{|p| \rightarrow \infty}  \|f_n - \tau_{-p} f_{n,\infty} \|_{L^2(V_p)} = 0.$$ Then, for any $\varepsilon>0$, there exists $N\in \mathbb{N}$ such that for all $n>N$, $k>0$ : 
\begin{align}
    & \|f_{n+k} - f_{n}\|_{L^2_{unif}} \leq \varepsilon, \\
    & \|f_{n+k, \infty} - f_{n, \infty}\|_{L^2(\mathbb{R}^d)} \leq \varepsilon, \\
    \label{CauchyVP}
    &\sup_{p \in \mathcal{P}} \|\left(f_{n+k} - \tau_{-p}f_{n+k, \infty}\right) - \left(f_{n} - \tau_{-p}f_{n, \infty}\right)\|_{L^2(V_p)} \leq \frac{\varepsilon}{2}.
\end{align}
Since $L^2$ and $L^2_{unif}$ are Banach spaces, there exist $f \in L^2_{unif}(\mathbb{R}^d)$ and $f_{\infty} \in L^2(\mathbb{R}^d)$ such that $f_n \underset{n\to+\infty}{\longrightarrow} f $ in $L^2_{unif}(\mathbb{R}^d)$ and $f_{n, \infty} \underset{n\to+\infty}{\longrightarrow} f_{\infty} $ in $L^2(\mathbb{R}^d)$. We consider the limit in \eqref{CauchyVP} when $k \rightarrow \infty$ and we obtain : 
$$\sup_{p \in \mathcal{P}}  \|\left(f - \tau_{-p}f_{\infty}\right) - \left(f_{n} - \tau_{-p}f_{n, \infty}\right)\|_{L^2(V_p)} \leq \frac{\varepsilon}{2}.$$
Since $\varepsilon$ can be chosen arbitrary small, we deduce : 
$$\lim_{n \rightarrow \infty} \sup_{p \in \mathcal{P}}  \|\left(f - \tau_{-p}f_{\infty}\right) - \left(f_{n} - \tau_{-p}f_{n, \infty}\right)\|_{L^2(V_p)} = 0.$$
The function $f$ is therefore the limit of $f_n$ for the norm \eqref{normdef}. We just have to show that $f \in \mathcal{B}^2(\mathbb{R}^d)$ to conclude.
Indeed, for a fixed $n>N$ and for $p$ sufficiently large, we have : 
$$\|f_n - \tau_{-p} f_{n,\infty} \|_{L^2(V_p)} \leq \frac{\varepsilon}{2}. $$
Using a triangle inequality, it follows : 
\begin{align*}
\|f - \tau_{-p} f_{\infty} \|_{L^2(V_p)} & \leq \|f_n - \tau_{-p} f_{n,\infty} \|_{L^2(V_p)} + \sup_{p \in \mathcal{P}} \|\left(f - \tau_{-p}f_{\infty}\right) - \left(f_{n} - \tau_{-p}f_{n, \infty}\right)\|_{L^2(V_p)} \\
& \leq \varepsilon.
\end{align*}
Finally, we obtain $\displaystyle\lim_{|p| \rightarrow \infty}  \|f - \tau_{-p} f_{\infty} \|_{L^2(V_p)} = 0$. 
\end{proof}

\begin{prop}
\label{density}
Let $\alpha \in ]0,1[$, then $\mathcal{C}^{0,\alpha}(\mathbb{R}^d)\cap \mathcal{B}^2(\mathbb{R}^d)$ is dense in $\left(\mathcal{B}^2(\mathbb{R}^d), \|.\|_{\mathcal{B}^2(\mathbb{R}^d)}\right).$
\end{prop}

\begin{proof}
We consider $f \in \mathcal{B}^2(\mathbb{R}^d) $ and $f_{\infty} \in L^2(\mathbb{R}^d)$ the associated limit function defined by~\eqref{spacedef}. 
First, for any $\varepsilon >0$, there exists $\phi \in \mathcal{D}(\mathbb{R}^d)$ such that $\| \phi - f_{\infty} \|_{L^2(\mathbb{R}^d)} < \dfrac{\varepsilon}{3}$, thus $\| \tau_{-p} \phi - \tau_{-p} f_{\infty}\|_{L^2(V_p)} \leq \dfrac{\varepsilon}{3}$ for all $p \in \mathcal{P}$.
Second, since $f \in \mathcal{B}^2(\mathbb{R}^d)$ :
$$\exists P^* \in \mathbb{N}, \ \forall p \in \mathcal{P}, \ |p|>P^* \ \Rightarrow 
\| f - \tau_{-p}f_{\infty} \|_{L^2(\textit{V}_p)} < \dfrac{\varepsilon}{3}.$$
Since $\phi$ is compactly supported there also exists $P$, which we can always assume larger than $P^*$, such that for every $|p|>P$ and for all $q\neq p$, we have $(\tau_{-q} \phi)_{|\textit{V}_p} = 0$.

The finite sum $\displaystyle \sum_{|q|\leq P} 1_{\textit{V}_q} f$ (where $1_A$ denotes the indicator function of $A$) is compactly supported and then belongs to $L^2(\mathbb{R}^d)$. Again, we can find $\psi \in \mathcal{D}(\mathbb{R}^d)$ such that $\left\|\psi -~\displaystyle \sum_{|q|\leq P} 1_{\textit{V}_q} f\right\|_{L^2(\mathbb{R}^d)}\leq \dfrac{\varepsilon}{3}$. 
We fix $\displaystyle g = \psi + \sum_{|p|>P} \tau_{-p}\phi$ and we want to show that $g$ is a good approximation of $f$ in $\mathcal{B}^2(\mathbb{R}^d) \cap \mathcal{C}^{0, \alpha}(\mathbb{R}^d)$, that is $g$ is close to $f$ on each $V_p$, uniformly in $p$. First, we have  : 
\begin{equation*}
g_{|\textit{V}_p} = \left\{
    \begin{array}{cc}
        \psi  & \text{if} \ \left|p\right| \leq P, \\
        \psi + \tau_{-p}\phi & \text{else}.
    \end{array}
\right.
\end{equation*}
Therefore $g$ is bounded and we can easily prove that $g \in \mathcal{B}^2(\mathbb{R}^d)\cap \mathcal{C}^{\infty}(\mathbb{R}^d)$ where the associated limit function in $L^2(\mathbb{R}^d)$ is given by $g_{\infty} = \phi$. Furthermore, $g$ is in $\mathcal{C}^{0,\alpha}(\mathbb{R}^d)$ since it is a $\mathcal{C}^{\infty}$ function and all of its derivatives are bounded. Indeed, for every $k$ in $\mathbb{N}^d$, we denote $\partial_k = \partial_{x_1}^{k_1}\partial_{x_2}^{k_2}...\partial_{x_d}^{k_d}$ and we have : 
\begin{equation*}
\left(\partial_k g\right)_{|\textit{V}_p} = \left\{
    \begin{array}{cc}
        \partial_k \psi & \text{if} \ \left|p\right| \leq P, \\
        \partial_k \psi + \tau_{-p}\partial_k \phi & \text{else}.
    \end{array}
\right.
\end{equation*}
and $\partial_k g$ is clearly bounded.

Let $p \in \mathcal{P}$, we consider two cases.  If $|p|\leq P$, then :
$$\| g - f \|_{L^2(\textit{V}_p)} = \left\| \psi - \sum_{|q|\leq P} 1_{\textit{V}_q} f \right\|_{L^2(\textit{V}_p)} \leq \varepsilon.$$
Else, if $|p|>P$, we have : 
\begin{align*}
\| g - f \|_{L^2(\textit{V}_p)} & = \| \psi +\tau_{-p}\phi - f \|_{L^2(\textit{V}_p)} \\
& \leq \|  \psi \|_{L^2(\textit{V}_p)}+ \|  \tau_{-p}\phi - \tau_{-p} f_{\infty} \|_{L^2(\textit{V}_p)} + \| \tau_{-p} f_{\infty} -  f\|_{L^2(\textit{V}_p)}\\
& \leq \varepsilon.
\end{align*}
And we can conclude. 
\end{proof}
We now establish a property regarding multiplication of elements of $\mathcal{B}^2(\mathbb{R}^d)$.

\begin{prop}
\label{stability}
Let $g$ and $h$ be in $\mathcal{B}^2(\mathbb{R}^d) \cap L^{\infty}(\mathbb{R}^d)$. We assume the associated $L^2$ function of $g$, denoted by $g_{\infty}$, is in $L^{\infty}(\mathbb{R}^d)$, then $hg \in \mathcal{B}^2(\mathbb{R}^d)$.
\end{prop}

\begin{proof}
Since $g_{\infty} \in L^{\infty}(\mathbb{R}^d)$, we clearly have $g_{\infty} h_{\infty} \in L^2(\mathbb{R}^d)$. Using that for all $p \in \mathcal{P}$ : 
$$g h - \tau_{-p}(g_{\infty} h_{\infty}) = (h - \tau_{-p}h_{\infty})\tau_{-p}g_{\infty} + (g - \tau_{-p}g_{\infty})h.$$
We have by the triangle inequality : 
\begin{align*}
\|g h - \tau_{-p}(g_{\infty}h_{\infty})\|_{L^2(\textit{V}_p)}& \leq  \|h - \tau_{-p}h_{\infty}\|_{L^2(\textit{V}_p)} \| g_{\infty}\|_{L^{\infty}(\mathbb{R}^d)}\\
& + \|g - \tau_{-p}g_{\infty}\|_{L^2(\textit{V}_p)} \|h\|_{L^{\infty}(\mathbb{R}^d)}.
\end{align*}
It follows, taking the limit for $|p| \rightarrow \infty$, that $gh \in \mathcal{B}^2(\mathbb{R}^d)$ and that $\left(gh\right)_{\infty} = g_{\infty}h_{\infty}$. 
\end{proof}

Our next result is one of the most important properties for the sequel. As we shall see in section \ref{Section5}, it first implies that the homogenized coefficient in our setting is the same as the homogenized coefficient in the periodic case, that is, without perturbation.  In addition, it gives some information about the growth of the corrector defined in Theorem \ref{theoreme1} (in particular, we give a proof in proposition \ref{propsouslinearite} of the strict sublinearity of the corrector). We will use all of these properties to prove the convergence stated in Theorem \ref{theoreme3} in our case. 
\begin{prop}
\label{moyenne}
Let $u \in \mathcal{B}^2(\mathbb{R}^d)$. Then, for every $x_0 \in \mathbb{R}^d$  :  
\begin{equation}
\label{convergencemeanB2}
\lim_{R\rightarrow\infty}\dfrac{1}{|B_R|}\int_{B_R(x_0)}|u(x)|dx = 0,
\end{equation}
with the following convergence rate : 
\begin{equation}
\label{convergencerate}
\dfrac{1}{|B_R|}\int_{B_R(x_0)}|u(x)|dx \leq C \left(\dfrac{\log R}{R^d}\right)^{\frac{1}{2}},
\end{equation}
where $C>0$ is independent of $R$ and $x_0$. 
\end{prop}

\begin{proof}
We fix $R>0$. Using the Cauchy-Schwarz inequality, we have :
\begin{align*}
\dfrac{1}{|B_R|}\int_{B_R(x_0)}|u(x)|dx & \leq \dfrac{1}{\sqrt{|B_R|}} \left( \int_{B_R(x_0)}|u(x)|^2dx \right) ^{\frac{1}{2}}= \dfrac{1}{\sqrt{|B_R|}} \left( \sum_{p \in \mathcal{P}}   \int_{\textit{V}_p \cap B_R(x_0)}|u(x)|^2dx \right) ^{\frac{1}{2}}.
\end{align*}
Since the number of $V_p$ such that $B_R(x_0)\cap V_p \neq \emptyset$ is bounded by $\log(R)$ according to Corollary \ref{corollogboundVP}, we obtain : 
\begin{align*}
\dfrac{1}{|B_R|}\int_{B_R(x_0)}|u(x)|dx&\leq \dfrac{\left(\log R\right)^{\frac{1}{2}}}{\sqrt{|B_R|}}\sup_p\|u\|_{ L^2(\textit{V}_p)} \leq C(d)   \left(\dfrac{\log(R)}{R^d}\right)^{\frac{1}{2}}\sup_p \|u\|_{ L^2(\textit{V}_p)}.  
\end{align*}
Here, $C(d)$ depends only on the ambient dimension $d$. The last inequality yields \eqref{convergencerate} and conclude the proof. 
\end{proof}

\begin{corol}
\label{convergenceLinfinistar}
Let $u \in \mathcal{B}^2(\mathbb{R}^d) \cap L^{\infty}(\mathbb{R}^d)$, then $|u(./\varepsilon)|$ is convergent to 0 in the weak*-$L^{\infty}$ topology when $\varepsilon \rightarrow 0$. 
\end{corol}

\begin{proof}
We fix $R>0$ and we first consider $\varphi = 1_{B_R}$. For any $\varepsilon>0$, we have :
\begin{align*}
\left|\int_{\mathbb{R}^d} |u(x/\varepsilon)|\varphi(x) dx\right| & \leq \int_{B_R} \left| u(x/\varepsilon)  \right| dx \\
& \stackrel{y=x/\varepsilon}{=} \varepsilon^d   \int_{B_{R/\varepsilon}} \left| u(y) \right|dy \\
&= \left|B_R\right| \frac{\varepsilon^d}{\left|B_R\right|}  \int_{B_{R/\varepsilon}} \left| u(y) \right|dy \\
&= \|\varphi\|_{L^{1}(\mathbb{R}^d)}\frac{\varepsilon^d}{\left|B_R\right|}  \int_{B_{R/\varepsilon}} \left| u(y) \right|dy.
\end{align*}
We next use \eqref{convergencerate} in the right-hand term and we obtain the existence of $C>0$ independent of $\varepsilon$ and $\varphi$ such that : 
$$ \left|\int_{\mathbb{R}^d} u(x/\varepsilon)\varphi(x) dx\right| \leq C \|\varphi\|_{L^{1}(\mathbb{R}^d)} \left(\varepsilon^d \log(1/\varepsilon)\right)^{\frac{1}{2}} \underset{\varepsilon\to 0}{\longrightarrow} 0.$$
We conclude using the density of simple functions in $L^1(\mathbb{R}^d)$.
\end{proof}

We next introduce the notion of sub-linearity which is actually a fundamental property in homogenization. Indeed, in order to precise the convergence of the approximated sequence of solutions \eqref{approximatesequence}, we have to study the behavior of the sequences $\varepsilon w_{e_i}(./\varepsilon)$ when $\varepsilon \rightarrow 0$. The convergence to zero of these sequences and the understanding of the rate of convergence are key for establishing estimates \eqref{estimate1} and \eqref{estimate2} stated in Theorem \ref{theoreme3}. In the sequel, we therefore study this phenomenon for the functions with a gradient in $\mathcal{B}
^2(\mathbb{R}^d)$. 

\begin{dftn}
A function $u$ is strictly sub-linear at infinity if : 
\begin{equation}
\label{defsublinearity}
\lim_{|x| \rightarrow \infty} \dfrac{|u(x)|}{1 + |x|} = 0.
\end{equation}
\end{dftn}

In the next proposition we prove the sub-linearity of all the functions $u$ such that $\nabla u \in \left(\mathcal{B}^2(\mathbb{R}^d) \cap L^{\infty}(\mathbb{R}^d)\right)^d$. We assume, for this general property only, that $d \geq 2$.

\begin{prop}
\label{propsouslinearite}
Assume $d\geq 2$. Let $u \in H^1_{loc}(\mathbb{R}^d)$ with $\nabla u \in \left(\mathcal{B}^2(\mathbb{R}^d) \cap L^{\infty}(\mathbb{R}^d)\right)^d$. Then $u$ is strictly sub-linear at infinity and for all $s >d $, there exists $C>0$ such that for every $x,y \in \mathbb{R}^d$ with $x \neq y$ : 
\begin{equation}
\label{souslinearite}
\left|u(x) - u(y)\right| \leq C \left|\log(\left|x-y\right|)\right|^{\frac{1}{s}} \left|x-y \right|^{1-\frac{d}{s}}.
\end{equation}
\end{prop}

\begin{proof}
Let $x,y \in \mathbb{R}^d$ with $x \neq y$ and fix $r = |x-y|$. Since $\nabla u \in \left(L^{\infty}(\mathbb{R}^d)\right)^d$, we have $\nabla u \in \left(L^s_{loc}(\mathbb{R}^d)\right)^d$ for every $s \geq 1$. We next fix $s>d$. We know there exists a constant $C>0$, depending only on $d$, such that :
\begin{equation}
\label{morreys inequality}
|u(x) - u(y)| \leq C r \left( \frac{1}{r^d}\int_{B_r(x)} \left| \nabla u (z)\right|^s dz \right)^{\frac{1}{s}}.
\end{equation}
This estimate is established for instance in \cite[Remark p.268]{evans10} as corollary of the Morrey's inequality (\cite[Theorem 4 p.266]{evans10}).
Since $s>d\geq 2$, we use the boundedness of $\nabla u$ to obtain : 
\begin{equation}
\label{integralbound}
|u(x) - u(y)| \leq C \displaystyle \|\nabla u\|^{(s-2)/s}_{L^{\infty}(\mathbb{R}^d)} r \left( \frac{1}{r^d}\int_{B_r(x)} \left| \nabla u (z)\right|^2 dz \right)^{\frac{1}{s}}.
\end{equation}
We next split the integral of \eqref{integralbound} on each $V_p$ such that $V_p\cap B_r(x)\neq \emptyset$ and we have : 
\begin{equation*}
\label{integralbound2}
|u(x) - u(y)| \leq C \displaystyle \|\nabla u\|^{(s-2)/s}_{L^{\infty}(\mathbb{R}^d)} r \left( \frac{1}{r^d} \sum_{p\in \mathcal{P}} \int_{B_r(x)\cap V_p} \left| \nabla u (z)\right|^2 dz \right)^{\frac{1}{s}}.
\end{equation*}
We finally use Corollary \ref{corollogboundVP} and we obtain the existence of a constant $C_1 >0$ such that : $$|u(x) - u(y)| \leq C_1 \|\nabla u\|^{(s-2)/s}_{L^{\infty}(\mathbb{R}^d)} \|\nabla u\|^{2/s}_{\mathcal{B}^2(\mathbb{R}^d)} \left|\log(r)\right|^{\frac{1}{s}}r^{1-\frac{d}{s}}.$$
This inequality is true for all $s>d$, which allows us to conclude. In addition, the sub-linearity of $u$ is obtained fixing $y=0$ and letting $|x|$ go to the infinity in estimate \eqref{souslinearite}.  
\end{proof}

\begin{remark}
In the case $d=1$, since $s\geq 2$, the above proof gives  : 
\begin{equation}
|u(x) - u(y)| \leq C \left|\log\left|x-y\right|\right|^{\frac{1}{2}}\left|x-y\right|^{\frac{1}{2}}. 
\end{equation}
\end{remark}

The last proposition of this section gives an uniform estimate of the integral remainders of the functions of $\mathcal{B}^2(\mathbb{R}^d)$. The idea here is that the functions of $\mathcal{B}^2(\mathbb{R}^d)$ behave like a fixed $L^2$-functions at the vicinity of the points of $\mathcal{G}$ and therefore, have to be small in a $L^2$ sens far from these points. This property will be used in the proof of Lemma \ref{lemmeexistenceperiodique1} in next section to establish an estimate in $\mathcal{B}^2(\mathbb{R}^d)$ satisfied by the solutions to diffusion equation~\eqref{eqper1}. 

\begin{prop}
\label{uniformmajoration}
Let $f$ be in $\mathcal{B}^2(\mathbb{R}^d)$ and $f_{\infty}$ the associated limit function in $L^2(\mathbb{R}^d)$. For any $\varepsilon >0$, there exists $R^*>0$ such that for every $R>R^*$ and every $p,q \in \mathcal{P}$ : 
\begin{equation}
    \left( \int_{V_q\cap B_R(2^q)^c} \left|f - \tau_{-p}f_{\infty}\right|^2 \right)^{1/2} < \varepsilon,
\end{equation}
where $B_R(2^q)^c$ denotes the set $\mathbb{R}^d \setminus{B_R(2^q)}$. Therefore, we have the following limit : 
\begin{equation}
    \lim_{R \rightarrow \infty} \sup_{\substack{(p,q) \in \mathcal{P}^2 \\ p \neq q}} \left( \int_{V_q\cap B_R(2^q)^c} \left|f - \tau_{-p}f_{\infty}\right|^2 \right)^{1/2} = 0.
\end{equation}
\end{prop}

\begin{proof}
Let $\varepsilon>0$. First, for every $R>0$, $p, q \in \mathcal{P}$ we use a triangle inequality and we obtain the following upper bound : 
\begin{align*}
\left( \int_{V_q\cap B_R(2^q)^c} \left|f - \tau_{-p}f_{\infty}\right|^2  \right)^{1/2}  & \leq  \left( \int_{V_q\cap B_R(2^q)^c} \left|f - \tau_{-q}f_{\infty}\right|^2 \right)^{1/2} + \left( \int_{V_q\cap B_R(2^q)^c} \left| \tau_{-q}f_{\infty}\right|^2 \right)^{1/2} \\
& + \left( \int_{V_q\cap B_R(2^q)^c} \left| \tau_{-p}f_{\infty}\right|^2 \right)^{1/2} \\
& =  I^{p,q}_{1}(R) + I^{p,q}_{2}(R) + I^{p,q}_{3}(R).
\end{align*}
We want to bound the three terms $\displaystyle I^{p,q}_{1}(R)$, $I^{p,q}_{2}(R)$ and $I^{p,q}_{3}(R)$ by $\varepsilon$ uniformly in $p,q$. 

We start by considering $I^{p,q}_{1}$. We have assumed that $f \in \mathcal{B}^2(\mathbb{R}^d)$, then, by definition, there exists $P>0$ such that for every $q \in \mathcal{P}$ satisfying $|q|>P$, we have : 
$$\left( \int_{V_q}\left|f - \tau_{-q}f_{\infty}\right|^2 \right)^{1/2} < \frac{\varepsilon}{3}.$$
In addition, since the volume of each $V_q$ is finite according to assumption \eqref{H1}, there exists $R_1>0$ such that for every $|q|\leq P$, $B_{R_1}(2^q)^c \cap V_q = \emptyset$. Therefore, as soon as $|q|\leq P$ and $R\geq R_1$, we have $I_{p,q}(R)=0$. Finally, considering successively the case $|q|\leq P$ and the case $|q|>P$, we obtain for every $q,p\in \mathcal{P}$ and $R\geq R_1$ : 

\begin{equation}
    I^{p,q}_{1}(R) < \frac{\varepsilon}{3}.
\end{equation}
We next study the second term $I^{p,q}_{2}$. Since $f_{\infty}$ is in $L^2(\mathbb{R}^d)$, there exists $R_2>0$, which we can always assume larger than $R_1$, such that for every $q \in \mathcal{P}$ :
\begin{equation}
\label{majorationqueueL2}
\left( \int_{B_{R_2}(2^q)^c} \left| \tau_{-q}f_{\infty}(y)\right|^2 dy\right)^{1/2} \stackrel{x=y-2^q}{=} \left( \int_{B_{R_2}^c} \left| f_{\infty}(x)\right|^2 dx \right)^{1/2} < \frac{\varepsilon}{3}.
\end{equation}
And we directly obtain, for every $R \geq R_2$ : 
\begin{equation}
    I^{p,q}_{2}(R) < \frac{\varepsilon}{3}.
\end{equation}
Finally, in order to bound the last term, we know that $\displaystyle \lim_{|l| \rightarrow \infty} D\left(2^l, \mathcal{G}\setminus{\{2^l\}}\right) = +\infty$ as a consequence of Assumption \eqref{H2}. Therefore, there exists a finite number of indices $l$ such that : 
\begin{equation}
\label{distance}
  D\left(V_l , \mathcal{G}\right) \leq R_2.
\end{equation}
Thus, we deduce the existence of a positive radius $R_3$ independent of $q,p$ such that for every $l$ satisfying \eqref{distance} we have $B_{R_3}(2^l)^c \cap V_l = \emptyset$. Again we can always assume $R_3$ larger than $R_2$. There are two cases depending on the value of q: 
\begin{itemize}
    \item[1)] If $q$ satisfies \eqref{distance}, we have $B_{R_3}(2^q)^c \cap V_q = \emptyset$ and we obtain $I^{p,q}_{3}(R_3) = 0$
    \item[2)] Else, for every $y \in V_q$, we have $|y - 2^p| > R_2$. Therefore : 
    \begin{align*}
        I^{p,q}_{3}(R_3) & \leq \left( \int_{V_q} \left| \tau_{-p}f_{\infty}\right|^2 \right)^{1/2} \stackrel{x=y-2^p}{=} \left( \int_{V_q-2^p} \left| f_{\infty}\right|^2 \right)^{1/2} \leq \left( \int_{B_{R_2}^c} \left| f_{\infty}\right|^2 \right)^{1/2}.
    \end{align*}
Using \eqref{majorationqueueL2}, we have for every $R \geq R_3$, $I^{p,q}_{3}(R) \leq \frac{\varepsilon}{3}$.
\end{itemize}
In the two cases , we obtain for $R\geq R_3$ : 
\begin{equation}
    I^{p,q}_{3}(R) \leq \frac{\varepsilon}{3}.
\end{equation}
Since the values of $R_3$ is independent of $p$ and $q$ we can conclude the proof for $R^* =R_3$.

\end{proof}

\section{Existence result for the corrector equation}
\label{Section4}

This section is devoted to the proof of Theorem \ref{theoreme1}. Equation \eqref{correcteur} being posed on the whole space $\mathbb{R}^d$, we need to use here the geometric distribution of the $2^p$ and introduce some constructive techniques involving the fundamental solution of the operator $-\operatorname{div}(a \nabla.)$ to solve it.
To start with, we establish some general results on equations  
\begin{equation}
\label{equationref}
    -\operatorname{div}(a \nabla u) = \operatorname{div}(f) \quad \text{in } \mathbb{R}^d,
\end{equation}
for coercive coefficients $a$ of the form \eqref{coefficientform} and right hand side $f$ in $\left(\mathcal{B}^2(\mathbb{R}^d)\right)^d$ in order to deduce the existence of the corrector stated in Theorem \ref{theoreme1}. For this purpose, we consider the following strategy adapted from \cite{blanc2018correctors}  : We first study diffusion problem \eqref{eqper1} in the periodic context, that is, when the diffusion coefficient $a=a_{per}$ is periodic. Secondly we show in lemma \ref{lemme2} the continuity of the associated reciprocal linear operator $\nabla\left(-\operatorname{div}a\nabla\right)^{-1}\operatorname{div}$ from $\mathcal{B}^2(\mathbb{R}^d)$ to $\mathcal{B}^2(\mathbb{R}^d)$. Finally, we use this continuity in order to generalize the existence results of the periodic context to the general context when $a$ is a perturbed coefficient of the form \eqref{coefficientform}. To this end, we apply a method based on the connexity of the set $\mathcal{I} = [0,1]$ as we shall see in the proof of Lemma \ref{lemmexistence}.

\subsection{Preliminary uniqueness results}

We begin by establishing the uniqueness of a solution $u$ to \eqref{equationref} such that $\nabla u \in \mathcal{B}^2(\mathbb{R}^d)^d$. This result is actually essential in the proof of Theorem \ref{theoreme1} since it both ensures the uniqueness of the corrector solution \eqref{correcteur} and also allows us to establish the continuity estimate of Lemma \ref{lemme2} which is key in our approach to show the existence of a solution to \eqref{equationref}. 

\begin{lemme}
\label{lemme1}
Let $a$ be an elliptic and bounded coefficient, and $u \in H^1_{loc}(\mathbb{R}^d)$, such that $\displaystyle \sup_{p \in \mathcal{P}}  \int_{\textit{V}_p}|\nabla u|^2 < \infty$, be a solution to : 
\begin{equation}
\label{equationhomog}
 -\operatorname{div}(a\nabla u) = 0 \quad\text{in } \mathbb{R}^d,  
\end{equation}
in the sense of distribution. Then $\nabla u$ = 0.
\end{lemme}

\begin{proof}
we consider $u\in H^1_{loc}(\mathbb{R}^d)$ solution of \eqref{equationhomog}.  Since $u$ is a solution to \eqref{equationhomog}, there exists $C>0$ such that for every $R >0$, we have the following estimate (for details see for instance \cite[Proposition 2.1 p.76]{giaquinta1983multiple} and \cite[ Remark 2.1 p.77]{giaquinta1983multiple} )  :
\begin{equation}
\int_{B_R} |\nabla u | ^2 \leq \frac{C}{R^2} \int_{A_{R,2R}} |u- \langle u \rangle_{A_{R,2R}}| ^2,
\end{equation}
where : 
$$\langle u \rangle_{A_{R,2R}} = \frac{1}{\left|A_{R,2R}\right|}\int_{A_{R,2R}} u(x) dx.$$
We use the Poincar\'e-Wirtinger inequality on the right-hand side and we obtain : 
\begin{equation}
\label{3}
\int_{B_R} |\nabla u| ^2 \leq C \int_{A_{R,2R}} |\nabla u| ^2.
\end{equation}
Furthermore, we can write this inequality in the following form : 
\begin{equation}
\label{3bis}
\int_{B_R} |\nabla u| ^2 \leq \dfrac{C}{1+C} \int_{B_{2R}} |\nabla u| ^2. 
\end{equation}
In addition, using Proposition \eqref{logboundVP}, we know there exists a constant $C_1>0$ independent of $R$ such that :  
\begin{equation}
\label{4}
\int_{B_{2R}} |\nabla u| ^2 = \sum_{\textit{V}_p \cap B_{2R} \ne \emptyset } \int_{\textit{V}_p \cap B_{2R}} |\nabla u|^2 \leq C_1 \log(2R) \sup_p \int_{\textit{V}_p} |\nabla u|^2.
\end{equation}

Next, we define $F(R)= \displaystyle \int_{ B_{R}} |\nabla u|^2$. The inequalities \eqref{3bis} and \eqref{4} yield
 for all $R > 0$ and for every $n \in \mathbb{N}^{*}$, we obtain  
 
$$F(R) \leq \big(\dfrac{C}{1+C}\big)^n F(2^nR) \leq C_1 \big(\dfrac{C}{1+C}\big)^n \log(2^n R)\sup_p \int_{\textit{V}_p} |\nabla u|^2.$$
Since $ \dfrac{C}{1+C}<1$, we have :  
$$\lim_{n \rightarrow \infty} \big(\dfrac{C}{1+C}\big)^n \log(2^n R) =0,$$
and it therefore follows, letting $n$ go to infinity, that $F(R)=0$ for all $R>0$, thus $\nabla u = 0$. 
\end{proof}

\begin{corol}
\label{remarkuniqueness}
Let $f \in \mathcal{B}^2(\mathbb{R}^d)^d$, then a solution $u$ of \eqref{equationref} 
with $\nabla u \in \mathcal{B}^2(\mathbb{R}^d)^d$ is unique up to an additive constant. 
\end{corol}

\begin{remark}
Here the restriction made on the dimension is actually not necessary. The result and the proof of Lemma \ref{lemme1} of uniqueness still hold if we assume $ d = 1 $ or $ d = 2$.
\end{remark}

\begin{remark}
We remark that Assumptions \eqref{coefficientform} and \eqref{hypothèses2} regarding the structure and the regularity of the coefficient $a$ are not required to establish the uniqueness result of Lemma \ref{lemme1}. In the proof, we only use the "Hilbert" structure of $L^2$, induced by the assumptions satisfied by $u$, and the fact that $a$ is elliptic and bounded. 
\end{remark}

\subsection{Existence results in the periodic problem}

Now that uniqueness has been dealt with, we turn to the existence of the solution to \eqref{equationref}. We need to first establish it for a periodic coefficient considering the equation : 
\begin{equation}
\label{eqper1}
    - \operatorname{div}(a_{per} \nabla u) = \operatorname{div}(f) \quad \text{in } \mathcal{D}'(\mathbb{R}^d). 
\end{equation}

We start by introducing the Green function $G_{per}: \mathbb{R}^d \times \mathbb{R}^d \rightarrow \mathbb{R}$ associated with the operator $-\operatorname{div}\left(a_{per}\nabla. \right)$ on $\mathbb{R}^d$. That is, the unique solution to 
\begin{equation}
    \label{defGreen}
    \left\{ 
\begin{array}{cc}
   - \operatorname{div}_x \left(a_{per}(x) \nabla_x G_{per}(x,y) \right) = \delta_y (x)  &  \text{in } \mathcal{D}'(\mathbb{R}^d), \\
    \displaystyle \lim_{|x-y| \rightarrow \infty} G_{per}(x,y) = 0. &
\end{array}  
\right.
\end{equation}

According to the results established in \cite[Section 2]{avellaneda1991lp} about the asymptotic growth of the Green function (see also \cite[theorem 13, proof of lemma 17]{avellaneda1987compactness} and \cite{gruter1982green} for bounded domain or \cite[proposition 8]{ anantharaman2011asymptotic} for additional details), there exists $C_1>0$, $C_2>0$ and $C_3>0$ such that for every $x,y \in \mathbb{R}^d$ with $x\neq y$ :
\begin{align}
\label{estimgreen1}
|\nabla_y G_{per}(x,y)| & \leq C_1 \dfrac{1}{|x-y|^{d-1}},\\
\label{estimgreen3}
|\nabla_x G_{per}(x,y)| & \leq C_2 \dfrac{1}{|x-y|^{d-1}},\\
\label{estimgreen2}
|\nabla_x \nabla_y G_{per}(x,y)| & \leq C_3 \dfrac{1}{|x-y|^{d}}.
\end{align}

We first introduce a result of existence in the $L^2(\mathbb{R}^d)$ case. The following lemma allows us to define a solution to \eqref{eqper1} using the Green function when $f$ belongs to $\left(L^2(\mathbb{R}^d)\right)^d$. The proof of this result is established in \cite{avellaneda1991lp}.

\begin{lemme}
\label{reultatexistenceL2}
Let $f$ be in $\left(L^2(\mathbb{R}^d)\right)^d$, then the function : 
\begin{equation}
    \label{defuintegralL2}
    u = \int_{\mathbb{R}^d} \nabla_y G_{per}(.,y) .f(y) dy,
\end{equation}
is a solution in $H^1_{loc}(\mathbb{R}^d)$ to \eqref{eqper1}
such that $\nabla u \in \left(L^2(\mathbb{R}^d)\right)^d$. 
\end{lemme}

Our aim is now to generalize the above result to our case and, in particular, to give a sense to the function $u$ define by \eqref{defuintegralL2} when $f\in \left(\mathcal{B}^2(\mathbb{R}^d)\right)^d$. The idea here is to split the function $f$ into a sum of $L^2$-functions $f_p$ compactly supported in each $V_p$ for $p\in \mathcal{P}$. Using Lemma \ref{reultatexistenceL2}, we shall obtain the existence of a collection $u_p$ of solution to \eqref{eqper1} when $f = f_p$. The main difficulty here is to show that the function $u$ defined as the sum of the $u_p$ is bounded. 
 
\begin{lemme}
\label{lemmeexistenceperiodique1}
Let $f\in \left(L^2_{loc}(\mathbb{R}^d)\right)^d$ such that $ \displaystyle \sup_{p \in \mathcal{P}} \|f\|_{L^2(V_p)} < \infty$, then the function $u$ defined by
\begin{equation}
\label{defuintegrale}
 u = \int_{\mathbb{R}^d}  \nabla_y G_{per}(.,y) f(y) dy
\end{equation}
is a solution in $H^1_{loc}(\mathbb{R}^d)$ to \eqref{eqper1}.
In addition, $u$ is the unique solution to \eqref{eqper1} which satisfies $\displaystyle \sup_{p \in \mathcal{P}} \|\nabla u \|_{L^2(V_p)} < \infty$ and there exists $C>0$ independent of $f$ and $u$ such that we have the following estimate : 
\begin{equation}
\label{estimper}
\sup_{p\in \mathcal{P}} \|\nabla u\|_{L^2(\textit{V}_p)} \leq C\sup_{p\in \mathcal{P}} \|f\|_{L^2(\textit{V}_p)}.
\end{equation}
\end{lemme}

\begin{proof}
\textit{Step 1 :  $u$ is well defined}

We start by proving that definition \eqref{defuintegrale} makes sense and, in particular, that the above integral defines a function $u$ solution to \eqref{eqper1} in $H^1_{loc}(\mathbb{R}^d)$. In the sequel the letter $C$ denotes a generic constant that may change for one line to another. For every $q \in \mathcal{P}$, we first introduce a set $W_q$ and five constants $C_1$, $C_2$, $C_3$, $C_4$ and $C_5$ independent of $q$ and defined by Proposition \ref{openconstruction} such that :  
\begin{enumerate}[label=(\roman*)]
\item $\textit{V}_q \subset{W}_q$, \label{i)}  
\item $Diam(W_q) \leq C_1 2^{|q|}$, and $D(\textit{V}_q,\partial W_q) \geq C_2 2^{|q|},$ \label{ii)} 
\item  $\forall r \in \mathcal{P} \setminus{\{q\}}$, $Dist(2^r, W_q) \geq C_3 2^{|q|}$,  \label{iii)} 
\item $\sharp \left \{ r \in \mathcal{P} \middle| V_r \cap W_q \neq \emptyset \right\} \leq C_4$, \label{iv)} 
\item $\forall r \in \mathcal{P} \setminus{\{q\}}$, $D\left(V_q, V_r\setminus{W_q}\right) \geq C_5 2^{|r|}.$ 
\label{v)}
\end{enumerate}
To start with, we define for each $q \in \mathcal{P}$ the function : 
\begin{equation}
\label{defuk}
    u_q = \int_{\mathbb{R}^d} \nabla_y G_{per}(.,y) f(y)1_{V_q}(y) dy. 
\end{equation}
Lemma \ref{reultatexistenceL2} ensures this function is a solution in $H^1_{loc}(\mathbb{R}^d)$ to : 
\begin{equation}
\left\{
\begin{array}{cc}
    - \operatorname{div}(a_{per} \nabla u_q) = \operatorname{div}(f 1_{V_p})  &  \text{in } \mathbb{R}^d,\\
    \nabla u_q \in \left(L^2(\mathbb{R}^d)\right)^d. & 
\end{array}
\right.
\end{equation}
Considering the gradient of \eqref{defuk}, we have for every $x \in \mathbb{R}^d\setminus{V_q}$ : 
\begin{equation}
\nabla u_q(x) = \int_{\mathbb{R}^d} \nabla_x \nabla_y G_{per}(x,y) f(y) 1_{V_q}(y) dy.    
\end{equation}
Next, for every $N \in \mathbb{N^*}$, we define :
\begin{equation}
\label{SeriesUN}
 U_N = \sum_{q \in \mathcal{P}, \ |q|\leq N}  u_q,   
\end{equation}
and
\begin{equation}
\label{SeriesSN}
S_N = \nabla U_N = \sum_{q \in \mathcal{P}, \ |q|\leq N} \nabla u_q.    
\end{equation}

We next show that the two series $U_N$ ans $S_N$ are convergent in $L^2_{loc}(\mathbb{R}^d)$. To this aim, since the collection $\left(V_p\right)_{p\in \mathcal{P}}$ is a partition of $\mathbb{R}^d$, it is sufficient to prove that they normally converge in $L^2(V_p)$ for every $p \in \mathcal{P}$. 
We fix $p \in \mathcal{P}$ and for every $q \in \mathcal{P}$ such that $V_q \cap W_p = \emptyset $, we use the Cauchy-Schwarz inequality to obtain : 
\begin{align*}
    \|u_q\|_{L^2(V_p)} & = \left(\int_{V_p} \left| \int_{V_q} \nabla_y G_{per}(x,y) f(y) dy \right|^2 dx \right)^{1/2}\\
    & \leq \left(\int_{V_p} \int_{V_q} \left| \nabla_y G_{per}(x,y) \right|^2 dy \int_{V_q} \left| f(y) \right|^2 dy \ dx \right)^{1/2}.
\end{align*}
Next, estimate \eqref{estimgreen1} gives : 
\begin{align}
\label{formuleimportante1}
    \|u_q\|_{L^2(V_p)} & \leq C \sup_{r \in \mathcal{P}} \|f\|_{L^2(V_r)} \left(\int_{V_p} \int_{V_q}  \frac{1}{|x-y|^{2d-2}}dy \ dx \right)^{1/2}.
\end{align}
Since $V_q \cap W_p = \emptyset$, Property \ref{v)} gives the existence of $C>0$ such that for every $x\in V_p$ and $y\in V_q$, we have $|x-y|\geq C 2^{|q|}$. We next use Propositions \ref{volumeVx} and \ref{norm2p} to obtain the existence of a constant $C>0$ independent of $p$ and $q$ such that $\left|V_q\right| \leq C 2^{d|q|}$. Finally : 
\begin{align*}
    \|u_q\|_{L^2(V_p)} & \leq C \sup_{r \in \mathcal{P}} \|f\|_{L^2(V_r)} \left(\int_{V_p} \int_{V_q}  \frac{1}{2^{(2d-2)|q|}}dy \ dx \right)^{1/2} \\
    & \leq C \sup_{r \in \mathcal{P}} \|f\|_{L^2(V_r)} \left(\int_{V_p} \frac{\left|V_q\right|}{2^{(2d-2)|q|}} dx \right)^{1/2} \\
    & \leq C \sup_{r \in \mathcal{P}} \|f\|_{L^2(V_r)} \left|V_p\right|^{1/2}  \frac{1}{2^{|q|(d/2-1)}}.
\end{align*}
We thus obtain the following upper bound : 
\begin{align*}
 \sum_{q \in \mathcal{P}} \|u_q\|_{L^2(V_p)} & = \sum_{\substack{q \in \mathcal{P}, \\ V_q \cap W_p \neq \emptyset}} \|u_q\|_{L^2(V_p)} + \sum_{\substack{q \in \mathcal{P}, \\ V_q \cap W_p = \emptyset}} \|u_q\|_{L^2(V_p)} \\
 & \leq \sum_{\substack{q \in \mathcal{P}, \\ V_q \cap W_p \neq \emptyset}} \|u_q\|_{L^2(V_p)} + C \sum_{\substack{q \in \mathcal{P}}} \frac{1}{2^{|q|(d/2-1)}}.
\end{align*}
The first sum is finite according to Property \ref{iv)} and we only have to prove the convergence of the second one. We have assumed $d>2$ and consequently $d/2-1 >0$. In addition, since the number of $q\in \mathcal{P}$ such that $|q|=n\in \mathbb{N}$ is bounded independently of $n$ (as a consequence of Proposition \ref{numberpointAn}), we have : 
\begin{equation}
\label{serieimportante}
\sum_{q \in \mathcal{P}} \frac{1}{2^{|q|(d/2-1)}} \leq C \sum_{n \in \mathbb{N}} \frac{1}{2^{n(d/2-1)}} < \infty.
\end{equation}
Therefore, for every $p \in \mathcal{P}$, the absolute convergence of $U_N$ to $u$ in $L^2(V_p)$ is proved. That is, since the sequence of the sets $V_q$ defines a partition of $\mathbb{R}^d$, $U_N$ converges to $u$ in $L^2_{loc}(\mathbb{R}^d)$. Using asymptotic estimate \eqref{estimgreen2} for $\nabla_x \nabla_y G_{per}$ we can conclude with the same arguments to prove the convergence of $S_N$ in $L^2_{loc}(\mathbb{R}^d)$. In addition, the gradient operator being continuous in $\mathcal{D}'(\mathbb{R}^d)$, we have :
$$\sum_{q \in \mathcal{P}} \nabla u_q = \displaystyle \lim_{N \rightarrow \infty} S_N = \displaystyle \lim_{N \rightarrow \infty} \nabla U_N =  \nabla u.$$
To complete the proof, we have to show that $u$ is a solution to \eqref{eqper1}. Let $N$ be in $\mathbb{N}$. By linearity of the operator $\operatorname{div(a_{per} \nabla.)}$, $U_N$ is a solution in $H^1_{loc}(\mathbb{R}^d)$ to : 
\begin{equation}
\label{equationtronquee}
-\operatorname{div}\left(a_{per} \nabla U_N\right) = \operatorname{div}\left(\sum_{q \in \mathcal{P}, \ |q|\leq N} 1_{V_q} f \right) \quad \text{in } \mathbb{R}^d.
\end{equation}
We take the $L^2_{loc}$-limit when $N \rightarrow \infty$ in \eqref{equationtronquee} and we obtain : 
$$ -\operatorname{div}(a_{per} \nabla u) = \operatorname{div}(f) \quad \text{in } \mathbb{R}^d.$$  
Therefore, $u$ is a solution to \eqref{eqper1} in $\mathcal{D}'(\mathbb{R}^d)$. \\

\textit{Step 2 : Proof of Estimate \eqref{estimper}}

Let $p$ be in $\mathcal{P}$, we want to split $u$ in two parts. For every $x \in V_p$, we write : 
\begin{align*}
u(x) & = \displaystyle \int_{W_p} \nabla_y G_{per}(x,y)f(y)dy + \displaystyle \int_{\mathbb{R}^d \setminus{W_p}} \nabla_y G_{per}(x,y)f(y)dy \\
& = I_{1,p}(x) + I_{2,p}(x).
\end{align*}
$I_{1,p}$ and $I_{2,p}$ are two distributions (they are in $L^2_{loc}(\mathbb{R}^d$)), so we can consider their gradients in a distribution sense. In addition, $I_{2,p}$ is a differentiable function on $V_p$ and 
$$\nabla I_{2,p}(x) = \displaystyle \int_{\mathbb{R}^d \setminus{W_p}} \nabla_x \nabla_y G_{per}(x,y) f(y) dy.$$
We start by establishing a bound for $\|\nabla I_{1,p}\|_{L^2(V_p)}$.
First, we use estimate \eqref{estimgreen1} for  $\nabla_y G_{per}$ and we obtain :
\begin{align*}
  \|I_{1,p}\|^2_{L^2(W_p)} &  \leq C \int_{W_p} \left(  \int_{W_p} \frac{1}{|x-y|^{d-1}} |f(y)| dy\right)^2 dx.
\end{align*}

We next apply the Cauchy-Schwarz inequality : 
\begin{align*}
  \|I_{1,p}\|^2_{L^2(W_p)} 
  & \leq C \int_{W_p} \left(  \int_{W_p} \frac{1}{|x-y|^{d-1}}  dy\right) \left(  \int_{W_p} \frac{1}{|x-y|^{d-1}} |f(y)|^2 dy\right) dx.
\end{align*}
Property \ref{ii)} implies that $W_p \subset Q_{C_1 2^{|p|}}(2^p)$. Therefore, for every $x \in W_p$ and $y \in W_p$, we have by a triangle inequality that $x-y \in Q_{C 2^{|p|+1}}$ and then : 
\begin{equation}
\label{majorationintWP}
\displaystyle \int_{W_p} \frac{1}{|x-y|^{d-1}}  dy \leq
\int_{Q_{C 2^{|p|+1}}} \frac{1}{|y|^{d-1}}  dy \leq C 2^{|p|}.
\end{equation}
Using \eqref{majorationintWP} and the Fubini theorem, we finally obtain : 
\begin{align}
\label{inequalityI1p}
  \|I_{1,p}\|^2_{L^2(W_p)} 
  & \leq C 2^{|p|}   \int_{W_p} |f(y)|^2 \int_{Q_{C_1 2^{|p|+1}}(2^p)}  \frac{1}{|x-y|^{d-1}}  dx dy \leq C 2^{2|p|} \|f\|^2_{L^2(W_p)}.
\end{align}
Lemma \ref{reultatexistenceL2} ensures that $I_{1,p}$ is a solution in $\mathcal{D}'(\mathbb{R}^d)$ to : 
\begin{equation}
    \label{equationI1p}
    -\operatorname{div}(a_{per}\nabla I_{1,p}) = \operatorname{div}(f 1_{W_p}).
\end{equation}
Since Property \ref{ii)} ensures $D(\partial V_p, W_p) \geq C2^{|p|}$, we can apply a classical inequality of elliptic regularity (see for instance \cite[Theorem 4.4 p.63]{MR3099262})  to equation \eqref{equationI1p} in order to establish the following estimate :
\begin{align}
\label{inequalitynablaI1p}
\|\nabla I_{1,p}\|_{L^2(V_p)}^2  
& \leq C\left(\dfrac{1}{2^{2\left|p\right|}} \|I_{1,p}\|^2_{L^2(W_p)}  + \|f\|^2_{L^2(W_p)}\right),
\end{align}
and we deduce from previous inequalities \eqref{inequalityI1p} and \eqref{inequalitynablaI1p} that : 
\begin{equation}
\label{estimatecontinuityI1p}
\|\nabla I_{1,p}\|^2_{L^2(V_p)} \leq C \|f\|^2_{L^2(W_p)}.  
\end{equation}
In addition, we have : 
$$ \|f\|^2_{L^2(W_p)} \leq \sum_{\substack{q \in \mathcal{P}, \\ V_q \cap W_p \neq \emptyset}} \|f\|^2_{L^2(V_q)} \leq \sum_{\substack{q \in \mathcal{P}, \\ V_q \cap W_p \neq \emptyset}} \sup_{r \in \mathcal{P}}\|f\|^2_{L^2(V_r)}.$$
Next, we use a triangle inequality and Property \ref{iv)} of $W_p$ to obtain : 
$$ \|f\|^2_{L^2(W_p)} \leq C \sup_{r \in \mathcal{P}}\|f\|^2_{L^2(V_r)}.$$
We apply this inequality in \eqref{estimatecontinuityI1p} and we finally obtain : 
\begin{equation}
\label{boundL2VPnablaI1p}
\|\nabla I_{1,p}\|_{L^2(V_p)} \leq C  \sup_{r \in \mathcal{P}}\|f\|_{L^2(V_r)},  
\end{equation}
where $C>0$ is independent of $p$.

We next prove a similar bound for $\displaystyle \| \nabla I_{2,p}\|_{L^2(V_p)}$. To start with, we want to show there exists a constant $C>0$ such that : 
\begin{equation}
\label{estimunifnablaI2p}
\| \nabla I_{2,p}\|_{L^{\infty}(V_p)} \leq C \frac{1}{2^{d|p|/2}} \sup_{r \in \mathcal{P}}\|f\|_{L^2(V_r)}. 
\end{equation}
To this aim, we fix $x \in V_p$ and we use estimate \eqref{estimgreen2} for $\nabla_x \nabla_y G_{per}$ to obtain : 
\begin{align*}
    \left|\nabla I_{2,p}(x)\right| & \leq C \sum_{q\neq p} \int_{V_q\setminus{W_p}} \frac{1}{|x-y|^d} |f(y)| dy \\
    & \leq C \sum_{q\neq p} \left(\int_{V_q\setminus{W_p}} \frac{1}{|x-y|^{2d}}  dy\right)^{1/2}\left(\int_{V_q\setminus{W_p} } |f(y)|^2 dy\right)^{1/2}\\
    & \leq C \sum_{q\neq p} \left(\int_{V_q\setminus{W_p}} \frac{1}{|x-y|^{2d}}  dy\right)^{1/2} \sup_{r \in \mathcal{P}}\|f\|_{L^2(V_r)}  
\end{align*}
Next, using Property \ref{ii)} of $W_p$, there exists C>0 such that for every $q \neq p$,
$$D(V_q\setminus{W_p},V_p)>C2^{|p|},$$
and it follows : 
\begin{align*}
 \sum_{|q|<|p|} \left(\int_{V_q\setminus{W_p}} \frac{1}{|x-y|^{2d}}  dy\right)^{1/2} & \leq C  \sum_{|q|<|p|} \left(\int_{V_q\setminus{W_p}} \frac{1}{2^{|p|2d}}  dy\right)^{1/2} \\
 &\leq C \sum_{|q|< |p|} \left( \frac{\left|V_q\right|}{2^{|p|2d}}  \right)^{1/2} \\
 &\leq C \sum_{|q|< |p|} \frac{2^{|q|d/2}}{2^{|p|d}}.
\end{align*}
The last inequality is actually a direct consequence of Propositions \ref{volumeVx} and \ref{norm2p}. In addition, we have proved in proposition \ref{numberpointAn} there exists a constant $C>0$ such that for every $n\in \mathbb{N}$, the number of $q\in \mathcal{P}$ such that $|q|=n$ is bounded by $C$. Therefore we have : 
$$\sum_{|q|< |p|} \frac{2^{|q|d/2}}{2^{|p|d}} = \sum_{n=0}^{|p|} \sum_{q\in \mathcal{P}, |q|=n} \frac{2^{|q|d/2}}{2^{|p|d}} \leq C \sum_{n=0}^{|p|}\frac{2^{nd/2}}{2^{|p|d}} = C \frac{2^{|p|d/2}}{2^{|p|d}} = C \frac{1}{2^{|p|d/2}}.$$
And finally : 
\begin{equation}
    \label{majorationfirstsum1}
\sum_{|q|< |p|} \left(\int_{V_q\setminus{W_p}} \frac{1}{|x-y|^{2d}}  dy\right)^{1/2} \leq C \frac{1}{2^{|p|d/2}}.
\end{equation}

Furthermore, we have with similar arguments : 
\begin{align*}
 \sum_{|q|\geq |p|} \left(\int_{V_q\setminus{W_p}} \frac{1}{|x-y|^{2d}}  dy\right)^{1/2} & \leq C  \sum_{|q|\geq |p|} \left(\int_{V_q\setminus{W_p}} \frac{1}{2^{|q|2d}} dy\right)^{1/2} \\
 &\leq C \sum_{|q|\geq |p|} \left( \frac{\left|V_q\right|}{2^{|q|2d}}\right)^{1/2} \\
 &\leq C \sum_{|q|\geq |p|} \frac{1}{2^{|q|d/2}}. 
\end{align*}
And we obtain again : 
$$\sum_{|q|\geq |p|} \frac{1}{2^{|q|d/2}}  \leq C  \sum_{n\geq |p|} \frac{1}{2^{nd/2}} = C \frac{1}{2^{|p|d/2}}.$$ 
That is : 
\begin{equation}
    \label{estimsecondsum1}
\sum_{|q|\geq |p|} \left(\int_{V_q\setminus{W_p} } \frac{1}{|x-y|^{2d}}  dy\right)^{1/2} \leq C \frac{1}{2^{|p|d/2}}.
\end{equation}
Using estimates \eqref{majorationfirstsum1} and \eqref{estimsecondsum1}, we have finally proved \eqref{estimunifnablaI2p} and it follows :
\begin{align*}
\|\nabla I_{2,p}\|_{L^2(V_p)} & \leq \left|V_p\right|^{1/2}\|\nabla I_{2,p}\|_{L^{\infty}(V_p)}\\
& \leq C 2^{|p|d/2}\frac{1}{2^{|p|d/2}} \sup_{r \in \mathcal{P}}\|f\|_{L^2(V_r)}.  
\end{align*}
Therefore we have the existence of a constant $C>0$ independent of $p$ such that : 
\begin{equation}
    \label{estimnormL2nablaI2p}
    \|\nabla I_{2,p}\|_{L^2(V_p)} \leq C \sup_{r \in \mathcal{P}}\|f\|_{L^2(V_r)}. 
\end{equation}
For every $p \in \mathcal{P}$, using estimates \eqref{boundL2VPnablaI1p} and \eqref{estimnormL2nablaI2p} and a triangle inequality, we conclude that : 
\begin{equation*}
  \|\nabla u\|_{L^2(V_p)} \leq \|\nabla I_{1,p}\|_{L^2(V_p)} + \|\nabla I_{2,p}\|_{L^2(V_p)} \leq C \sup_{r \in \mathcal{P}}\|f\|_{L^2(V_r)}.
\end{equation*}
We finally obtain expected estimate \eqref{estimper} taking the supremum over all $p \in \mathcal{P}$ in the above inequality. 
\end{proof}

To conclude the study of problem \eqref{eqper1} with a periodic coefficient, we next show that the solution to \eqref{defuintegrale} given in Lemma \ref{lemmeexistenceperiodique1} has a gradient in $\mathcal{B}^2(\mathbb{R}^d)$. 

\begin{lemme}
\label{existenceper}
Let $f\in \left(\mathcal{B}^2(\mathbb{R}^d)\right)^d$, then the function $u$ defined by \eqref{defuintegrale} is the unique solution to \eqref{eqper1} such that $\nabla u \in \left(\mathcal{B}^2(\mathbb{R}^d)\right)^d$. 
\end{lemme}

\begin{proof}

We want to prove there exists a function $g \in \left(L^2(\mathbb{R}^d)\right)^d$ such that
$$\displaystyle \lim_{|p| \rightarrow \infty} \|\nabla u - \tau_{-p} g\|_{L^2(V_p)} =0.$$ In this proof, the letter C  also denotes a generic constant independent of $p$, $u$ and $f$ that may change from one line to another.
Using the result of Lemma \ref{reultatexistenceL2}, we can define a function $u_{\infty} \in L^2_{loc}(\mathbb{R}^d)$ by : 
$$u_{\infty}(x) = \int_{\mathbb{R}^d} \nabla_y G_{per}(x,y)f_{\infty}(y) dy$$ 
solution in $\mathcal{D}'(\mathbb{R}^d)$ to : 
\begin{equation}
\label{eql2}
-\operatorname{div}(a_{per}\nabla u_{\infty}) = \operatorname{div}(f_{\infty}) \quad \text{in } \mathbb{R}^d, 
\end{equation}
such that $\nabla u_{\infty} \in \left(L^2(\mathbb{R}^d)\right)^d$.
For every $p\in \mathcal{P}$, by subtracting a $2^p$-translation of \eqref{eql2} from \eqref{eqper1}, the periodicity of $a_{per}$ implies : 
$$-\operatorname{div}\left(a_{per}\nabla\left(u-\tau_{-p}u_{\infty}\right)\right)=\operatorname{div}\left(f-\tau_{-p}f_{\infty}\right).$$
For every $p \in \mathcal{P}$, in the sequel we denote $u_p=u-\tau_{-p}u_{\infty}$ and $f_p=f-\tau_{-p}f_{\infty}$. In order to prove $\nabla u \in \left(\mathcal{B}^2(\mathbb{R}^d)\right)^d$, the idea is to show that $ \displaystyle \lim_{|p| \rightarrow \infty}  \int_{\textit{V}_p} |\nabla u_p|^2 dx = 0$. We start by fixing $\varepsilon>0$. Since $f\in \left(\mathcal{B}^2(\mathbb{R}^d)\right)^d$, Proposition \ref{uniformmajoration} gives the existence of a radius $R>0$, such that for every $p,q \in \mathcal{P}$, 
\begin{equation}
\label{majoration_uniform_B2}
    \left( \int_{V_q\cap B_R(2^q)^c} \left|f - \tau_{-p}f_{\infty}\right|^2 \right)^{1/2} < \varepsilon.
\end{equation}
In the sequel, the idea is to repeat step by step the method used in the proof of Lemma \ref{lemmeexistenceperiodique1}. For $p\in \mathcal{P}$, we thus introduce the set $W_p$ as in the previous proof and we split $u_p$ in two parts. For every $x \in V_p$, we can write : 
\begin{align*}
u_p(x) & = \displaystyle \int_{W_p} \nabla_y G_{per}(x,y)f_p(y)dy + \displaystyle \int_{\mathbb{R}^d \setminus{W_p}} \nabla_y G_{per}(x,y)f_p(y)dy \\
& = I_{1,p}(x) + I_{2,p}(x).
\end{align*}

In the sequel, we denote $A_p$ the set $W_p  \setminus V_p$. As in the previous proof (see the details of the proof of estimate \eqref{estimatecontinuityI1p}) we can show that  : 
$$\|\nabla I_{1,p}\|^2_{L^2(V_p)} \leq C \|f_p\|^2_{L^2(W_p)},$$
and we next prove that $\displaystyle \lim_{|p| \rightarrow \infty} \|f_p\|^2_{L^2(W_p)} =0$. First, since $f\in \left(\mathcal{B}(\mathbb{R}^d)\right)^d$, we already know that $\displaystyle \lim_{|p| \rightarrow \infty} \int_{\textit{V}_{p}} |f_p|^2 = 0$  and we only have to treat the integration term on $A_p$. Using Property \ref{iii)} of $W_p$, we know that the distance $D(2^q,W_p)$, for $q\neq p$, is bounded from below by $2^{|p|}$. Therefore, if $2^{|p|} > R$, we obtains  :
$$ A_p = \bigcup_{\substack{q \in \mathcal{P} \setminus \{p\}\\ V_q \cap W_p \neq \emptyset}} V_q \cap W_p \subset \bigcup_{\substack{q \in \mathcal{P} \setminus \{p\}\\ V_q \cap W_p \neq \emptyset}} \textit{V}_q\cap B_R(2^q)^c.$$
In addition, Property \ref{iv)} of $W_p$ gives the existence of a constant $C>0$ such that the cardinality of the set of $q$ satisfying $V_q\cap W_p \neq \emptyset$ is bounded by $C$. Estimate \eqref{majoration_uniform_B2} therefore implies that  
$$\int_{A_{p}} |f_p|^2 \leq \sum_{{\substack{q \in \mathcal{P} \setminus \{p\}\\ V_q \cap W_p \neq \emptyset}}} \int_{ \textit{V}_q \cap B_R(2^q)^c} |f_p|^2 \leq C \varepsilon.$$
Since $\varepsilon$ can be chosen arbitrarily small, we finally obtain $\displaystyle\lim_{|p| \rightarrow \infty} \int_{A_{p}} |f_p|^2 = 0$, that is
\begin{equation}
\displaystyle \lim_{|p| \rightarrow \infty} \|\nabla I_{1,p}\|^2_{L^2(V_p)} =0.
\end{equation}
\\ 
We next prove that $ \displaystyle \lim_{|p| \rightarrow \infty} \| \nabla I_{2,p}\|^2_{L^2(V_p)} = 0$. We split $\nabla I_{2,p}$ in two parts such that for every $x \in V_p$ : 
\begin{align*}
\nabla I_{2,p}(x) & = \sum_{\substack{q\in \mathcal{P} \\ q\neq p}} \int_{\left(V_q\setminus{W_p}\right) \cap B_R(2^q)^c} \nabla_x \nabla_y G_{per}(x,y) f_p(y) dy \\
& + \displaystyle \sum_{\substack{q\in \mathcal{P} \\ q\neq p}} \int_{\left(V_q\setminus{W_p}\right) \cap B_R(2^q)} \nabla_x \nabla_y G_{per}(x,y) f_p(y) dy \\
&= J_{1,p}(x) + J_{2,p}(x).
\end{align*}

We want to estimate $\| J_{1,p}\|_{L^2(V_p)}$ and $\| J_{1,p}\|_{L^2(V_p)}$. We proceed exactly in the same way as in the previous proof (see the details of estimate \eqref{estimunifnablaI2p}) and, using estimate \eqref{majoration_uniform_B2}, we obtain the following inequalities : 
\begin{equation}
\label{estimunifJ1}
    \| J_{1,p}\|_{L^{\infty}(V_p)} \leq C \frac{1}{2^{d|p|/2}} \sup_{q \in \mathcal{P}} \|f_p\|_{L^2\left(V_q \cap B_R(2^q)^c\right)}  \leq C \frac{1}{2^{d|p|/2}} \varepsilon,
\end{equation}
and
\begin{equation}
\label{estimunifJ2}
\|J_{2,p}\|_{L^{\infty}(V_p)} \leq C  R^d \frac{|p|}{2^{|p|d}} \sup_{q \in \mathcal{P}} \|f_p \|_{L^2(V_q)}.
\end{equation}
To conclude, we consider $P>0$ such that for every $p \in \mathcal{P}$ satisfying $|p|>P$, we have : 
$$R^d\frac{|p|}{2^{|p|d/2}}< \varepsilon.$$
Therefore, for every $|p|>P$, we use \eqref{estimunifJ1} and \eqref{estimunifJ2} and we obtain : 
\begin{align*}
\|\nabla I_{2,p}\|_{L^2(V_p)} &\leq \|J_{1,p}\|_{L^2(V_p)} + \|J_{2,p}\|_{L^2(V_p)} \\
& \leq \left|V_p\right|^{1/2}\left( \|J_{1,p}\|_{L^{\infty}(V_p)} + \|J_{2,p}\|_{L^{\infty}(V_p)} \right) \\
& \leq C 2^{|p|d/2}\left( \frac{1}{2^{|p|d/2}}\varepsilon +  R^d\frac{|p|}{2^{|p|d}} \right) \leq C \varepsilon.
\end{align*}
Since we can choose $\varepsilon$ arbitrarily small, we conclude that  $ \displaystyle \lim_{|p| \rightarrow \infty} \|\nabla I_{2,p}\|_{L^2(V_p)} = 0$. Finally, by a triangle inequality we have $ \displaystyle \lim_{|p| \rightarrow \infty} \|\nabla u_{p}\|_{L^2(V_p)} = 0$, that is $\nabla u \in \left(\mathcal{B}^2(\mathbb{R}^d)\right)^d$.
\end{proof}

\begin{remark}
It is important to note that the essential point of the two above proofs is the convergence of the series of the form $\displaystyle \sum_{q \in \mathcal{P}} \int_{V_q} \dfrac{1}{|x-y|^d} f(y) dy$ given in estimates \eqref{formuleimportante1}, \eqref{majorationfirstsum1} and \eqref{estimsecondsum1}. These convergence results are not specific to the set \eqref{defG} considered in this article and are actually ensured by Assumptions \eqref{H1}, \eqref{H2} and \eqref{H3}, particularly by the logarithmic bound given in Corollary \ref{corollogboundVP}. Therefore, the results of existence still hold if we consider a general set $\mathcal{G}$ satisfying these assumptions.  
\end{remark}

\begin{remark}
\label{remarkcorrectordim2}
In the two-dimensional context, the results of Lemmas \ref{reultatexistenceL2}, \ref{lemmeexistenceperiodique1} and \ref{existenceper} remain true since estimates \eqref{estimgreen1}, \eqref{estimgreen3} and \eqref{estimgreen1} still hold. However the proof requires some additional technicalities, in particular to prove that the function $u$ defined by \eqref{defuintegrale} makes sens. In this case the series \eqref{serieimportante} does not actually converges but it is still possible to prove that the series of the gradients \eqref{SeriesSN} converges. Here, the difficulty is to show that the limit of \eqref{SeriesSN}, denoted by $T$ here, is the gradient in a distribution sense of a solution to \eqref{eqper1}. To this end, it is actually sufficient to show that $\partial_i T_j = \partial_j T_i$ for every $i,j \in \{1,...,d\}$. This result is obtained considering the property of the limit of $\eqref{SeriesSN}$ in $\mathcal{D}(\mathbb{R}^d)$.
\end{remark}

\subsection{Existence results in the general problem}

Our aim is now to generalize the results established in the case of periodic coefficients to our original problem \eqref{equationref}. Here, our approach is to prove in Lemma \ref{lemme2} the continuity of the linear operator $\nabla\left(-\operatorname{div}a\nabla\right)^{-1}\operatorname{div}$ from $\left(\mathcal{B}^2(\mathbb{R}^d)\right)^d$ to $\left(\mathcal{B}^2(\mathbb{R}^d)\right)^d$ in order to apply a method adapted from \cite{blanc2018correctors} and based on the connexity of the set $[0,1]$. This method is used in the proof of existence of Lemma \ref{lemmexistence}. Finally, this result allows us to prove the existence of a corrector stated in Theorem \ref{theoreme1}.

Actually, we could have proved Lemmas \ref{lemme2} and \ref{lemmexistence} simultaneously but, in the interest of clarity, we first prove a priori estimate \eqref{aprioriestimate} and next, we  establish the existence result in the general case.

\begin{lemme}[A priori estimate]
\label{lemme2}
There exists a constant $C>0$ such that for every $f$ in $\left(\mathcal{B}^2(\mathbb{R}^d)\right)^d$ and $u$ solution in $\mathcal{D}'(\mathbb{R}^d)$ to \eqref{equationref} 
with $\nabla u$ in~$\left(\mathcal{B}^2(\mathbb{R}^d)\right)^d$, we have the following estimate :
\begin{equation}
\label{aprioriestimate}
 \|\nabla u\|_{\mathcal{B}^2(\mathbb{R}^d)}  \leq C  \|f\|_{\mathcal{B}^2(\mathbb{R}^d)}. 
\end{equation}
 
\end{lemme}

\begin{proof}

We give here a proof by contradiction using a compactness-concentration method. We assume that there exists a sequence $f_n$ in $\left(\mathcal{B}^2(\mathbb{R}^d)\right)^d$ and an associated sequence of solutions $u_n$ such that $\nabla u_n$ is in $\left(\mathcal{B}^2(\mathbb{R}^d)\right)^d$ and : 
\begin{equation}
\label{9}
-\operatorname{div}((a_{\textit{per}}+\Tilde{a})\nabla u_n) = \operatorname{div}(f_n),
\end{equation}
\begin{equation}
\label{10}
\lim_{n \rightarrow \infty}\|f_n\|_{\mathcal{B}^2(\mathbb{R}^d)} =0,  
\end{equation}
\begin{equation}
\label{11}
\forall n \in \mathbb{N} \quad \|\nabla u_n\|_{\mathcal{B}^2(\mathbb{R}^d)} = 1.   
\end{equation}
 First of all, a property of the supremum bound ensures that for every $n\in \mathbb{N}$, there exists $x_n \in \mathbb{R}^d $ such that :
$$ \| \nabla u_n\|_{L^2_{\textit{unif}}} \geq \| \nabla u_n \|_{L^2(B_1(x_n))} \geq \| \nabla u_n\|_{L^2_{\textit{unif}}} - \dfrac{1}{n}.$$
Next, in the spirit of the method of concentration-compactness \cite{AIHPC_1984__1_4_223_0}, we denote $\Bar{u}_n = \tau_{x_n}u_n$, $\Bar{f}_n = \tau_{x_n}f_n$, $\Bar{a}_n = \tau_{x_n}a$ and $\Bar{\Tilde{a}}_n = \tau_{x_n}\Tilde{a}$ and we have for every $n \in \mathbb{N}$ : 
\begin{equation}
\label{limunif}
\| \nabla u_n\|_{L^2_{\textit{unif}}} \geq \| \nabla \Bar{u}_n \|_{L^2(B_1)} \geq \| \nabla u_n\|_{L^2_{\textit{unif}}} - \dfrac{1}{n}.
\end{equation}

Next, for every $n\in \mathbb{N}$,
$\Bar{u}_n$ is a solution to : 
\begin{equation*}
-\operatorname{div}(\Bar{a}_n\nabla \Bar{u}_n) = \operatorname{div}(\Bar{f}_n) \quad \text{in } \mathbb{R}^d.
\end{equation*}
Since the norm of $L^2_{\textit{unif}}$ is invariant by translation, \eqref{10} and \eqref{11} ensure that $\Bar{f}_n$ strongly converges to 0 in $L^2_{unif}(\mathbb{R}^d)$ and that the sequence $\displaystyle (\nabla \Bar{u}_n)_{n\in \mathbb{N}}$ is bounded in $L^2_{unif}(\mathbb{R}^d)$. Therefore, up to an extraction, $\nabla\Bar{u}_n$ weakly converges to a function $\nabla \Bar{u}$ in $L^2_{loc}(\mathbb{R}^d)$. \\
The idea is now to study the limit of $\Bar{a}_n$. To start with, we denote $\mathbf{x}_n = \left(x_{n,i} \operatorname{mod}(1)\right)_{i \in \{1,...,d\}}$. Since $a_{per}$ is periodic, we have $\tau_{x_n}a_{per} =~\tau_{\mathbf{x}_n}a_{per}$. In addition, the sequence $\mathbf{x}_n$ belongs to the unit cube of $\mathbb{R}^d$ and, therefore, it converges (up to an extraction) to $\mathbf{x} \in \mathbb{R}^d$. Since $a_{per}$ is Holder continuous,  $\tau_{\mathbf{x}_n}a_{per}$ converges uniformly to $\tau_{\mathbf{x}}a_{per}$, which also belongs to $\left(L^2_{per}(\mathbb{R}^d) \cap \mathcal{C}^{0,\alpha}(\mathbb{R}^d)\right)^{d\times d}$. \\
In order to study the convergence of $\Bar{\Tilde{a}}_n$, we consider several cases depending on $x_n$:
\begin{itemize}
\item[1.] If $x_n$ is bounded, it converges (up to an extraction) to $x_{lim}\in \mathbb{R}^d$. Then, since $\Tilde{a}$ is Holder-continuous, $\Bar{\Tilde{a}}_n$ strongly converges in $L^2_{loc}(\mathbb{R}^d)$ to $\tau_{x_{lim}}\Tilde{a} \in \left(\mathcal{B}^2(\mathbb{R}^d)\right)^{d\times d}$.
\item[2.] If $x_n$ is not bounded, since $\left(V_p\right)_{p\in \mathcal{P}}$ is a partition of $\mathbb{R}^d$, there exists an unbounded sequence $(p_n)_{n}$ in $\mathcal{P}$ such that $x_n = 2^{p_n} + t_n$ with $t_n \in \textit{V}_{p_n}-2^{p_n}$.
\begin{itemize}
    \item If $t_n$ is bounded, it converges (up to an extraction) to $t_{lim} \in \mathbb{R}^d$. In this case, for any compact subset $K$ of $\mathbb{R}^d$, we have
    \begin{align*}
 \|\Bar{\Tilde{a}}_n - \Tilde{a}_{\infty}(.+t_{lim})\|_{L^2(K)} & \leq    \|\Tilde{a}(.+2^{p_n} + t_n) - \Tilde{a}_{\infty}(.+t_n)\|_{L^2(K)} \\
 & + \|\Tilde{a}_{\infty}(.+t_n)- \Tilde{a}_{\infty}(.+t_{lim})\|_{L^2(K)} \\
 & =  \|\Tilde{a} - \tau_{-p_n}\Tilde{a}_{\infty}\|_{L^2(K+2^{p_n}+t_n)}\\
 & + \|\Tilde{a}_{\infty}(.+t_n)- \Tilde{a}_{\infty}(.+t_{lim})\|_{L^2(K)}.
\end{align*}
First, since $t_n$ is bounded and $p_n$ is unbounded, we have $K+2^{p_n}+t_n \subset V_{p_n}$ for $n$ sufficiently large. Therefore, $\displaystyle \|\Tilde{a} - \tau_{-p_n}\Tilde{a}_{\infty}\|_{L^2(K+2^{p_n}+t_n)}$ converges to 0 when $n \to \infty$. Second, $\Tilde{a}_{\infty}$ is Holder-continuous and $t_n$ converges to $t_{lim}$. Thus, $\Tilde{a}_{\infty}(.+t_n)$ converges uniformly to $\Tilde{a}_{\infty}(.+t_{lim})$ and $\|\Tilde{a}_{\infty}(.+t_n)- \Tilde{a}_{\infty}(.+t_{lim})\|_{L^2(K)}$ converges to 0. Finally, $\Bar{\Tilde{a}}_n$ converges to  $\Tilde{a}_{\infty}(.+t_{lim})$ in $L^2(K)$ for every compact subset $K$. \\

\item  If $t_n$ is unbounded, we can always assume that $|t_n| \to \infty$ up to an extraction. We have for every $K$ compact of $\mathbb{R}^d$, 
\begin{align*}
 \|\Bar{\Tilde{a}}_n\|_{L^2(K)} & \leq    \|\Tilde{a}(.+2^{p_n} + t_n) - \Tilde{a}_{\infty}(.+t_n)\|_{L^2(K)} + \|\Tilde{a}_{\infty}(.+t_n)\|_{L^2(K)} \\
 & =  \|\Tilde{a} - \tau_{-p_n}\Tilde{a}_{\infty}\|_{L^2(K+2^{p_n}+t_n)} + \|\Tilde{a}_{\infty}\|_{L^2(K+t_n)}.
\end{align*}
First, since $\Tilde{a}_{\infty} \in \left(L^2(\mathbb{R}^d)\right)^{d\times d}$ and $t_n$ is unbounded we have that $\|\Tilde{a}_{\infty}\|_{L^2(K+t_n)}$ converges to 0 when $n \to \infty$. Secondly, we introduce the set $W_{2^{p_n}}$ defined as in Proposition \ref{openconstruction}. For every $R>0$, the properties of $W_{2^{p_n}}$ allow to show that there exists $N\in \mathbb{N}$ such that for all $n>N$, we have $K + 2^{p_n} +t_n \subset W_{2^{p_n}}$ and : 
$$K+2^{p_n}+t_n \subset \displaystyle \bigcup_{\substack{q \in \mathcal{P} \\ V_q \cap W_{2^{p_n}} \neq \emptyset}}V_{q}\setminus{B_R(2^q)}.$$ 
Using Proposition \ref{openconstruction}, we know that the number of $q$ such that  $V_q \cap W_{2^{p_n}} \neq \emptyset$ is uniformly bounded with respect to $n$ and Proposition \ref{uniformmajoration} finally ensures that $\|\Tilde{a} - \tau_{-p_n}\Tilde{a}_{\infty}\|_{L^2(K+2^{p_n}+t_n)} \rightarrow 0$. Therefore, $\Bar{\Tilde{a}}_n$ strongly converges to 0 in $L^2_{loc}(\mathbb{R}^d)$. 
\end{itemize}
\end{itemize}
In any case, the sequence $a_{per}+\Bar{\Tilde{a}}_n$ therefore converges to a coefficient $A = \tau_{\mathbf{x}}a_{per} + \Tilde{A}$, where $\Tilde{A}$ is of the form
$$\Tilde{A} = \left\{ \begin{array}{cc}
    \tau_{x_{lim}} \Tilde{a} \in \left(\mathcal{B}^2(\mathbb{R}^d)\right)^{d\times d} & \text{ if } x_n \text{ is bounded}, \\
    \tau_{t_{lim}} \Tilde{a}_{\infty} \in \left(L^{2}(\mathbb{R}^d)\right)^{d \times d} & \text{ if } x_n=2^{p_n} + t_n \text{ where } p_n \text{ is not bounded and } t_n \text{ is bounded},\\
    0 & \text{ if } x_n=2^{p_n} + t_n \text{ where } p_n, t_n \text{ are not bounded}.
\end{array}
\right.$$
In the three cases, as a consequence of Assumptions \eqref{hypothèses1} and \eqref{hypothèses2}, the coefficient $A$ is clearly bounded, elliptic and belongs to $\left(\mathcal{C}^{0,\alpha}(\mathbb{R}^d)\right)^{d \times d}$. Moreover, as a consequence of the uniform Holder-continuity (with respect to $n$) of $\Bar{a}_n - A$, the convergence of $\Bar{a}_n$ to $A$ is also valid in $L^{\infty}_{loc}(\mathbb{R}^d)$.  

The next step of the proof is to study the limit $\nabla \Bar{u}$ of $\nabla \Bar{u}_n$ in these three cases. First, since $\Bar{a}_n$ strongly converges to $A$ in $L^{2}_{loc}(\mathbb{R}^d)$, considering the weak limit in \eqref{9} when $n\to \infty$, we obtain 
\begin{equation}
\label{12}
-\operatorname{div}(A\nabla \Bar{u}) =  0 \quad \text{in } \mathbb{R}^d.
\end{equation}
We now state that $\nabla \Bar{u} =0$. Indeed, 
\begin{itemize}
    \item[1.] if $x_n$ is bounded, assumption \eqref{11} ensures that there exists a constant $C>0$ such that for all $n \in \mathbb{N}$ and $p \in \mathcal{P}$, we have  : 
\begin{equation}
\label{bound_unbar_Vp}
  \|\nabla \Bar{u}_n\|_{L^2(\textit{V}_{p})} = \|\nabla u_n\|_{L^2(\textit{V}_{p} + x_n)} \leq C.  
\end{equation} 
Therefore, the property of lower semi-continuity satisfied by the norm $\|.\|_{L^2}$ implies  
$$\forall p \in \mathcal{P}, \quad \| \nabla \Bar{u} \|_{L^2(\textit{V}_p)} \leq \liminf_{n \rightarrow \infty} \| \nabla \Bar{u}_n \|_{L^2(\textit{V}_p)} < C.$$ 
And we obtain $\displaystyle \sup_p \| \nabla \Bar{u} \|_{L^2(\textit{V}_p)} < \infty$. Finally, since $A$ is elliptic and bounded and  $\Bar{u}$ is solution to \eqref{12}, the uniqueness results of Lemma \ref{lemme1} gives $\nabla \Bar{u} = 0$ on $\mathbb{R}^d$. \\
\item[2.] if $x_n$ is not bounded, we know that $x_n = 2^{p_n} + t_n$ where $|p_n| \to \infty$. For every $ n \in \mathbb{N}$  : 
$$\|\nabla \Bar{u}_n\|_{L^2(\textit{V}_{p_n} - x_n)} = \|\nabla u_n\|_{L^2(\textit{V}_{p_n})} \leq 1.$$
Up to an extraction, the sequence $\textit{V}_{p_n} - x_n$ is an increasing sequence of sets, and we can show that $\displaystyle \bigcup _{n\in \mathbb{N}} \left(\textit{V}_{p_n} - x_n\right) = \mathbb{R}^d$ (see the proof of Proposition \ref{tailledesVP}). Consequently,
for every $R>0$, there exists $N\in \mathbb{N}$ such that $B_R \subset \left(V_{p_N}-x_N\right)$ and  
$$\forall n>N, \quad \|\nabla \Bar{u}_n\|_{L^2(B_R)} \leq 1.$$ 
Using again lower semi-continuity, we have for every $R>0$ : 
$$\|\nabla \Bar{u}\|_{L^2(B_R)} \leq \liminf_{n \rightarrow \infty} \|\nabla \Bar{u}_n\|_{L^2(B_R)} \leq 1.$$ 
We obtain that $\nabla \Bar{u}\in L^2(\mathbb{R}^d)$. Since $A$ is bounded and elliptic, a result of uniqueness established in \cite[Lemma 1]{blanc2012possible} finally ensures that $\nabla \Bar{u} =0$.
\end{itemize}
We are now able to show that $\nabla \Bar{u}_n$ strongly converges to 0 in $L^2(B_1)$. To this aim, we note that, for every $n$, the addition of a constant to $\Bar{u}_n$ does not affect $\nabla \Bar{u}_n$. Then, without loss of generality, we can always assume that $\displaystyle \int_{B_2} \Bar{u}_n$ = 0 and the Poincar\'e-Wirtinger inequality gives the existence of a constant $C>0$ independent of $n$ such that : 
$$\|\Bar{u}_n\|_{L^2(B_2)} \leq C \|\nabla \Bar{u}_n\|_{L^2(B_2)}.$$
$\Bar{u}_n$ is therefore bounded in $H^1(B_2)$ according to Assumption \eqref{11}. The Rellich theorem ensures that, up to an extraction, $\Bar{u}_n$ strongly converges to $\Bar{u}$, that is to 0, in $L^2(B_2)$. Since $\Bar{u}_n$ is solution to \eqref{9}, a classical inequality of elliptic regularity gives the following estimate (see for instance \cite[Theorem 4.4 p.63]{MR3099262}) : 
$$\int_{B_1} |\nabla \Bar{u}_n|^2 \leq C\left(\int_{B_{2}} |\Bar{u}_n|^2 + \int_{B_{2}} |\Bar{f}_n|^2\right),$$
where $C$ depends only of $a$ and the ambient dimension $d$. We therefore consider the limit when $n\rightarrow \infty$ to conclude that $\nabla \Bar{u}_n$ strongly converges to $0$ in $L^2(B_1)$. We next use \eqref{limunif} and the strong convergence of $\nabla \Bar{u}_n$ to 0 in $L^2(B_1)$ to conclude that 
$$\lim_{n \rightarrow \infty}\| \nabla u_n\|_{L^2_{\textit{unif}}(\mathbb{R}^d)}=0.$$
That is, $\nabla u_n$ strongly converges to 0 in $L^2_{unif}(\mathbb{R}^d)$. \\

In order to conclude this proof, we will show that $\nabla u_n$ actually converges to 0 in $\mathcal{B}^2(\mathbb{R}^d)$ and obtain a contradiction. 

First of all, we study the behavior of the sequence $\nabla u_{n,\infty}$. For $p \in \mathcal{P}$, we consider the $2^p$-translation of \eqref{9} and we have 
$$- \operatorname{div}((a_{per} + \tau_p \Tilde{a}) \tau_p \nabla u_n) = \operatorname{div}(\tau_p f_n).$$
Letting $|p|$ go to the infinity, for every $n \in \mathbb{N}$, we obtain that $\nabla u_{n,\infty}$ is a solution to : 
$$ -\operatorname{div}((a_{\textit{per}}+\Tilde{a}_{\infty})\nabla u_{n,\infty}) = \operatorname{div}(f_{n,\infty}) \quad \text{in } \mathbb{R}^d.$$
An estimate established in \cite[Proposition 2.1]{blanc2018correctors}, gives the existence of a constant $C>0$ independent of $n$ such that : 
$$\|\nabla u_{n,\infty}\|_{L^2(\mathbb{R}^d)} \leq C \| f_{n,\infty}\|_{L^2(\mathbb{R}^d)}.$$
By assumption, we have $\displaystyle \lim_{n \rightarrow \infty}\| f_{n,\infty}\|_{L^2(\mathbb{R}^d)} = 0$ and we deduce that $\nabla u_{n,\infty}$ strongly converges to 0 in $L^2(\mathbb{R}^d)$, that is : 
$$ \lim_{n \rightarrow \infty} \|\nabla u_{n,\infty}\|_{L^2(\mathbb{R}^d)} = 0.$$

The last step is to establish that : 
$$\lim_{n \rightarrow \infty} \sup_p \| \nabla u_n\|_{L^2(\textit{V}_p)} = 0.$$
Let $\varepsilon>0$. Since $\Tilde{a}$ belongs to $\left(\mathcal{B}^2(\mathbb{R}^d)\right)^{d\times d}$ and is uniformly continuous, a direct consequence of Proposition \ref{uniformmajoration} gives the existence of $R>0$ such that : 
$$\forall q \in \mathcal{P}, \quad \|\Tilde{a}\|_{L^{\infty}(\textit{V}_q \cap B_R(2^q)^c)}< \dfrac{\varepsilon}{2}.$$
In addition, since $\nabla u_n$ strongly converges to 0 in $L^2_{unif}$, there exists $N\in \mathbb{N}$ such that :
$$\forall n>N, \quad \| \nabla u_n\|_{L^2_{\textit{unif}}(\mathbb{R}^d)} < \frac{\varepsilon}{2|B_R| \|\Tilde{a}\|_{L^{\infty}(\mathbb{R}^d)}}.$$
Using the last two inequalities, we obtain for every $q \in \mathcal{P}$ : 
\begin{align*}
\int_{\textit{V}_q} |\Tilde{a}(x)\nabla u_n(x)|^2 dx  & \leq \int_{\textit{V}_q \cap B_R(2^q)^c} |\Tilde{a}(x)\nabla u_n(x)|^2 dx + \int_{\textit{V}_q\cap B_R(2^q)} |\Tilde{a}(x)\nabla u_n(x)|^2 dx \\
& \leq \|\Tilde{a}\|_{L^{\infty}(\textit{V}_q \cap B_R(2^q)^c)}  \int_{\textit{V}_q \cap B_R(2^q)^c} |\nabla u_n(x)|^2 dx \\
& + \|\Tilde{a}\|_{L^{\infty}(\mathbb{R}^d)}  \int_{\textit{V}_p\cap B_R(2^q)} |\nabla u_n(x)|^2 dx \\
& \leq  \|\Tilde{a}\|_{L^{\infty}(\textit{V}_q \cap B_R(2^q)^c)} \sup_p (\| \nabla u_n \|_{L^2(\textit{V}_p)} ) + \|\Tilde{a}\|_{L^{\infty}(\mathbb{R}^d)}|B_R|\| \nabla u_n\|_{L^2_{\textit{unif}}(\mathbb{R}^d)}  \\
& \leq \dfrac{\varepsilon}{2} + \dfrac{\varepsilon}{2} = \varepsilon.
\end{align*}
Therefore : 
$$\lim_{n \rightarrow \infty} \sup_{p} \int_{\textit{V}_p} |\Tilde{a}(x)\nabla u_n(x)|^2 dx = 0.$$

We next consider equation \eqref{9} and we use Lemma \ref{lemme1} to ensure that, up to the addition of a constant, $u_n$ is the unique solution to  : 
$$- \operatorname{div}(a_{per} \nabla u_n) = \operatorname{div}(f_n + \Tilde{a}\nabla u_n) \quad \text{in } \mathbb{R}^d.$$
such that $\displaystyle \sup_p \|\nabla u_n\|_{ L^2(\textit{V}_p)} < \infty$. Then, Estimate \eqref{estimper} established in Lemma \ref{existenceper} gives the existence of a constant $C>0$ independent of $n$ such that  : 
$$\sup_p \|\nabla u_n\|_{ L^2(\textit{V}_p)}  \leq C\left(\sup_p \|f_n\| _{ L^2(\textit{V}_p)} + \sup_{p}\| \Tilde{a}\nabla u_n\|_{ L^2(\textit{V}_p)} \right).$$
Letting $n$ go to the infinity, we deduce that $ \displaystyle \lim_{n \rightarrow \infty} \sup_p \|\nabla u_n\|_{ L^2(\textit{V}_p)} = 0$. We can finally conclude that 
$$\displaystyle \lim_{n \rightarrow \infty}\|  \nabla u_n\|_{\mathcal{B}^2(\mathbb{R}^d)}=0,$$ 
and, since $\nabla u_n$  satisfies \eqref{11}, we have a contradiction.

\end{proof}

\begin{lemme}
\label{lemmexistence}
Let $f \in \left(\mathcal{B}^2(\mathbb{R}^d)\right)^d$, there exists $u\in H^1_{loc}(\mathbb{R}^d)$ solution to \eqref{equationref} such that $\nabla u \in~\left(\mathcal{B}^2(\mathbb{R}^d)\right)^d$. 
\end{lemme}

\begin{proof}
First of all, we remark that it is sufficient to prove this existence result when $f \in \mathcal{B}^2(\mathbb{R}^d)\cap \mathcal{C}^{0,\alpha}(\mathbb{R}^d)$. Indeed, if we denote $\Phi = \nabla\left(-\operatorname{div}a\nabla\right)^{-1}\operatorname{div}$ the reciprocal linear operator from $\left(\mathcal{B}^2(\mathbb{R}^d)\cap \mathcal{C}^{0,\alpha}(\mathbb{R}^d)\right)^d$ to $\left(\mathcal{B}^2(\mathbb{R}^d)\right)^d$ associated with equation \eqref{equationref} and we assume that $\Phi$ is well defined, Lemma \ref{lemme2} ensures it is continuous with respect to the norm of $\mathcal{B}^2(\mathbb{R}^d)$. Then, we are able to conclude in the general case using the density result stated in Property \ref{density}. In the sequel of this proof, we therefore assume that $f$ belongs to $\left(\mathcal{C}^{0,\alpha}(\mathbb{R}^d)\right)^d$.

To start with, we show a preliminary result of regularity satisfied by the solutions to \eqref{equationref}.  Assuming $f\in \left(\mathcal{B}^2(\mathbb{R}^d)\cap \mathcal{C}^{0,\alpha}(\mathbb{R}^d)\right)^d$, we want to prove that a solution $u$ to \eqref{equationref} such that $\nabla u \in \left(\mathcal{B}^2(\mathbb{R}^d)\right)^d$ also satisfies $\nabla u \in  \left(\mathcal{C}^{0,\alpha}(\mathbb{R}^d)\right)^d$. Indeed, if $u$ is such a solution to \eqref{equationref}, a consequence of a regularity result established in \cite[Theorem 5.19 p.87]{MR3099262} (see also \cite[Theorem 3.2 p.88]{giaquinta1983multiple}) gives the existence of $C>0$ such that for all $x \in \mathbb{R}^d$ : 
\begin{equation}
\label{estcontinuite2}
\|\nabla u\|_{\mathcal{C}^{0,\alpha}(B_{1}(x))} \leq C \left(\| \nabla u\|_{L^2_{unif}(\mathbb{R}^d)} + \|f\|_{\mathcal{C}^{0,\alpha}(\mathbb{R}^d)}\right).
\end{equation} 
Therefore, $\nabla u $ belongs to $\left(\mathcal{C}^{0,\alpha}(\mathbb{R}^d)\cap \mathcal{B}^2(\mathbb{R}^d)\right)^d$.


In the sequel of the proof, we use an argument of connexity adapted from \cite{blanc2018correctors}. Let $\mathbf{P}(a)$ the following assertion : "There exists a solution $u\in \mathcal{D}'(\mathbb{R}^d)$ to :
\begin{equation}
\label{eqconnectedness}
    -\operatorname{div}\left(a\nabla u\right) = \operatorname{div}(f) \quad \text{in } \mathbb{R}^d
\end{equation}
such that $\nabla u \in \left(\mathcal{B}^2(\mathbb{R}^d)\cap \mathcal{C}^{0,\alpha}(\mathbb{R}^d)\right)^d$." \\
For $t\in  [0,1]$, we denote $a_t = a_{per} + t\Tilde{a}$ and we define the following set : 
\begin{equation}
\mathcal{I} = \left\{t\in [0,1] \ \middle| \ \forall s \in [0,t], \text{$\mathbf{P}(a_s)$ is true}\right\}.
\end{equation}
Our aim is to show that $\mathbf{P}(a_1) = \mathbf{P}(a)$ is true. To this end, we will prove that $\mathcal{I}$ is non empty, closed and open for the topology of $[0,1]$ and conclude that $\mathcal{I} = [0,1]$.

\textit{$\mathcal{I}$ is not empty}\\
For $t=0$, the existence of a solution $u$ such that $\nabla u \in \left(\mathcal{B}^2(\mathbb{R}^d)\right)^d$ is a direct consequence of Lemma \ref{existenceper}. We just have to use \eqref{estcontinuite2} to show the uniform Holder continuity of the gradient of the solution.

\textit{$\mathcal{I}$ is open}\\
We assume there exists $t\in \mathcal{I}$ and we will find $\varepsilon>0$ such that $[t,t+\varepsilon] \subset \mathcal{I}$. For $f \in \left(\mathcal{B}^2(\mathbb{R}^d) \cap \mathcal{C}^{0,\alpha}(\mathbb{R}^d)\right)^d$, we want to solve : 
\begin{equation}
    \label{eqconnectednessopen}
    -\operatorname{div}((a_t + \varepsilon \Tilde{a})\nabla u) = \operatorname{div}(f) \quad \text{in $\mathbb{R}^d$},
\end{equation}
where $\nabla u \in \left(\mathcal{B}^2(\mathbb{R}^d)\cap \mathcal{C}^{0, \alpha}(\mathbb{R}^d)\right)^d$. According to Proposition \ref{stability}, for such a solution, we have $\varepsilon \Tilde{a} \nabla u \in \left(\mathcal{B}^2(\mathbb{R}^d)\cap \mathcal{C}^{0,\alpha}(\mathbb{R}^d)\right)^d$. Next, we remark that equation \eqref{eqconnectednessopen} is equivalent to : 
\begin{equation}
\label{contraction}
\nabla u = \Phi_t (\varepsilon \Tilde{a}\nabla u + f),
\end{equation}
where $\Phi_t$ is the reciprocal linear operator associated with the equation when $a = a_t$. Lemma~\ref{lemme2} and Estimate \eqref{estcontinuite2} imply the continuity of $\Phi_t$ from $\left(\mathcal{B}^2(\mathbb{R}^d) \cap \mathcal{C}^{0,\alpha}(\mathbb{R}^d)\right)^d$ to $\left(\mathcal{B}^2(\mathbb{R}^d) \cap \mathcal{C}^{0,\alpha}(\mathbb{R}^d)\right)^d$ for the norm $\|.\|_{\mathcal{B}^2(\mathbb{R}^d)} + \|.\|_{\mathcal{C}^{0,\alpha}(\mathbb{R}^d)}$. We fix $\varepsilon$ such that : $$\varepsilon \left(\|\Tilde{a}\|_{L^{\infty}(\mathbb{R}^d)} + \|\Tilde{a}_{\infty}\|_{L^{\infty}(\mathbb{R}^d)}\right) \|\Phi_t\|_{\mathcal{L}\left(\left(\mathcal{B}^2(\mathbb{R}^d) \cap \mathcal{C}^{0,\alpha}(\mathbb{R}^d)\right)^d\right)}  <1.$$ 
Therefore $g \rightarrow \Phi_t(\varepsilon\Tilde{a}g + f) \in \mathcal{L}\left(\left(\mathcal{B}^2(\mathbb{R}^d) \cap \mathcal{C}^{0,\alpha}(\mathbb{R}^d)\right)^d \right)$ is a contraction in a Banach space. Finally, we can apply the Banach fixed-point theorem  to obtain the existence and the uniqueness of a solution to \eqref{contraction} and we deduce that $[t,t+ \varepsilon] \subset \mathcal{I}$. 

\textit{$\mathcal{I}$ is closed}\\
We assume there exist a sequence $(t_n) \in \mathcal{I}^{\mathbb{N}}$ and $t\in [0,1]$ such that $\lim_{n \rightarrow \infty} t_n =t$ and $t_n<t$. For every $t_n$, there exists $u_n$ solution to : 
\begin{equation}
\label{eqsuite}
-\operatorname{div}(a_{t_n}\nabla u_n)=f \quad \text{in $\mathbb{R}^d$},
\end{equation}
such that $\nabla u_n \in \left(\mathcal{B}^2(\mathbb{R}^d) \right)^d$.
For every $n \in \mathbb{N}$, Lemma \ref{lemme2} gives the existence of a constant $C_n$ such that : 
$$\|\nabla u_n\|_{B^2(\mathbb{R}^d)}  \leq C_n  \|f\|_{B^2(\mathbb{R}^d)}. $$

We first assume that $C_n$ is bounded independently of $n$ by a constant $C>0$. Therefore, up to an extraction, $\nabla u_n$ weakly converges to a gradient $\nabla u$ in $L^2_{loc}(\mathbb{R}^d)$ and, using the lower semi-continuity of the $L^2$-norm, we have  $$\|\nabla u\|_{L^2_{unif}(\mathbb{R}^d)} + \sup_p \|\nabla u \|_{L^2(\textit{V}_p)}  \leq \liminf_{n \to \infty} \|\nabla u_n\|_{L^2_{unif}(\mathbb{R}^d)} + \sup_p \|\nabla u_n \|_{L^2(\textit{V}_p)} \leq   C  \|f\|_{\mathcal{B}^2(\mathbb{R}^d)}.$$
In addition, for every $n\in \mathbb{N}$, $u_n$ is a solution to the equivalent equation :
\begin{equation}
\label{eqmodif}
-\operatorname{div}(a_t \nabla u_n) = \operatorname{div}(f + (a_{t_n} - a_{t})\nabla u_n)
\end{equation}
Next, since $t_n$ converges to $t$, we directly obtain that $a_{t_n}$ converges to $a_t$ in $\mathcal{B}^2(\mathbb{R}^d)$. In addition, since $\nabla u_n$ is bounded by a constant independent of $n$ in $L^2_{unif}(\mathbb{R}^d)$, the sequence $(a_{t_n} - a_{t})\nabla u_n$ strongly converges to $0$ in $L^2_{loc}(\mathbb{R}^d)$. We can therefore consider the limit in \eqref{eqmodif} when $n \to \infty$ and deduce that $u$ is a solution to : 
\begin{equation}
\label{eqlimite}
-\operatorname{div}(a_t \nabla u) = \operatorname{div}(f).    
\end{equation}

We have to prove that $\nabla u \in \mathcal{B}^2(\mathbb{R}^d)$. For every $m, \ n \in \mathbb{N}$, $u_n - u_m$ is a solution to :
$$-\operatorname{div}(a_t (\nabla u_n -\nabla u_m)) = \operatorname{div}((a_{t_n} - a_{t})\nabla u_n - (a_{t_m} - a_{t})\nabla u_m),$$ 
and we have the following estimate : 
$$\|\nabla u_n - \nabla u_m\|_{\mathcal{B}^2(\mathbb{R}^d)}  \leq C  \|(a_{t_n} - a_{t})\nabla u_n - (a_{t_m} - a_{t})\nabla u_m\|_{\mathcal{B}^2(\mathbb{R}^d)}.$$
Therefore, $u_n$ is a Cauchy-sequence in $\left(\mathcal{B}^2(\mathbb{R}^d)\right)^d$ and since this space is a Banach space, we directly obtain that $\nabla u$ belongs to $\left(\mathcal{B}^2(\mathbb{R}^d)\right)^d$ and we have the expected result. 

Now, we want to prove that $C_n$ is bounded independently of $n$ using a proof by contradiction. 
We assume there exist two sequences $f_n$ and $\nabla u_n$ in $\left(\mathcal{B}^2(\mathbb{R}^d)\right)^d$ such that : 
\begin{equation}
-\operatorname{div}(a_{t_n}\nabla u_n) = \operatorname{div}(f_n) \quad \text{in $\mathbb{R}^d$},
\end{equation}
\begin{equation}
\lim_{n \rightarrow \infty}\|f_n\|_{\mathcal{B}^2(\mathbb{R}^d)} =0,  
\end{equation}
\begin{equation}
\forall n \in \mathbb{N} \quad \|\nabla u_n\|_{\mathcal{B}^2(\mathbb{R}^d)} = 1.  
\end{equation}
For every $n \in \mathbb{N}$, the above equation is equivalent to : 
$$-\operatorname{div}(a_t \nabla u_n) = \operatorname{div}(f_n + (a_{t_n} - a_{t})\nabla u_n).$$
 We can next remark that the boundedness of $\nabla u_n$ in $\mathcal{B}^2(\mathbb{R}^d)$ ensures that the sequence $(a_{t_n} - a_{t})\nabla u_n$ is strongly convergent to 0 in $\mathcal{B}^2(\mathbb{R}^d)$. Finally, we can conclude exactly as in the proof of Lemma \ref{lemme2}.

Since $[0,1]$ is a connected space, we can finally conclude that $\mathcal{I}=[0,1]$. In addition, if $u \in \mathcal{D}'(\mathbb{R}^d)$ is such that $\nabla u \in \mathcal{B}^2(\mathbb{R}^d)\subset{L^2_{loc}(\mathbb{R}^d)}$, the result of \cite[corollary 2.1]{deny1954espaces} finally ensures that $u \in L^2_{loc}(\mathbb{R}^d)$.
\end{proof}

In the above proof, we have proved the following result :

\begin{corol}
\label{corol_regularity_corrector}
Let $f \in \mathcal{B}^2(\mathbb{R}^d) \cap \mathcal{C}^{0,\alpha}(\mathbb{R}^d)$ and $u \in H^1_{loc}(\mathbb{R}^d)$ solution to \eqref{equationref} such that $\nabla u \in \mathcal{B}^2(\mathbb{R}^d)$. Then $\nabla u \in  \mathcal{C}^{0,\alpha}(\mathbb{R}^d)$.
\end{corol}

\begin{remark}
Again, we do not need the restriction that we did on the dimension to prove the existence results stated in this section and we can easily generalize the existence of a solution to \eqref{equationref} in a two-dimensional context.
\end{remark}

\subsection{Existence of the corrector}

To conclude this section, we finally give a proof of Theorem \ref{theoreme1} and, therefore, we obtain the existence of a unique corrector solution to \eqref{correcteur} such its gradient belongs to $L^2_{per}(\mathbb{R}^d) + \mathcal{B}^2(\mathbb{R}^d)$. To this end, we remark that corrector equation \eqref{correcteur} is equivalent to a an equation in form \eqref{equationref} and we use the preliminary results of uniqueness and existence proved in this section. 

\begin{proof}[Proof of theorem \ref{theoreme1}]
$\textit{Existence}$\\
Let $p$ be in $\mathbb{R}^d$. We want to find a solution to \eqref{correcteur} of the form $w_{per,p} + \Tilde{w}_p$ where $w_{per,p}$ is the unique periodic corrector (that is the unique periodic solution to the corrector equation \eqref{correcteur} when $\Tilde{a} = 0$) and such that $\nabla \Tilde{w}_p\in \left(\mathcal{B}^2(\mathbb{R}^d)\right)^d$. First of all, we remark that equation \eqref{correcteur} is equivalent to : 
\begin{equation}
\label{correcteur2}
-\operatorname{div}((a_{\textit{per}} + \Tilde{a})\nabla \Tilde{w}_p) = \operatorname{div}(\Tilde{a}(p + \nabla w_{per,p})) \quad \text{in $\mathbb{R}^d$}.      
\end{equation}
It is well known that $\nabla w_{per,p} \in \left(L^2_{per}(\mathbb{R}^d) \cap \mathcal{C}^{0,\alpha}(\mathbb{R}^d)\right)^d$ and therefore, using the periodicity of $\nabla w_{per,p}$, we can easily show that $\Tilde{a}(p + \nabla w_{per,p}) \in \left(\mathcal{B}^2(\mathbb{R}^d)\cap \mathcal{C}^{0,\alpha}(\mathbb{R}^d)\right)^d$. Then, the existence of $\Tilde{w}_p$ such that $\nabla \Tilde{w}_p \in \left(\mathcal{B}^2(\mathbb{R}^d)\cap \mathcal{C}^{0,\alpha}(\mathbb{R}^d)\right)^d$ is given by Lemma \ref{lemmexistence} and Corollary \ref{corol_regularity_corrector}. Since $\nabla \Tilde{w}_p \in \left(\mathcal{B}^2(\mathbb{R}^d)\cap \
L^{\infty}(\mathbb{R}^d)\right)^d $, the sub-linearity at infinity of $\Tilde{w}_p$ is a direct consequence of Proposition \ref{propsouslinearite}.

$\textit{Uniqueness}$\\
We assume there exist two solutions $u_1$ and $u_2$ to \eqref{correcteur} such that $\nabla u_1$ and $\nabla u_2$ belong to $\left(L^2_{per}(\mathbb{R}^d) + \mathcal{B}^2(\mathbb{R}^d)\right)^d$. We denote $v=u_1-u_2$ and we have $\nabla v = g_{per} + \Tilde{g}$ where $g_{per}  \in \left(L^2_{per}(\mathbb{R}^d)\right)^d$ and $\Tilde{g}\in\left(\mathcal{B}^2(\mathbb{R}^d)\right)^d$. For every $q\in \mathcal{P}$, we have $\tau_{q}\nabla v = g_{per} + \tau_{q}\Tilde{g}$ by periodicity of $g_{per}$. Since $\Tilde{g}$ belongs to $\left(\mathcal{B}
^2(\mathbb{R}^d)\right)^d$, there exists $\Tilde{g}_{\infty} \in \left(L^2(\mathbb{R}^d)\right)^d$ such that $\tau_q\nabla v$ converges in $\mathcal{D}'(\mathbb{R}^d)$ to $\nabla v_{\infty} = g_{per} + \Tilde{
g}_{\infty}$ when $|q| \rightarrow \infty$. In addition, considering the limit in equation \eqref{correcteur}, we obtain that $v_{\infty}$ is a solution to : 
$$-\operatorname{div}((a_{per}+ \Tilde{a}_{\infty})\nabla v_{\infty}) = 0 \quad \text{in } \mathbb{R}^d.$$
Since $a$ satisfies assumption \eqref{hypothèses1} and \eqref{hypothèses2}, the coefficient $a_{per} + \Tilde{a}_{\infty}$ is a bounded and elliptic matrix-valued coefficient. Therefore, the result established in  \cite[Lemma 1]{blanc2012possible} allows us to conclude that $g_{per} = 0$ and finally, that $\nabla v = \Tilde{g} \in \left(\mathcal{B}^2(\mathbb{R}^d)\right)^d$. Since $v$ is a solution to : 
$$-\operatorname{div}((a_{per} + \Tilde{a})\nabla v) = 0 \quad \text{in } \mathbb{R}^d,$$ 
we use Lemma \ref{lemme1} to obtain that $\nabla v =0$ and the uniqueness is proved.  
\end{proof}

\section{Homogenization results and convergence rates }
\label{Section5}

In this section we use the corrector, solution to \eqref{correcteur} and defined in Theorem \ref{theoreme1}, to establish an homogenization theory similar to that established in \cite{blanc2018precised} for the periodic case with local perturbations. In proposition \ref{proplimithomogenization} we first study the homogenized equation associated with~\eqref{equationepsilon} and we conclude showing estimates \eqref{estimate1} and \eqref{estimate2} stated in Theorem \ref{theoreme3}.

\subsection{Homogenization results}

To start with, we determine here the limit of the sequence $u^{\varepsilon}$ of solutions to \eqref{equationepsilon}. In Proposition \ref{proplimithomogenization} below we show the homogenized equation is actually the diffusion equation~\eqref{homog} where the diffusion coefficient $a^*$ is defined by \eqref{homogenizedcoeff}, that is the homogenized coefficient is the same as in the periodic case when $a = a_{per}$. This phenomenon is similar to the results established in \cite{blanc2018correctors} in the case of localized defects of $L^p(\mathbb{R}^d)$. It is a direct consequence of Proposition \ref{moyenne} regarding the average of the functions in $\mathcal{B}^2(\mathbb{R}^d)$ which is satisfied by our perturbations. The idea is that, on average, the perturbations belonging to $\mathcal{B}^2(\mathbb{R}^d)$ therefore do not impact the periodic background. 

 \begin{prop}
 \label{proplimithomogenization}
 Assume $\Omega$ is an open bounded set of $\mathbb{R}^d$, let $f \in L^2(\Omega)$ and consider the sequence $u^{\varepsilon}$ of solutions in $H^1_0(\Omega)$ to \eqref{equationepsilon}.
Then the homogenized (weak-$H^1(\Omega)$ and strong-$L^2(\Omega)$) limit $u^*$ obtained when $\varepsilon \rightarrow 0$ is the solution to \eqref{homog} where the homogenized coefficient is identical to the periodic homogenized coefficient \eqref{homogenizedcoeff}.
 \end{prop}

 \begin{proof}
 We denote $w = (w_{e_i})_{i \in \{1,...,d\}}$, the correctors given by Theorem \ref{theoreme1} for $p=e_i$. The general homogenization theory of equations in divergence form (see for instance \cite[Chapter 6, Chapter 13]{tartar2009general}), gives the convergence, up to an extraction, of $u^{\varepsilon}$ to a function $u^*$ solution to an equation in the form \eqref{homog}. In addition, for all $1 \leq i,j \leq d$, the homogenized matrix $a^*$ associated with $a$ is given by : 
$$ \left[ a^* \right]_{i,j} = weak \lim_{\varepsilon \rightarrow 0} a(./\varepsilon)(I_d + \nabla w(./\varepsilon)),$$
where the weak limit is taken in $L^{2}(\Omega)^{d \times d}$. By assumption, we have $a = a_{per} + \Tilde{a}$ and we know that $\nabla w_{e_i} = \nabla w_{per,e_i} + \nabla \Tilde{w}_{e_i}$ where $\Tilde{a} \in \left(\mathcal{B}^2(\mathbb{R}^d) \cap L^{\infty}(\mathbb{R}^d)\right)^{d\times d}$ and $\nabla \Tilde{w}_{e_i}~\in \left(\mathcal{B}^2(\mathbb{R}^d) \cap L^{\infty}(\mathbb{R}^d)\right)^{d}$. Therefore, Corollary \ref{convergenceLinfinistar} ensures that $|\Tilde{a}(./\varepsilon)|$ and $|\nabla \Tilde{w}_{e_i}(./ \varepsilon)|$ converge to 0 in $weak^*-L^{\infty}$ and, since $a_{per}$ and $\nabla w_{per}$ are bounded, we can deduce that : 
$$weak \lim_{\varepsilon \rightarrow 0} a_{per}(./\varepsilon) \nabla \Tilde{w}(./\varepsilon) + \Tilde{a}(./\varepsilon)(I_d + \nabla w(./\varepsilon)) = 0.$$
Consequently, we have  
$$ \left[ a^* \right]_{i,j} = weak \lim_{\varepsilon \rightarrow 0} a_{per}(./\varepsilon)(I_d + \nabla w_{per}(./\varepsilon)) = \left[ a_{per}^* \right]_{i,j}.$$
This limit being independent of the extraction, all the sequence $u^{\varepsilon}$ converges to $u^*$ and we have the equality $a^* = a^*_{per}$.
 \end{proof}

\subsection{Approximation of the homogenized solution and quantitative estimates}

The existence of the corrector established in Theorem \ref{theoreme1} allows to consider a sequence of approximated solutions defined by $u^{\varepsilon,1} = u^* + \varepsilon \sum_{i=1}^d \partial_{i}u^* w_{i}(./ \varepsilon)$ where for every $i$ in $\{1,...,d\}$, we have denoted $w_i=w_{e_i}$. Our aim here is to estimate the accuracy of this approximation for the topology of $H^1(\Omega)$. In particular, we want to prove the convergence to $0$ of the sequence $R^{\varepsilon}$ defined by : 
\begin{equation}
\label{defR}
R^{\varepsilon}(x) = u^{\varepsilon}(x) - u^*(x) - \varepsilon \sum_{j=1}^d w_j\left(\dfrac{x}{\varepsilon}\right)\partial_j u^*(x),
\end{equation}
and specify the convergence rate in $H^1(\Omega)$.

A classical method in homogenization used to obtain some expected quantitative estimates consists to define the divergence-free matrix $M_k^i(x) = a_{i,k}^* - \sum_{j=1}^d a_{i,j}(x)(\delta_{j,k} + \partial_j w_k(x))$ and to find a potential $B$ which formally solves $M = \operatorname{curl}(B)$. Knowing that both the coefficient $a$ and $\nabla w$ belong to $L^2_{per} + \mathcal{B}^2(\mathbb{R}^d)$, we can split $M$ in two terms and obtain $M = M_{per} + \Tilde{M} \in \left(L^2_{per}(\mathbb{R}^d) + \mathcal{B}^2(\mathbb{R}^d)\right)^{d\times d}$. Therefore, we expect to find a potential of the same form, that is $B = B_{per} + \Tilde{B}$. Rigorously, for every $i,j \in \{1,...,d\}$, we want to solve the equation : 
\begin{equation}
\label{equationdefB}
 - \Delta B^{i,j}_k = \partial_j M_k^i - \partial_i M_k^j  \quad \text{in } \mathbb{R}^d.
\end{equation}
The existence of a periodic potential $B_{per}$ solution to $M_{per} = \operatorname{curl}(B_{per})$ is well known since, component by component, $M_{per}$ is divergence-free. Here, the main difficulty is to show the existence of the potential $\Tilde{B}$ associated with the $\mathcal{B}^2$-perturbation. This result is given by the following lemma.

\begin{lemme}
\label{lemme_existence_B}
Let $\Tilde{M} = \big( \Tilde{M}^i_k\big)_{1 \leq i,k\leq d} \in \left(\mathcal{B}^2(\mathbb{R}^d)\right)^{d \times d}$ such that $\operatorname{div}(\Tilde{M}_k) = 0$ for every $k \in \{1, ...,d\}$. Then, the potential $\Tilde{B}^{i,j}_k$ defined by  : 
\begin{equation}
\label{defpotentiel}
 \Tilde{B}^{i,j}_k(x) = C(d) \int_{\mathbb{R}^d} \left( \dfrac{x_i - y_i}{|x-y|^d} \Tilde{M}^j_k(y) - \dfrac{x_j - y_j}{|x-y|^d}\Tilde{M}^i_k(y) \right) dy,   
\end{equation}
where $C(d)>0$ is a constant associated with the unit ball surface of $\mathbb{R}^d$, satisfies $\nabla \Tilde{B}^{i,j}_k \in \left(\mathcal{B}^2(\mathbb{R}^d)\right)^d$ and for all $i,j,k \in \{1,...,d\}$ :
\begin{align}
\label{potentielperturbe}
- \Delta &\Tilde{B}^{i,j}_k = \partial_j\Tilde{M}_k^i - \partial_i \Tilde{M}_k^j,\\
\label{antisymetrie}
& \Tilde{B}_k^{i,j} = - \Tilde{B}_k^{j,i}, \\
\label{divergencepot}
& \sum_{i=1}^{d} \partial_i \Tilde{B}_k^{i,j} = \Tilde{M}_k^j.
\end{align}
In addition, there exists a constant $C_1>0$ which only depends of the ambient dimension $d$ such that : 
\begin{equation}
\label{estimpotentiel}
\|\nabla \Tilde{B}\|_{\mathcal{B}^2(\mathbb{R}^d)} \leq C_1 \|\Tilde{M}\|_{\mathcal{B}^2(\mathbb{R}^d)}.   
\end{equation}
\end{lemme}

\begin{proof}
First, for every $i,j,k\in \{1,...,d\}$, equation \eqref{potentielperturbe} is equivalent to an equation of the following form : 
$$ - \Delta \Tilde{B}^{i,j}_k = \operatorname{div}\left(\mathcal{M}^{i,j}_k \right),$$
where $\mathcal{M}^{i,j}_k$ is a vector function defined by : 
$$ \left(\mathcal{M}^{i,j}_k \right)_l = \left\{
\begin{array}{cc}
   \Tilde{M}_k^i   & \text{if  } l = j, \\
    -\Tilde{M}_k^j  & \text{if  } l = i, \\ 
    0  & \text{else. } \\ 
\end{array}
\right.$$
Since $\mathcal{M}^{i,j}_k \in \left(\mathcal{B}^2(\mathbb{R}^d)\right)^d$, the existence of $\Tilde{B}$ and estimate \eqref{estimpotentiel} are given by Lemmas \ref{lemmeexistenceperiodique1}, \ref{existenceper} and \ref{lemme2} (here $a_{per} \equiv 1$). Equality \eqref{antisymetrie} is a direct consequence of the definition of $\Tilde{B}$. Property \eqref{divergencepot} can be easily obtained applying the divergence operator to \eqref{defpotentiel}. 
\end{proof}

\begin{corol}
The potential $B = B_{per} + \Tilde{B}$, where $\Tilde{B}$ is given by Lemma \ref{lemme_existence_B}, is the expected potential solution of \eqref{equationdefB}. In addition, the couple $(M,B)$ satisfies the following equalities :
\begin{align}
\label{antisymetrie2}
& B_k^{i,j} = - B_k^{j,i}, \\
\label{divergencepot2}
& \sum_{i=1}^{d} \partial_i B_k^{i,j} = M_k^j.
\end{align}
\end{corol}

Now that existence of the potential $B$ has been deal with, we can remark that $R^{\varepsilon}$ is a solution to the following equation :
\begin{equation}
\label{equationR}
-\operatorname{div}\left(a\left(\dfrac{x}{\varepsilon}\right)\nabla R^{\varepsilon}\right) = \operatorname{div}(H^{\varepsilon}) \quad \text{in } \Omega,
\end{equation}
where : 
\begin{equation}
\label{defHeps}
H_i^{\varepsilon}(x) = \varepsilon \sum_{j,k = 1}^d a_{i,j}\left(\dfrac{x}{\varepsilon}\right)w_{k}\left(\dfrac{x}{\varepsilon}\right) \partial_j \partial_k u^*(x) - \varepsilon \sum_{j,k =1}^d B_k^{i,j}\left(\dfrac{x}{\varepsilon}\right)\partial_j \partial_k u^*(x).
\end{equation}
For a complete proof of equality \eqref{equationR}, we refer to \cite[Proposition 2.5]{blanc2018precised}.

To conclude, we have to study the properties of $H^{\varepsilon}$. In particular, we next prove that both the corrector $\Tilde{w}$ and the potential $\Tilde{B}$ are bounded. This property is essential for establishing the estimates of Theorem \ref{theoreme3}.

\begin{lemme}
\label{lemmeborne}
The corrector $w=\left(w_i\right)_{i \in \{1,...,d\}}$ defined by Theorem \ref{theoreme1} and the potential $B$ solution to \eqref{equationdefB} are in $L^{\infty}(\mathbb{R}^d)$. 
\end{lemme}

\begin{proof}
First, it is well known that both $w_{per}$ and  $B_{per}$ belong to $L^{\infty}(\mathbb{R}^d)$. Next, for all $i \in \{1,...,d\}$, $\Tilde{w}_i$ solves : 
$$-\operatorname{div} \left(a_{per} \nabla \Tilde{w}_i \right) = \operatorname{div}\left(\Tilde{a}\left(e_i + \nabla w_{per,i} + \nabla \Tilde{w}_i \right)\right).$$
We know the gradient of the corrector defined in Theorem \ref{theoreme1} is in $\left(\mathcal{C}^{0,\alpha}(\mathbb{R}^d)\right)^d$. A direct consequence of Assumption \eqref{hypothèses2} and Proposition \ref{stability} ensures that $f = \Tilde{a}\left(e_i + \nabla w_{per,i} + \nabla \Tilde{w}_i \right)$ belongs to $\left(L^{\infty}(\mathbb{R}^d) \cap \mathcal{B}^2(\mathbb{R}^d)\right)^d$ and the results of uniqueness and existence established in Lemmas \ref{lemme1} and \ref{existenceper} imply we have the following representation : 
\begin{equation}
\Tilde{w}_i(x) = \int_{\mathbb{R}^d} \nabla_y G_{per}(x,y)f(y) dy.
\end{equation}
We want to prove that the integral is bounded independently of $x$. We take $x\in \mathbb{R}^d$ and denote $p_x$ the unique element of $\mathcal{P}$ such that $x \in \textit{V}_{p_x}$. We define $W_{p_{x}} = W_{2^{p_x}}$ such as in Proposition \ref{openconstruction} and we split the integral in three parts  : 
\begin{align*}
\int_{\mathbb{R}^d} \nabla_y G_{per}(x,y)f(y) dy & = \int_{B_1(x)} \nabla_y G_{per}(x,y)f(y) dy +  \int_{W_{p_x} \setminus B_1(x)} \nabla_y G_{per}(x,y)f(y) dy \\
& +  \int_{ \mathbb{R}^d \setminus W_{p_x}} \nabla_y G_{per}(x,y)f(y) dy  = I_1(x) + I_2(x) + I_3(x).
\end{align*}
\\
We start by finding a bound for $I_1(x)$. To this end, we use Estimate \eqref{estimgreen1} for the Green function and we obtain
\begin{align*}
\left| I_1(x) \right| & \leq  \|f\|_{
L^{\infty}(\mathbb{R}^d)} \int_{B_1(x)} \left| \nabla_y G_{per}(x,y)\right|dy   \\ 
& \leq C \|f\|_{
L^{\infty}(\mathbb{R}^d)} \int_{B_1(x)} \dfrac{1}{\left|x-y\right|^{d-1}}dy \leq C \|f\|_{
L^{\infty}(\mathbb{R}^d)}.
\end{align*}
Where $C$ denotes a positive constant independent of $x$.  Indeed, the value of the integral in the last inequality depends only of the dimension $d$.\\

Next, using Proposition \ref{openconstruction}, we know there exists $C_1>0$ and $C_2>0$ independent of $x$ such that $W_{p_x} \subset~B_{C_12^{p_x}(x)}$ and the number of $q \in \mathcal{P}$ such that $V_q \cap W_{p_x} \neq \emptyset$ is bounded by $C_2$. Therefore, using the Cauchy-Schwarz inequality, we have :

\begin{align*}
|I_2(x)| & \leq \int_{W_{p_x}\setminus{B_1(x)}} \dfrac{1}{|x-y|^{(d-1)}} |f(y)|dy \\
& \leq \left(\int_{W_{p_x}\setminus{B_1(x)}} \dfrac{1}{|x-y|^{2(d-1)}} dy\right)^{1/2}\left(\int_{W_{p_x}\setminus{B_1(x)}}  |f(y)|^2dy\right)^{1/2} \\
& \leq C_2 \left(\int_{ B_{C_1 2^{p_x}(x)} \setminus{B_1(x)}} \dfrac{1}{|x-y|^{2(d-1)}} dy\right)^{1/2} \sup_{p \in \mathcal{P}} \|f\|_{L^2(V_q)}.
\end{align*}

In addition since $d>2$, we have : 
$$\int_{ B_{C_1 2^{p_x}(x)} \setminus{B_1(x)}} \frac{1}{\left|x-y\right|^{2(d-1)}} dy = \int_{ B_{C_1 2^{p_x}(0)} \setminus{B_1(0)}} \dfrac{1}{\left|y\right|^{2(d-1)}} dy \leq C\left(1 - \dfrac{1}{2^{|p_x|(d-2)}}\right),$$
and therefore :
$$I_2(x) \leq C \left(1 - \dfrac{1}{2^{|p_x|(d-2)}}\right)^{1/2} \leq C.$$

Finally, to bound $I_3(x)$ we split the integral on each cell $V_q$ for $q \in \mathcal{P}$. Using the Cauchy-Schwarz inequality, we obtain : 
\begin{align*}
\left|I_3(x)\right| & \leq  \sum_{q \in \mathcal{P}} \int_{ V_q \setminus{W_{p_x}}} \left|\nabla_y G_{per}(x,y)f(y) \right| dy  \\
& \leq \sum_{q \in \mathcal{P}} \left( \int_{V_q \setminus{W_{p_x}}} \left|\nabla_y G_{per}(x,y)\right| ^2 dy \right)^{\frac{1}{2}} \left(\int_{V_q \setminus{W_{p_x}}}\left|f(y)\right| ^2  dy\right)^{\frac{1}{2}} \\ 
 & \leq \|f\|_{\mathcal{B}^2(\mathbb{R}^d)} \sum_{q \in \mathcal{P} } \left( \int_{V_q \setminus{W_{p_x}}} \left|\nabla_y G_{per}(x,y)\right| ^2 dy \right)^{\frac{1}{2}}.
\end{align*}

We proceed exactly as in the proof of Lemma \ref{lemmeexistenceperiodique1} (see the proof of estimate \eqref{estimunifnablaI2p} for details) to obtain : 
\begin{align*}
\sum_{q \in \mathcal{P}} \left( \int_{V_q \setminus{W_{p_x}}} \left| \nabla_y G_{per}(x,y) \right|^2 dy \right)^{\frac{1}{2}} & \leq C \sum_{q \in \mathcal{P}} \left( \int_{V_q \setminus{W_{p_x}}} \dfrac{1}{\left|x-y\right|^{2(d-1)}} dy \right)^{\frac{1}{2}} \\
& \leq C \sum_{q \in \mathcal{P}}  \frac{1}{2^{|q|(d-2)/2}} < \infty.
\end{align*}

Finally we have bounded the integral independently of $x$ and we deduce that $\Tilde{w}_i \in ~L^{\infty}(\mathbb{R}^d)$. With the same method we obtain the same result for $B = B_{per} + \Tilde{B}$ which allows us to conclude. 
\end{proof}

\begin{remark}
\label{remarkbornedimension2}
The assumption $d>2$ is essential in the above proof and the boundedness of $\Tilde{w}$ in $L^{\infty}(\mathbb{R}^d)$ may be false if $d=1$ or $d=2$. If $d=1$ we give a counter-example in Remark~\ref{dim1}. The case $d=2$ is a critical case and we are not able to conclude. This phenomenon is closely linked to the critical integrability of the function $|x|^{-1}$ in $L^2(\mathbb{R}^2)$.
\end{remark}

We are now able to give a complete proof of Theorem \ref{theoreme3}.
\begin{proof}[Proof of Theorem \ref{theoreme3}]
First, we use the explicit definition of $H^{\varepsilon}$ given by \eqref{defHeps} and a triangle inequality to obtain the following estimate : 
$$ \|H^{\varepsilon}\|_{L^2(\Omega)} \leq 
\left( 1 + \|a\|_{L^{\infty}(\mathbb{R}^d)}\right) \| D^2 u^*\|_{L^2(\Omega)} \left(\| \varepsilon w(./\varepsilon) \|_{L^{\infty}(\Omega)} + \| \varepsilon B(./\varepsilon) \|_{L^{\infty}(\Omega)} \right).$$
Applying Lemma \ref{lemmeborne}, we obtain the existence of $C>0$ independent of $\varepsilon$ such that : 
\begin{equation}
\label{estimationH}
\|H^{\varepsilon}\|_{L^2(\Omega)} \leq C \varepsilon \| D^2 u^*\|_{L^2(\Omega)}.
\end{equation}

Next, we use the two following estimates satisfied by $R^{\varepsilon}$ : 
\begin{equation}
\label{estimationR}
\| R^{\varepsilon}\|_{L^2(\Omega)} \leq C_1 \left(\varepsilon \left(\| w(./\varepsilon) \|_{L^{\infty}(\Omega)} + \| B(./\varepsilon) \|_{L^{\infty}(\Omega)} \right) \|f\|_{L^2(\Omega)}   + \|H^{\epsilon}\|_{L^2(\Omega)}\right),
\end{equation}
and for every $\Omega_1 \subset \subset \Omega$ : 
\begin{equation}
\label{estimationnablaR}
\|\nabla R^{\varepsilon} \|_{L^2(\Omega_1)} \leq C_2 \left( \| H^{\varepsilon} \|_{L^2(\Omega)} + \| R^{\varepsilon} \|_{L^2(\Omega)} \right),
\end{equation}
where $C_1>0$ and $C_2>0$ are independent of $\varepsilon$. These estimates are established for instance in
\cite{blanc2018precised} where the authors use the elliptic regularity associated with equation \eqref{equationR} and the properties of the Green function associated with the operator $-\operatorname{div}(a^*\nabla.)$ on $\Omega$ with homogeneous Dirichlet boundary condition.  The first estimate is established in the proof of \cite[Lemma 4.8]{blanc2018precised} and the second estimate is a classical inequality of elliptic regularity proved in \cite[Proposition 4.2]{blanc2018precised} and applied to equation \eqref{equationR}.

In addition, an application of elliptical regularity to equation \eqref{homog} provides the existence of $C_3>0$ such that : 
\begin{equation}
\label{D2u}
\|u^*\|_{H^2(\Omega)} \leq C_3 \|f\|_{L^2(\Omega)}.
\end{equation}

To conclude we use Lemma \ref{lemmeborne} to bound $w$ and $B$ and estimates \eqref{estimationH}, \eqref{estimationR}, \eqref{estimationnablaR} and \eqref{D2u}. We finally obtain : 
\begin{equation*}
\|R^{\varepsilon}\|_{L^2(\Omega)} \leq C \varepsilon \|f\|_{L^2(\Omega)},
\end{equation*}
and
\begin{equation*}
\|\nabla R^{\varepsilon}\|_{L^2(\Omega_1)} \leq \Tilde{C} \varepsilon  \|f\|_{L^2(\Omega)},
\end{equation*}
where $C$ and $\Tilde{C}$ are independent of $\varepsilon$. We have proved Theorem 2.
\end{proof}

\begin{remark}
\label{dim1}
In the one-dimensional case, that is when $d=1$, we are not able to conclude in the same way. With an explicit calculation, we obtain :
\begin{align}
&(u^{\varepsilon})'(x)  = \left(a_{per} + \Tilde{a}\right)^{-1} (x/\varepsilon)\left( F(x) + C^{\varepsilon}\right), \\
&(u^{*})'(x)  = \left(a^*\right)^{-1} \left( F(x) + C^{*}\right),\\
&w(x) = - x + a^* \int_{0}^{x} \dfrac{1}{a_{per}(y)}dy  - a^* \int_{0}^{x} \dfrac{\Tilde{a}}{a_{per} \left(a_{per} + \Tilde{a}\right)}(y) dy,
\end{align}
where : 
\begin{align}
F(x) & = \int_{0}^x f(y)dy, \\
C^{\varepsilon} & = - \left( \int_0^1 (a_{per} + \Tilde{a})^{-1}(y/\varepsilon) \right)^{-1} \int_0^1 (a_{per} + \Tilde{a})^{-1}(y / \varepsilon) F(y) dy, \\
C^* & = - \int_0^1 F(y) dy.
\end{align}
In this case, $ \displaystyle w_{per}(x) = - x + a^* \int_{0}^{x} \dfrac{1}{a_{per}(y)}dy$  and $ \displaystyle \Tilde{w}(x) = - a^* \int_{0}^{x} \dfrac{\Tilde{a}}{a_{per} \left(a_{per} + \Tilde{a}\right)}(y) dy$ and we can show the corrector $w$ is not necessarily bounded.  However, the results of Proposition \ref{moyenne}, allow us to obtain the following estimate :
\begin{equation}
\label{estimate_one_dimensionnal}
    \|\left(R^{\varepsilon}\right)'\|_{L^2(\Omega)} \leq C \varepsilon^{\frac{1}{2}}\left|\log(\varepsilon)\right|^{\frac{1}{2}}.
\end{equation}
As an illustration, we can consider $\Omega =]0,1[$, $a_{per} = 1$ and $\displaystyle \Tilde{a} = \sum_{p \in \mathbb{Z}} \tau_{-2^p}\varphi$, where $\varphi$ is a positive function of  $\mathcal{D}(\mathbb{R})$, $\|\varphi\|_{L^{\infty}}=1$, $\displaystyle \int_{\mathbb{R}} \varphi> 0 $ and $Supp(\varphi) \in [0,1/2]$. With these parameters, for every $x\in \Omega$, we have : 
$$|\Tilde{w}(x/\varepsilon)| = \int_{0}^{x/\varepsilon} \dfrac{\Tilde{a}}{1 + \Tilde{a}}(y) dy \geq \frac{1}{2} \sum_{0\leq p < [\log_2(x/\varepsilon)]} \int_{0}^{1/2} \varphi \stackrel{\varepsilon \rightarrow 0}{\longrightarrow} + \infty.$$ 
And therefore, the corrector is actually not bounded. 
\end{remark}

\begin{remark}
The result of Theorem \ref{theoreme3} ensures that the corrector introduced in Theorem \ref{theoreme1} allows to precisely describe the behavior of the sequence $u^{\varepsilon}$ in $H^1$ using the approximation defined by $u^{\varepsilon,1} = u^* + \varepsilon \sum_{i=1}^d \partial_i u^* w_i(./\varepsilon)$. However, since the perturbations of $\mathcal{B}^2(\mathbb{R}^d)$ are "small" at the macroscopic scale (in the sens of average given by \eqref{convergencemeanB2}), we can naturally expect that it is also possible to approximate $u^{\varepsilon}$ in $H^1$ considering the sequence $u^{\varepsilon,1}_{per} := u^* + \varepsilon \sum_{i=1}^d \partial_i u^* w_{per,i}(./\varepsilon)$ which only uses the periodic part $w_{per}$ of our corrector.  To this aim, we can show that $u^{\varepsilon} - u^{\varepsilon,1}_{per}$ is solution to
$$-\operatorname{div}\left(a\left(\frac{.}{\varepsilon}\right) \nabla (u^{\varepsilon} - u^{\varepsilon,1}_{per})\right) = \operatorname{div}\left(H^{\varepsilon}_{per}\right) \quad \text{on } \Omega,$$
where the right-hand side $$H^{\varepsilon}_{per} = - a\left(\dfrac{.}{\varepsilon}\right)\left( \nabla \left( u^{\varepsilon} - u^{\varepsilon,1}\right) + \varepsilon \sum_{i=1}^d  \nabla \partial_i  u^* \Tilde{w}_i(./\varepsilon) + \sum_{i=1}^d  \partial_i  u^* \nabla \Tilde{w}_i(./\varepsilon)\right),$$ strongly converges to 0 in $L^2$ when $\varepsilon \to 0$. A method similar to that used in the proof of Theorem \ref{theoreme3} therefore allows to show the convergence to 0 of $u^{\varepsilon} - u^{\varepsilon,1}_{per}$ in $H^1$. It follows, at the macroscopic scale, that the choice of our adapted corrector instead of the periodic corrector seems to be not necessarily relevant in order to calculate an approximation of $u^{\varepsilon}$ in $H^1$. 
However, the choice of the periodic corrector is not adapted if the idea is to approximate the behavior of $u^{\varepsilon}$ at the microscopic scale $\varepsilon$. Indeed, at this scale, the perturbations of the periodic background can be possibly large and it is necessary to consider a corrector that take these perturbations into account. Particularly, we can easily show that $H^{\varepsilon}_{per}(\varepsilon.)$ does not converge to 0 in  any ball $B_R$ such that $\varepsilon B_R \subset \Omega$, which formally reflects a poor quality of the approximation of $u^{\varepsilon}$ by $u^{\varepsilon,1}_{per}$ at the scale $\varepsilon$. This fact particularly motivates the choice of our adapted corrector in order to approximate $u^{\varepsilon}$. We refer to \cite{goudey} for a rigorous formalization of the above argument. 
\end{remark}


\newpage

\appendix

\section{Appendix : The case of $\mathcal{B}^r(\mathbb{R}^d)$, $1<r<\infty$}

\label{AnnexeA}

The purpose of this section is to generalize the results established above in a context where the perturbation $\Tilde{a}$ locally behaves, at the vicinity of the points $x_p \in \mathcal{G}$, as a fixed function of $L^r(\mathbb{R}^d)$ truncated over the domain $V_p$, where $r$ can denote any Lebesgue exponent in $]1,+\infty[$. In this context, we can naturally generalize the definition \eqref{spacedef} of the space $\mathcal{B}
^2(\mathbb{R}^d)$ and, using the topology of $L^r$, introduce a collection of spaces $\mathcal{B}^r(\mathbb{R}^d)$ defined by \eqref{spacedefBr} in order to describe our perturbations of the periodic background. Exactly as in the case $r=2$, we consider here a perturbed coefficient $a$ of the form \eqref{coefficientform} which is the sum of a periodic coefficient $a_{per}$ and a defect $\Tilde{a}$ in $\mathcal{B}^r(\mathbb{R}^d)^{d\times d}$. Our aim is therefore to establish the homogenization of equation \eqref{equationepsilon} in this case, showing the existence of an adapted corrector solution to \eqref{correcteur} and establishing some convergence rates similar to that of Theorem \ref{theoreme3}. To this end, the idea is to follow the same strategy as in the case $r=2$. However, several difficulties appear when $r\neq 2$, some results established in our approach for the case $r = 2$ no longer hold and we have to adapt them. In particular, when $r\neq 2$, the space $\mathcal{B}^r$ defined below does not have a "Hilbert" structure induced by the topology of $L^2$, which prevents us from using some techniques (such as Caccioppoli-type inequalities for example) to prove the uniqueness results of Section \ref{Section4}. Moreover, when $r>d$, the decay far from the points $2^p$ of the functions of $\mathcal{B}^r$ is not sufficient to ensure the convergence of some series involving the Green function $G_{per}$ as in the proofs of Lemmas \ref{lemmeexistenceperiodique1}, \ref{existenceper} and \ref{lemmeborne}. Consequently, in this case we have to adapt our approach to prove the existence of a corrector $w$ given in Theorem \ref{theoreme4} and, in contrast to the case $r=2$ (see Lemma \ref{lemmeborne}), this corrector is not necessarily bounded in $L^{\infty}(\mathbb{R}^d)$. As we can see in Theorem \ref{theoreme5} below, this phenomenon implies in particular that the convergence rates of the sequence of rests $R^{\varepsilon}$ depend on the ratio $\dfrac{r}{d}$.

To start with, we fix $r\in ]1,+\infty[$ and we consider the following functional space : 
\begin{equation}
\label{spacedefBr}
 \mathcal{B}^r(\mathbb{R}^d) = \left\{f \in L^r_{unif}(\mathbb{R}^d) \ \middle| \ \exists f_{\infty} \in L^r(\mathbb{R}^d),  \lim_{|p| \rightarrow \infty}  \int_{\textit{V}_p} |f(x) - \tau_{-p} f_{\infty}(x)|^r dx = 0\right\},  
\end{equation}
equipped with the norm :
\begin{equation}
\label{normdefBr}
\|f \|_{\mathcal{B}^r(\mathbb{R}^d)} = \|f_{\infty}\| _{ L^r(\mathbb{R}^d)} + \|f\|_{L^r_{unif}(\mathbb{R}^d)}+ \sup_{p\in \mathcal{P}} \|f-\tau_{-p}f_{\infty}\| _{ L^r(\textit{V}_p)}.   
\end{equation}
In \eqref{spacedefBr}, \eqref{normdefBr} we have denoted by : 
\begin{equation}
L^r_{unif}(\mathbb{R}^d) = \left\{f \in L^r_{loc}(\mathbb{R}^d), \ \sup_{x \in \mathbb{R}^d}  \|f\|_{L^r(B_1(x))} < \infty\right\},
\end{equation}
and 
\begin{equation}
\|f \| _{ L^r_{unif}(\mathbb{R}^d)} = \sup_{x \in \mathbb{R}^d}  \|f\|_{L^r(B_1(x))}.
\end{equation}

In the sequel we assume again that the ambient dimension $d$ is equal or larger than~$3$. We consider here a matrix-valued coefficient of the form \eqref{coefficientform} where $a_{per}$ is periodic and $\Tilde{a}$ is in $\mathcal{B}^r(\mathbb{R}^d)^{d \times d}$. We also assume that $a_{per}$, $\Tilde{a}$ and the associated $L^r$-limit matrix-valued function $\Tilde{a}_{\infty}$ satisfy Assumptions \eqref{hypothèses1} and \eqref{hypothèses2} of coercivity, boundedness and Hölder continuity.

In this study we establish the two theorems below. In Theorem \ref{theoreme4}, we show the existence of an adapted corrector $w$, strictly sub-linear at the infinity, solution to \eqref{correcteur} and such that $\nabla w \in \left(L
^2_{per} + \mathcal{B}^r(\mathbb{R}^d)\right)^d$. In Theorem \ref{theoreme5}, we use this corrector in order to establish the homogenization of equation \eqref{equationepsilon}. Precisely, we show the convergence to zero in $H^1$ of the sequence of rests $R^{\varepsilon} = u^{\varepsilon} - u^{\varepsilon,1}$, where $u^{\varepsilon,1}$ is defined as in \eqref{approximatesequence} but using the adapted corrector of Theorem \ref{theoreme4}. We also precise the convergence rates for this topology. 

\begin{theorem}
\label{theoreme4}
For every $p \in \mathbb{R}^d$, there exists a unique (up to an additive constant)
function $\Tilde{w}_p \in W^{1,r}_{loc}(\mathbb{R}^d)$ such that $w_p = w_{per,p} + \Tilde{w}_p$ is solution to corrector equation \eqref{correcteur}, where $w_{per,p}$ is the unique periodic corrector solution to \eqref{problemeperiodique} and $\nabla \Tilde{w}_p \in \left(\mathcal{B}^r(\mathbb{R}^d)\cap \mathcal{C}^{0,\alpha}(\mathbb{R}^d)\right)^d$. 
\end{theorem}

\begin{remark}
We note that, in the case $r\neq 2$, we are actually not able to prove the uniqueness of a corrector $w_p$ such that $\nabla w_p \in \left(L^2_{per}(\mathbb{R}^d) + \mathcal{B}^r(\mathbb{R}^d)\right)^d$. We are only able to prove the uniqueness (up to an additive constant) of a solution of the form $w_p = w_{per,p} + \Tilde{w}_p$ where $\nabla \Tilde{w}_p \in \mathcal{B}^r(\mathbb{R}^d)^d$. This result is however sufficient to study the homogenization of problem \eqref{equationepsilon} stated in the next theorem and studied in Section \ref{section6Br}.
\end{remark}

\begin{theorem}
\label{theoreme5}
Assume that $\Omega$ is a ${C}^{2,1}$-bounded domain and that $r\neq d$. Let $\displaystyle\Omega_1 \subset \subset \Omega$. We define $ \displaystyle u^{\varepsilon,1} =  u^* + \varepsilon \sum_{i=1}^{d}\partial_i u^{*} w_{e_i}(./\varepsilon)$ where $w_{e_i}$ is the corrector given by Theorem \ref{theoreme4} for $p=e_i$ and $u^*\in H^1(\Omega)$ is the solution to \eqref{homog}. we define 
\begin{equation}
\label{def_nu_r}
    \nu_r = \min \left(1, \frac{d}{r}\right) \in ]0,1],
\end{equation}
and 
\begin{equation}
\label{def_mu_r}
    \mu_r = \left\{ 
   \begin{array}{cc}
      0  & \text{if } \ r<d, \\
      \frac{1}{r}  & \text{else. } 
   \end{array}
    \right.
\end{equation}
Then $R^{\varepsilon} = u^{\varepsilon} - u^{\varepsilon,1}$ satisfies the following estimates : 
\begin{equation}
\label{estimate1Bp}
\|R^{\varepsilon}\|_{L^2(\Omega)} \leq C_1 \left(\log |\varepsilon| \right)^{\mu_r} \varepsilon^{\nu_r} \|f\|_{L^2(\Omega)},
\end{equation}
\begin{equation}
\label{estimate2Bp}
\|\nabla R^{\varepsilon}\|_{L^2(\Omega_1)} \leq C_2 \left(\log |\varepsilon| \right)^{\mu_r} \varepsilon^{\nu_r}  \|f\|_{L^2(\Omega)},
\end{equation}
where $C_1$ and $C_2$ are two positive constants independent of $f$ and $\varepsilon$. 
\end{theorem}

Here, it is important to note that the convergence rate in $H^1$ depends on the values of the ambient dimension $d$ and the exponent $r$. One way to explain this phenomenon, is to remark that the behavior of $R^{\varepsilon}$ in $H^1(\Omega)$ is controlled by the sub-linearity of $w$, that is by the rate of convergence to 0 of the sequence $\varepsilon w(./\varepsilon)$ when $\varepsilon$ tends to 0 (see the proof of Theorem \ref{theoreme5}). Moreover, in Proposition \ref{propsouslineariteBr} and Lemma \ref{lemmeborneBr} established in the sequel we show that the behavior of $w$ at the infinity actually depends on the integrability of $\nabla w$ (and, in a certain sens, on its decrease) far from the points $x_p$, that is, on the value of $r$. 

\begin{remark}
The case $r=d$ is a critical case for the study of the corrector. Indeed as we shall see in Sections \ref{section2Br} and \ref{section6Br}, this phenomenon is closely linked to the critical integrability of the function $|x|^{-1}$ in $L^d(\mathbb{R}^d)$ and, in this case, we do not know if the corrector $w_p$ is necessarily bounded. Since the data are bounded in $L
^{\infty}(\mathbb{R}^d)$, we can however apply the result of the case $r>d$ in order to have the above estimates, in which $\left(\log |\varepsilon| \right)^{\mu_r} \varepsilon^{\nu_r}$ is replaced by $\left(\log |\varepsilon| \right)^{\mu_s} \varepsilon^{\nu_s}$ for every $s>d$. 
\end{remark}

In the sequel of this work, our approach is similar to that of the case $r=2$ and we apply here the following strategy :

\begin{itemize}
    \item[1)] We first study the structure of the space $\mathcal{B}^r(\mathbb{R}^d)$ in Section \ref{section1Br} and we show several properties of its elements. In particular, we claim that the elements of $\mathcal{B}^r(\mathbb{R}^d)$ have an average value equal to zero and we study a property of strict sub-linearity satisfied by the functions with a gradient in this space.  
    \item[2)] In section \ref{section2Br}, we next study the diffusion problem $-\operatorname{div}(a\nabla u) = \operatorname{div}(f)$ when $f \in \mathcal{B}^r(\mathbb{R}^d)^d$ and the coefficient $a$ is periodic. Precisely, we establish the existence of a solution $u$ to \eqref{eqper1} such that $\nabla u$ belongs to $\mathcal{B}^r(\mathbb{R}^d)^d$. To this end, we use the results of M. Avellaneda and F.H. Lin established in \cite{avellaneda1987compactness, avellaneda1991lp} about the asymptotic behavior of the Green function associated with the operator $-\operatorname{div}(a \nabla.)$ on $\mathbb{R}^d$, in order to explicit the gradient of the solution $u$. 
    \item[3)] In section \ref{section4Br}, we generalize these results in the perturbed periodic context when $a$ is a coefficient of the form \eqref{coefficientform}. Here, we first establish the continuity of the operator $\nabla \left(-\operatorname{div}(a \nabla .)\right)^{-1}\operatorname{div}$ from $\mathcal{B}^r(\mathbb{R}^d)^d$ to $\mathcal{B}^r(\mathbb{R}^d)^d$ stated in Lemma \ref{lemme2Br} and, again, we adapt a method presented in \cite{blanc2018correctors} using an argument of connexity as in the proof of Lemma \ref{lemmexistence}.
    \item[4)] We finally apply this result in section \ref{section5Br} in order to prove the existence of the corrector stated in Theorem \ref{theoreme4} and we use its properties in section \ref{section6Br} to establish the convergence rates stated in Theorem \ref{theoreme5}. 
\end{itemize}

In order not to repeat what we have done in the case $r = 2$, we choose here to refer the reader to the corresponding propositions as soon as the proofs are similar and, here,  we only detail the proofs for which the arguments are specific to the case $ r \neq 2 $. In particular, we mainly detail the existence results in the periodic case established in Section \ref{section2Br}, the uniqueness results of Section \ref{section3Br} and the uniform bound satisfied by the corrector in the specific case $r<d$ stated in Proposition \ref{lemmeborneBr}. 

\subsection{Preliminary results}
\label{section1Br}

To start with, we establish here several properties regarding the elements of $\mathcal{B}^r(\mathbb{R}^d)$. As in the case $r=2$, we need to study the behavior of these functions, their average value and the property of sub-linearity, in order to obtain some information satisfied by the corrector $w_p$ (in particular we want to study its decrease at infinity) and to apply it to establish the homogenization of diffusion problem \eqref{equationepsilon} and the convergence rates stated in Theorem~\ref{theoreme5}. 

First of all, in the three following propositions, we naturally generalize the results of the case $r=2$ regarding the structure of the space $\mathcal{B}^r(\mathbb{R}^d)$ : in Proposition \ref{BanachBr} we claim that $\mathcal{B}^r(\mathbb{R}^d)$ is a Banach space, in Proposition \ref{densityBr} we introduce a result of density and, finally, in Proposition \ref{stabilityBr}, we give a result regarding the multiplication between two elements of $\mathcal{B}^r(\mathbb{R}^d)$. The proofs of these propositions can be easily adapted from the proofs of the similar propositions established in Section \ref{Section3} (see Propositions \ref{Banach}, \ref{density} and \ref{stability}). 

\begin{prop}
\label{BanachBr}
The space $\mathcal{B}^r(\mathbb{R}^d)$ equipped with the norm defined by \eqref{normdefBr}, is a Banach space.
\end{prop}

\begin{prop}
\label{densityBr}
Let $\alpha \in ]0,1[$, then $\mathcal{C}^{0,\alpha}(\mathbb{R}^d)\cap \mathcal{B}^r(\mathbb{R}^d)$ is dense in $\left(\mathcal{B}^r(\mathbb{R}^d), \|.\|_{\mathcal{B}^r(\mathbb{R}^d)}\right).$
\end{prop}

\begin{prop}
\label{stabilityBr}
Let $g$ and $h$ be in $\mathcal{B}^r(\mathbb{R}^d) \cap L^{\infty}(\mathbb{R}^d)$. We assume the associated limit $L^r$ function of $g$, denoted by $g_{\infty}$, is in $L^{\infty}(\mathbb{R}^d)$, then $hg \in \mathcal{B}^r(\mathbb{R}^d)$.
\end{prop}

We next claim that the functions of  $\mathcal{B}^r(\mathbb{R}^d)$ have an average value equal to zero in sens \eqref{zeroaverageBr}. As in the case $r=2$, this property is essential and allows us to prove that the homogenized equation \eqref{homog} is actually the same as in the periodic background, that is, when $\Tilde{a} = 0$. 

\begin{prop}
\label{moyenneBr}
Let $u \in \mathcal{B}^r(\mathbb{R}^d)$. Then, for every $x_0 \in \mathbb{R}^d$  :  
\begin{equation}
\label{zeroaverageBr}
\lim_{R\rightarrow\infty}\dfrac{1}{|B_R|}\int_{B_R(x_0)}|u(x)|dx = 0,
\end{equation}
with the following convergence rate : 
\begin{equation}
\label{convergencerateBr}
\dfrac{1}{|B_R|}\int_{B_R(x_0)}|u(x)|dx \leq C \left(\dfrac{\log R}{R^d}\right)^{\frac{1}{r}},   
\end{equation}
where $C>0$ is independent of $R$ and $x_0$. 
\end{prop}

\begin{proof}
We fix $R>0$. Using the Hölder inequality, we have :
\begin{align*}
\dfrac{1}{|B_R|}\int_{B_R(x_0)}|u(x)|dx & \leq \dfrac{1}{|B_R|^{\frac{1}{r}}} \left( \int_{B_R(x_0)}|u(x)|^rdx \right) ^{\frac{1}{r}} \\
& = \dfrac{1}{|B_R|^{\frac{1}{r}}} \left( \sum_{p \in \mathcal{P}}   \int_{\textit{V}_p \cap B_R(x_0)}|u(x)|^rdx \right) ^{\frac{1}{r}}.
\end{align*}
Since the number of $V_p$ such that $B_R(x_0)\cap V_p \neq \emptyset$ is bounded by $\log(R)$ according to Corollary \ref{corollogboundVP}, we obtain : 
\begin{align*}
\dfrac{1}{|B_R|}\int_{B_R(x_0)}|u(x)|dx&\leq \dfrac{\left(\log R\right)^{\frac{1}{r}}}{|B_R|^{\frac{1}{r}}}\sup_p\|u\|_{ L^r(\textit{V}_p)} \leq C(d)   \left(\dfrac{\log(R)}{R^d}\right)^{\frac{1}{r}}\sup_p \|u\|_{ L^r(\textit{V}_p)}.  
\end{align*}
Here, $C(d)$ depends only on the ambient dimension $d$. The last inequality yields \eqref{convergencerateBr} and conclude the proof. 
\end{proof}

\begin{corol}
\label{convergenceLinfinistarBr}
Let $u \in \mathcal{B}^r(\mathbb{R}^d) \cap L^{\infty}(\mathbb{R}^d)$, then $|u(./\varepsilon)|$ is convergent to 0 in the weak*-$L^{\infty}$ topology when $\varepsilon \rightarrow 0$. 
\end{corol}

We also have to study an other fundamental result regarding the strict sublinearity at infinity of a function $u$ such that $\nabla u \in \mathcal{B}^r(\mathbb{R}^d)^d$. Exactly as in the case $r=2$, this result is key for establishing estimates \eqref{estimate1Bp} and \eqref{estimate2Bp}. Indeed, as we shall see in Section \ref{section6Br}, the error $u^{\varepsilon} - u^{\varepsilon,1}$ is actually controlled by the $L
^{\infty}$-norm of the sequence $\varepsilon w_{e_i}(./\varepsilon)$ when $\varepsilon \rightarrow 0$. 

\begin{prop}
\label{propsouslineariteBr}
Let $u \in W^{1,r}_{loc}(\mathbb{R}^d)$ such that $\nabla u \in \left(\mathcal{B}^r(\mathbb{R}^d) \cap L^{\infty}(\mathbb{R}^d)\right)^d$. Then $u$ is strictly sub-linear at infinity and for every $s \in \mathbb{R}$ such that $s >d$ and $s\geq r$, there exists $C>0$ such that for every $x,y \in \mathbb{R}^d$ with $x \neq y$ : 
\begin{equation}
\label{souslineariteBr}
\left|u(x) - u(y)\right| \leq C \left|\log(\left|x-y\right|)\right|^{\frac{1}{s}} \left|x-y \right|^{1-\frac{d}{s}}.
\end{equation}
\end{prop}

The proof of estimate \eqref{souslineariteBr} is an easy adaptation to that of Proposition \ref{propsouslinearite} and the reader can see it for details.  

To conclude this section, we finally introduce a generalization of Proposition \ref{uniformmajoration} in order to obtain an uniform estimate of the integral remainders of the functions in $\mathcal{B}^r(\mathbb{R}^d)$.

\begin{prop}
\label{uniformmajorationBr}
Let $f$ be in $\mathcal{B}^r(\mathbb{R}^d)$ and $f_{\infty}$ the associated limit function in $L^r(\mathbb{R}^d)$. For any $\varepsilon >0$, there exists $R^*>0$ such that for every $R>R^*$ and every $p,q \in \mathcal{P}$ : 
\begin{equation}
    \left( \int_{V_q\cap B_R(2^q)^c} \left|f - \tau_{-p}f_{\infty}\right|^r \right)^{1/r} < \varepsilon,
\end{equation}
where $B_R(2^q)^c$ denotes the set $\mathbb{R}^d \setminus{B_R(2^q)}$. Therefore, we have the following limit : 
\begin{equation}
    \lim_{R \rightarrow \infty} \sup_{\substack{(p,q) \in \mathcal{P} \\ p \neq q}} \left( \int_{V_q\cap B_R(2^q)^c} \left|f - \tau_{-p}f_{\infty}\right|^r \right)^{1/r} = 0.
\end{equation}
\end{prop}

\subsection{Existence results in the periodic problem}
\label{section2Br}

In this section, we next turn to the study of the general equation \eqref{equationref} in a periodic background, that is for $a = a_{per}$ periodic, when $f \in \left(\mathcal{B}^r(\mathbb{R}^d)\right)^d$. In this case, our aim is to use asymptotic behaviors \eqref{estimgreen1}, \eqref{estimgreen3} and \eqref{estimgreen2} of the Green function associated with the operator $-\operatorname{div}(a_{per} \nabla .)$ on $\mathbb{R}^d$. Exactly as in the case $r=2$, we first use the asymptotic behavior of the Green function to explicit a solution such that $\displaystyle \sup_{p \in \mathcal{P}} \|\nabla u \|_{L^r(V_p)} < \infty$ and we next prove that this function satisfies $\nabla u \in \left(\mathcal{B}^r(\mathbb{R}^d)\right)^d$ as soon as $f \in \left(\mathcal{B}^r(\mathbb{R}^d)\right)^d$.

\begin{lemme}
\label{lemmeexistenceperiodiqueBP}
Let $f$ such that $\displaystyle \sup_{p \in \mathcal{P}} \|f\|_{L^r(V_p)} < \infty$, then there exists a function $u$ in $W^{1,r}_{loc}(\mathbb{R}^d)$ solution to \eqref{eqper1} such that $\displaystyle \sup_{p \in \mathcal{P}} \|\nabla u \|_{L^r(V_p)} < \infty$ and we have the following estimate : 
\begin{equation}
\label{estimperBP}
\sup_{p\in \mathcal{P}} \|\nabla u\|_{L^r(\textit{V}_p)} \leq C\sup_{p\in \mathcal{P}} \|f\|_{L^r(\textit{V}_p)},
\end{equation}
where $C>0$ is a constant independent of $f$ and $u$.
In addition, if $r<d$, this solution $u$ can be defined by : 
\begin{equation}
    \label{defuintegralBp}
    u = \int_{\mathbb{R}^d} \nabla_y G_{per}(.,y).f(y) dy.
\end{equation}
\end{lemme}

\begin{proof}
The method used here is similar to that employed for the proof of Lemma \ref{reultatexistenceL2}. We first introduce the Green function $G_{per}$ in order to find a solution to \eqref{eqper1} in $W^{1,r}_{loc}(\mathbb{R}^d)$ and we show this solution satisfies \eqref{estimperBP}.

\textit{Step 1 :  Definition of a solution $u$.}

 In the sequel the letter $C$ denotes a generic constant that may change for one line to another. To start with, for each $q \in \mathcal{P}$, we define the function : 
\begin{equation}
\label{defukBP}
    u_q = \int_{\mathbb{R}^d} \nabla_y G_{per}(.,y) f(y)1_{V_q}(y) dy. 
\end{equation}
The results of \cite{avellaneda1991lp} ensure this function is a solution in $H^1_{loc}(\mathbb{R}^d)$ to
\begin{equation}
    - \operatorname{div}(a_{per}\nabla u_q) = \operatorname{div}(f1_{V_q}) \quad \text{in } \mathbb{R}^d,
\end{equation}
such that $\nabla u_q \in \left(L^r(\mathbb{R}^d)\right)^d$. In particular, \cite[Theorem A]{avellaneda1991lp} gives the existence of a constant $C>0$ independent of $q$ such that 
\begin{equation}
    \|\nabla u_q\|_{L^r(\mathbb{R}^d)} \leq C \|\nabla f 1_{V_q}\|_{L^r(\mathbb{R}^d)}.
\end{equation}
For every $N \in \mathbb{N^*}$, we next define the two following sequences :
\begin{equation}
\label{SeriesUNBP}
 U_N = \sum_{q \in \mathcal{P}, \ |q|\leq N}  u_q,   
\end{equation}
and
\begin{equation}
\label{SeriesSNBP}
S_N = \nabla U_N = \sum_{q \in \mathcal{P}, \ |q|\leq N} \nabla u_q.    
\end{equation}
By linearity, we have that, for every $N \in \mathbb{N}^*$, $U_N$ is a solution to 
\begin{equation}
\label{equation_U_N_Br}
    - \operatorname{div}(a_{per} \nabla U_N) = \operatorname{div}\left(f \sum_{q \in \mathcal{P}, \ |q|\leq N} 1_{V_q}\right) \quad \text{in } \mathbb{R}^d.
\end{equation}
Here, our aim is to show that $S_N$ admits a limit in $\left(L^r_{loc}(\mathbb{R}^d)\right)^d$ and its limit is the gradient of a solution $u$ to \eqref{eqper1}. For every $q \in \mathcal{P}$, we introduce a set $W_q$ and five constants $C_1$, $C_2$, $C_3$, $C_4$ and $C_5$ independent of $q$ and defined by Proposition \ref{openconstruction} such that :  
\begin{enumerate}[label=(\roman*)]
\item $\textit{V}_q \subset{W}_q$, \label{i)Br}  
\item $Diam(W_q) \leq C_1 2^{|q|}$, and $D(\textit{V}_q,\partial W_q) \geq C_2 2^{|q|},$ \label{ii)Br} 
\item  $\forall s \in \mathcal{P} \setminus{\{q\}}$, $Dist(2^s, W_q) \geq C_3 2^{|q|}$,  \label{iii)Br} 
\item $\sharp \left \{ s \in \mathcal{P} \middle| V_s \cap W_q \neq \emptyset \right\} \leq C_4$, \label{iv)Br} 
\item $\forall s \in \mathcal{P} \setminus{\{q\}}$, $D\left(V_q, V_s\setminus{W_q}\right) \geq C_5 2^{|s|}.$ 
\label{v)Br}
\end{enumerate}

For every $q \in \mathcal{P}$ such that $W_p \cap~V_q =~\emptyset$, we use the asymptotic behavior estimate \eqref{estimgreen2} of the Green function $G_{per}$ and the Holder inequality to obtain : 
\begin{align*}
     \|\nabla u_q\|_{L^r(V_p)} & \leq C \sup_{q \in \mathcal{P}} \|f \|_{L^r(V_q)}  \left(\int_{V_p} \left(\int_{V_q}  \frac{1}{|x-y|^{r'd}}dy\right)^{r/r'} \ dx \right)^{1/r},
\end{align*}
where $\displaystyle r' = r/(r-1)$. We next use Property \ref{v)Br} of $W_p$ and the fact that $\left|V_q\right|\leq C 2^{|q|}$ (consequence of Propositions \ref{volumeVx} and \ref{norm2p}). We therefore have :
\begin{align*}
   \|\nabla u_q\|_{L^r(V_p)}  & \leq C \sup_{q \in \mathcal{P}} \|f \|_{L^r(V_q)}  \left(\int_{V_p}  \left(\frac{\left|V_q\right|}{2^{|q|r'd}}\right)^{r/r'} \ dx \right)^{1/r} \\
    & \leq C \sup_{q \in \mathcal{P}} \|f \|_{L^r(V_q)} \left|V_p\right|^{1/r}  \frac{1}{2^{|q|d(r'-1)/r'}}.
\end{align*}

Since $r'/(r'-1) = r$, we conclude that : 
\begin{equation}
\label{estimeenablau_qBp}
  \|\nabla u_q\|_{L^r(V_p)} \leq C \sup_{q \in \mathcal{P}} \|f \|_{L^r(V_q)}\left|V_p\right|^{1/r}  \frac{1}{2^{|q|d/r}}.
\end{equation}
In addition, Property \ref{iv)Br} ensures that the number of $q$ such that $W_p \cap V_q \neq \emptyset$ is bounded uniformly with respect to $q$. Since the sequence $\displaystyle \left(\frac{1}{2^{|q|d/r}}\right)_{q \in \mathbb{Z}^d}$ is summable, we finally deduce that : 
$$ \lim_{N \rightarrow \infty} \sum_{q \in \mathcal{P}, \ |q|\leq N}  \|\nabla u_q\|_{L^r(V_p)} < \infty.$$
Therefore, we have established the absolute convergence of $S_N$ in $\left(L^r(V_p)\right)^d$ for every $p \in \mathcal{P}$, that is the convergence of $S_N$ in $\left(L^r_{loc}(\mathbb{R}^d)\right)^d$. We denote by $T$ the limit of $S_N$ in $\left(L^r_{loc}(\mathbb{R}^d)\right)^d$. Letting $N$ go to the infinity in \eqref{equation_U_N_Br}, we obtain : 
$$- \operatorname{div}\left(a_{per} T \right) = \operatorname{div}(f) \quad \text{in }\mathbb{R}^d.$$
We next prove that there exists $u$ in $\mathcal{D}'(\mathbb{R}^d)$ such that $T = \nabla u$ showing that, for every $i,j$ in $\{1,...,d\}$, we have $\partial_i T_j = \partial_j T_j$ in $\mathcal{D}'(\mathbb{R}^d)$. We denote by $\langle .,. \rangle_{\mathcal{D}, \mathcal{D}'}$ the duality bracket in $\mathcal{D}'(\mathbb{R}^d)$ and we have : 
\begin{align*}
\langle \partial_i T_j , \varphi \rangle_{\mathcal{D}, \mathcal{D}'} & = - \langle  T_j , \partial_i \varphi \rangle_{\mathcal{D}, \mathcal{D}'} \\
& = - \sum_{q \in \mathcal{P}} \langle \partial_j u_q, \partial_i \varphi \rangle_{\mathcal{D}, \mathcal{D}'}.
\end{align*}
The last equality is justified by the normal convergence of $S_N$ in $L^r_{loc}(\mathbb{R}^d)$. We next use the Schwarz lemma and we have : 
\begin{align*}
\langle \partial_i T_j , \varphi \rangle_{\mathcal{D}, \mathcal{D}'} & = \sum_{q \in \mathcal{P}} \langle \partial_i \partial_j u_q,  \varphi \rangle_{\mathcal{D}, \mathcal{D}'} = \sum_{q \in \mathcal{P}} \langle \partial_j \partial_i u_q,  \varphi \rangle_{\mathcal{D}, \mathcal{D}'} \\
& = - \sum_{q \in \mathcal{P}} \langle  \partial_i u_q, \partial_j \varphi \rangle_{\mathcal{D}, \mathcal{D}'} = \langle \partial_j T_i , \varphi \rangle_{\mathcal{D}, \mathcal{D}'}.
\end{align*}
Finally, we obtain that $\partial_i T_j = \partial_j T_j$ and therefore, there exists $u \in \mathcal{D}'(\mathbb{R}^d)$ such that $T = \nabla u$. In addition, $u$ is a solution to \eqref{eqper1} in $\mathcal{D}'(\mathbb{R}^d)$. Finally, since $\nabla u$ belongs to $L^r_{loc}(\mathbb{R}^d)^d$, the result of \cite[corollary 2.1]{deny1954espaces} ensures that $u \in W^{1,r}_{loc}(\mathbb{R}^d)$.  

\textit{Step 2 :  The case $r<d$}

 In this step only, we assume that $r<d$. Our aim here is to show that the sequence $U_N$ is also convergent in $L^r_{loc}(\mathbb{R}^d)$. We use again the behavior of the Green function and for every $q \in \mathcal{P}$ such that $W_p \cap~V_q =~\emptyset$, we have : 
 
 \begin{align*}
     \|u_q\|_{L^r(V_p)} & \leq C \sup_{q \in \mathcal{P}} \|f \|_{L^r(V_q)}  \left(\int_{V_p} \left(\int_{V_q}  \frac{1}{|x-y|^{r'(d-1)}}dy\right)^{r/r'} \ dx \right)^{1/r} \\
    & \leq C \sup_{q \in \mathcal{P}} \|f \|_{L^r(V_q)} \left|V_p\right|^{1/r}  \frac{1}{2^{|q|((r'-1)d/r'-1)}}.
\end{align*}
In addition, we remark that $(r'-1)/r'=1/r$ and we obtain :
$$\|u_q\|_{L^r(V_p)}  \leq C \sup_{q \in \mathcal{P}} \|f \|_{L^r(V_q)} \left|V_p\right|^{1/r}  \frac{1}{2^{|q|(d/r-1)}}.$$
Since $r<d$, we have $d/r >1$ and we deduce that :
$$ \lim_{N \rightarrow \infty} \sum_{q \in \mathcal{P}, \ |q|\leq N}  \| u_q\|_{L^r(V_p)} < \infty.$$
Therefore, we have established the convergence of $U_N$ in $L^p_{loc}(\mathbb{R}^d)$. In addition, using the uniqueness of the limit in $\mathcal{D}'(\mathbb{R}^d)$, we deduce that the function $u$ defined in the first step is actually equal to $\displaystyle \lim_{N \rightarrow \infty} U_N = \int_{\mathbb{R}^d} \nabla_y G_{per}(.,y) f(y) dy$.

\textit{Step 3 :  Proof of estimate \eqref{estimperBP}}

Next, we have to establish estimate \eqref{estimperBP}. Let $p$ be in $\mathcal{P}$, we start by splitting $\nabla u$ in two parts : 
\begin{align*}
\nabla u & = \displaystyle \sum_{\substack{q \in \mathcal{P},\\ V_q \cap W_p \neq \emptyset}} \nabla u_q + \sum_{\substack{q \in \mathcal{P},\\ V_q \cap W_p = \emptyset}} \nabla u_q  \\
& =  \nabla I_{1,p}(x) + \nabla I_{2,p}(x).
\end{align*}
We denote $\mathcal{P}_p$ the set of indices $q\in \mathcal{P}$ such that $V_q \cap W_p \neq \emptyset$. First of all, since $\mathcal{P}_p$ is finite according to Property \ref{iv)Br}, the function $\displaystyle f \sum_{\substack{q \in \mathcal{P}_{p}}} 1_{V_q}$ belongs to $\left(L^{r}(\mathbb{R}^d)\right)^d$. The results of \cite{avellaneda1991lp} therefore ensure that $I_{1,p}$ is actually a solution in $\mathcal{D}'(\mathbb{R}^d)$ to : 
\begin{equation}
    -\operatorname{div}\left(a_{per} \nabla I_{1,p}\right) = \operatorname{div}\left(f \sum_{\substack{q \in \mathcal{P}_{p}}} 1_{V_q}\right) \quad \text{in } \mathbb{R}^d, 
\end{equation}
such that $\nabla I_{1,p} \in L^r(\mathbb{R}^d)$ and there exists a constant $C>0$ independent of $p$ and $f$ such that : 
\begin{equation}
  \left\|\nabla I_{1,p}\right\|_{L^r(V_p)} \leq  \|\nabla I_{1,p}\|_{L^r(\mathbb{R}^d)} \leq C \left\| f \sum_{\substack{q \in \mathcal{P}_{p}}} 1_{V_q} \right\|_{L^r(\mathbb{R}^d)}.
\end{equation}
Finally, we use a triangle inequality and Property \ref{iv)Br} of $W_p$ to conclude there exists $C>0$ independent of $p$
such that : 
\begin{equation}
\label{estimJ_1Bp}
   \|\nabla I_{1,p} \|_{L^r(V_p)} \leq C \sup_{q \in \mathcal{P}} \|f\|_{L^r(V_q)}.
\end{equation}
Next, for $\nabla I_{2,p}$ we proceed exactly as in the proof of lemma \ref{lemmeexistenceperiodique1} (see the proof of estimate \eqref{estimunifnablaI2p} for details) using Holder inequalities, the asymptotic behavior of $G_{per}$ and properties of $W_p$ in order to show the existence of $C>0$ such that : 
$$ \| \nabla I_{2,p} \|_{L^r(V_p)} \leq  C \sup_{q \in \mathcal{P}} \|f \|_{L^r(V_q)}\left|V_p\right|^{1/r}  \frac{1}{2^{|p|d/r}}.$$
In addition, since $\left|V_p\right| \leq C 2^{|p|d}$, we obtain : 
\begin{equation}
\label{estimJ_2Bp}
  \| \nabla I_{2,p} \|_{L^r(V_p)} \leq  C \sup_{q \in \mathcal{P}} \|f \|_{L^r(V_q)}.  
\end{equation}
Finally we use a triangle inequality and estimates \eqref{estimJ_1Bp} and \eqref{estimJ_2Bp} to conclude that : 
$$ \| \nabla u \|_{L^r(V_p)} \leq  \| \nabla I_{1,p} \|_{L^r(V_p)} + \| \nabla I_{2,p} \|_{L^r(V_p)} \leq C \sup_{q \in \mathcal{P}} \|f \|_{L^r(V_q)}.$$
We have established \eqref{estimperBP}.
\end{proof}

Now that we have defined a particular solution $u$ in $W^{1,r}_{loc}(\mathbb{R}^d)$ to \eqref{eqper1}, we have to show that the function $\nabla u$ belongs to $\left(\mathcal{B}^r(\mathbb{R}^d)\right)^d$ as soon as $f \in \left(\mathcal{B}^r(\mathbb{R}^d)\right)^d$. This result is given in the following lemma.

\begin{lemme}
\label{existenceperBP2}
Let $f\in \left(\mathcal{B}^r(\mathbb{R}^d)\right)^d$, then the function $u$ defined in Lemma \ref{lemmeexistenceperiodiqueBP} satisfies $\nabla u \in \left(\mathcal{B}^r(\mathbb{R}^d)\right)^d$. 
\end{lemme}

The proof of this result being extremely similar to the proof of Lemma \ref{existenceper} in the case $r=2$, we choose here not to detail it. The idea is to show that $\nabla u$ belongs to $\left(\mathcal{B}^r(\mathbb{R}^d)\right)^d$. In order to find its limit function in $\left(L^{r}(\mathbb{R}^d)\right)^d$ (denoted by $\nabla u_{\infty}$), it is possible to define : 
$$ \nabla u_{\infty} = \int_{\mathbb{R}^d} \nabla_y G_{per}(.,y) f_{\infty}(y) dy,$$
and to prove the convergence of $\nabla u$ to $\nabla u_{\infty}$ in the sense of definition \eqref{spacedefBr}. To this aim, the main idea is to use Proposition \ref{uniformmajorationBr} in order to prove several estimates similar to estimates \eqref{estimunifJ1} and \eqref{estimunifJ2} established in the proof of Lemma \ref{existenceper}. 

\subsection{Uniqueness results}
\label{section3Br}

We next turn to the uniqueness of solution $u$ to \eqref{equationref} such that $\nabla u \in \left(\mathcal{B}^r(\mathbb{R}^d)\right)^d$. In our approach, such a result of uniqueness is actually essential to prove the existence of a corrector stated in Theorem \ref{theoreme4}. Indeed, as we shall see in the proof of Lemma \ref{lemme2Br}, the uniqueness is key to establish the continuity of the operator $\nabla \left(- \operatorname{div}(a \nabla .) \right)
^{-1} \operatorname{div}$ from $\left(\mathcal{B}^r(\mathbb{R}^d)\right)^d$ to $\left(\mathcal{B}^r(\mathbb{R}^d)\right)^d$. By contrast with the uniqueness results established in Section \ref{Section4} for the space $\mathcal{B}^2(\mathbb{R}^d)$, the topology of $L^r$ used to define the space $\mathcal{B}^r(\mathbb{R}^d)$ for $r\neq 2$ does not allow us to use some "Hilbertian" techniques here. In order to overcome this difficulty, our idea consists to mainly use both the structure \eqref{coefficientform} of the coefficient $a$ and the results established in the previous section in the case $a=a_{per}$. To this aim, we need to first introduce a uniqueness result in the case $a=a_{per}$ given in Lemma \ref{uniquess_periodic_Br} in order to apply it in the proof of Lemma \ref{lemme2Br} regarding the general case when $a\neq a_{per}$. 

To start with, we introduce two technical lemmas related to the periodic case when $\Tilde{a}=0$. The first lemma is a result of uniqueness and the second establishes a property similar to a Gagliardo-Nirenberg-Sobolev inequality satisfied by the solutions to equation \eqref{eqper1}. 

\begin{lemme}
\label{uniquess_periodic_Br}
Let $r>1$ and $u$ be a solution in $W^{1,r}_{loc}(\mathbb{R}^d)$ to : 
\begin{equation}
    \label{equationhomogeneperiodique}
    - \operatorname{div}(a_{per} \nabla u) = 0 \quad \text{in } \mathbb{R}^d,
\end{equation}
such that : 
\begin{equation}
    \displaystyle \sup_{p \in \mathcal{P}} \int_{V_p} |\nabla u|^r < \infty.
\end{equation}
Then $\nabla u = 0$. 
\end{lemme}

\begin{proof}
First of all, since $u$ is a solution in $W^{1,r}_{loc}(\mathbb{R}^d)$ to \eqref{equationhomogeneperiodique} the result of \cite[Theorem 1]{brezis2008conjecture} ensures that $u$ actually belongs to $W^{1,s}_{loc}(\mathbb{R}^d)$ for every $s< \infty$ and is therefore locally Holder continuous according to \cite[Section 5]{moser1961harnack}. Next, for every $k \in \mathbb{Z}^d$, we translate equation \eqref{equationhomogeneperiodique} by $k$ and, using the periodicity of $a_{per}$, we obtain that $u - \tau_k u$ is a solution in $W^{1,r}_{loc}(\mathbb{R}^d)$ to : 
$$-\operatorname{div}(a_{per} \nabla(u - \tau_k u)) = 0 \quad \text{in } \mathbb{R}^d.$$
We claim that for every $i \in \{1,...,d\}$, we have $ \displaystyle \sup_{p \in \mathcal{P}} \|u - \tau_{e_i} u\|_{L^r(V_p)} < \infty$. Indeed, for every $x \in \mathbb{R}^d$, we have : 
$$ u(x + e_i) - u(x)   =  \int_{0}^1 \nabla u(x + te_i).e_i dt .$$
For every $p\in \mathcal{P}$, we use the H\"older inequality and we obtain : 
\begin{align*}
   \left(\int_{V_p} \left| u(x + e_i) - u(x) \right|^r dx\right)^{1/r} & \leq C_1  \left(\int_{V_p}  \int_{0}^1 \left|\nabla u(x + te_i)\right|^r dt dx \right)^{1/r} \\
    & \stackrel{Fubini}{=} C_1 \left(\int_{0}^1  \int_{V_p} \left|\nabla u(x + te_i)\right|^r dx  dt\right)^{1/r}\\
    & \leq C_2 \sup_{p \in \mathcal{P}}  \left(\int_{\textit{V}_p}|\nabla u|^r\right)^{1/r},
\end{align*}
where $C_1$ and $C_2$ are two positive constants which depend only on $d$. Taking the supremum over all $p\in \mathcal{P}$, we obtain the expected result. We can also easily show that, for every $x_0 \in \mathbb{R}^d$, we have 
\begin{equation}
    \|u - \tau_{e_i} u\|_{W^{1,r}(B_2(x_0))} \leq C_3 \sup_{p \in \mathcal{P}} \left(\|u - \tau_{e_i} u\|_{L^r(V_p)} + \|\nabla u - \tau_{e_i} \nabla u\|_{L^r(V_p)}\right),
\end{equation}
where $C_3>0$ does not depend on $x_0$.  
We use again \cite[Theorem 1]{brezis2008conjecture} and we obtain the existence of a constant $C>0$ such that for every $x_0 \in \mathbb{R}^d$ :
\begin{equation}
    \| u - \tau_{e_i} u \|_{H^1(B_1(x_0))} \leq C \|u - \tau_{e_i} u\|_{W^{1,r}(B_2(x_0))}.
\end{equation}
Here, $C$ only depends on the ellipticity constant of $a_{per}$, on $\|a\|_{\mathcal{C}^{0,\alpha}(\mathbb{R}^d)}$, on $d$ and on $r$. Finally the De Giorgi-Nash inequality (see for example \cite[Theorem 8.13 p.176]{MR3099262}) ensures the existence of a constant $C>0$ independent of $x_0$ such that : 
$$\|u - \tau_{e_i} u\|_{L^{\infty}(B_2(x_0))} \leq C \|u - \tau_{e_i} u\|_{L^2(B_2(x_0))},$$
and we conclude that $u - \tau_{e_i} u$ is bounded in $L^{\infty}(\mathbb{R}^d)$. The results of \cite[Section 6]{moser1961harnack} therefore ensure that $u - \tau_{e_i} u$ is constant. 

We next denote by $C_i$ the constant such that $u(x+e_i) - u(x) = C_i$ for every $x\in \mathbb{R}^d$ and we fix $C = (C_i)_{i \in \{1,..,d\}}$. We have that the function $w$ defined by $w(x) = u(x) - C.x$ is periodic and, therefore, $\nabla u = \nabla w + C$ is also periodic. Since  $\displaystyle \sup_{p \in \mathcal{P}}  \int_{\textit{V}_p}|\nabla u|^r < \infty$, we finaly conclude that $\nabla u = 0$. 
\end{proof}

\begin{lemme}
\label{GNS_periodic_Br}
Assume $r<d$ and let $f\in L^r_{loc}(\mathbb{R}^d)^d$ such that $\displaystyle \sup_{p \in \mathcal{P}}  \int_{\textit{V}_p}|f|^r < \infty$. We define $\displaystyle r^* = \frac{dr}{r-d}$. Then, the function $u$ defined by \eqref{defuintegralBp} satisfies $\displaystyle \sup_{p \in \mathcal{P}}  \int_{\textit{V}_p}|u|^{r^*} < \infty$.
\end{lemme}

\begin{proof}
We fix $p \in \mathcal{P}$. We begin by splitting $u$ in two parts such that for every $x \in V_p$, we have : 
\begin{align*}
u(x) & = \int_{W_p} \nabla_y G_{per}(x,y).f(y) dy + \sum_{q \in \mathcal{P}} \int_{V_q \setminus{W_p}} \nabla_y G_{per}(x,y).f(y) dy \\
& = I_{1,p}(x) + I_{2,p}(x),
\end{align*}
Where $W_p$ is defined by Proposition \ref{openconstruction}.
First of all, we use estimate \eqref{estimgreen1} and we have : 
$$|I_{1,p}(x)| \leq C \int_{\mathbb{R}^d} \frac{1}{|x-y|^{d-1}} |f(y)| 1_{W_p}(y)
dy.$$
and 
$$|I_{2,p}(x)| \leq C \sum_{q \in \mathcal{P}} \int_{V_q \setminus{W_p}}  \frac{1}{|x-y|^{d-1}}|f(y)| dy.$$

Next, since $\displaystyle 1 + \frac{1}{r^*} = \frac{d-1}{d} + \frac{1}{r}$, using the Hardy-Littlewood-Sobolev inequality (see for instance \cite[Theorem 7.25 p. 162]{MR3099262}), we obtain that $I_{1,p}$ belongs to $L^{r^*}(\mathbb{R}^d)$ and there exists a constant $C>0$ independent of $p$ such that :
 $$ \|I_{1,p}\|_{L^{r^*}(\mathbb{R}^d)} \leq C \|f 1_{W_p}\|_{L^{r}(\mathbb{R}^d)},$$
that is, using the property \ref{iv)Br} of $W_p$, there exists $C>0$ such that :
\begin{equation}
\label{estimGNSBr1}
    \|I_{1,p}\|_{L^{r^*}(V_p)} \leq C \sup_{q \in \mathcal{P}} \|f\|_{L^r(V_q)}.
\end{equation}
Next, using Holder inequalities, we repeat the different steps of the proof of Lemma \ref{lemmeexistenceperiodiqueBP} (see in particular the step 2 regarding the case $r<d)$ and we obtain for every $x$ in $V_p$ : 
\begin{align*}
    |I_{2,p}(x)| & \leq C \sum_{q \in \mathcal{P}} \left(\int_{V_q \setminus{W_p}}  \frac{1}{|x-y|^{r'(d-1)}}dy\right)^{1/r'} \left(\int_{V_q \setminus{W_p}} |f(y)|^{r} dy\right)^{1/r} \\
    & \leq C \sup_{q \in \mathcal{P}} \|f\|_{L^r(V_q)} \frac{1}{2^{|p|(d/r-1)}}.
\end{align*}
Since, $r<d$, we have $d/r>1$ and therefore : 
$$ |I_{2,p}(x)| \leq \sup_{q \in \mathcal{P}} \|f\|_{L^r(V_q)} \frac{1}{2^{|p|(d/r-1)}}.$$
Since $\left|V_p\right| \leq C 2^{|p|d}$, we deduce that : 
\begin{equation*}
    \|I_{2,p}\|^{r^*}_{L^{r^*}(V_p)} \leq C \sup_{q \in \mathcal{P}} \|f\|_{L^r(V_q)}^{r^*} \frac{1}{2^{|p|r^*(d/r-1)}} \left|V_p\right|  \leq  C \sup_{q \in \mathcal{P}} \|f\|_{L^r(V_q)}^{r^*} \frac{2^{|p|d}}{2^{|p|r^*(d/r-1)}}.
\end{equation*}
Next, we remark that $\displaystyle r^* \left(\frac{d}{r} - 1\right)= \frac{dr}{r-d}\left(\frac{d}{r} - 1\right) = d$ and we deduce that :
\begin{equation}
\label{estimGNSBr2}
    \|I_{2,p}\|_{L^{r^*}(V_p)} \leq C \sup_{q \in \mathcal{P}} \|f\|_{L^r(V_q)}.
\end{equation}
Using a triangle inequality and estimates \eqref{estimGNSBr1} and \eqref{estimGNSBr2}, we finally conclude that for every $p \in \mathcal{P}$, we have : 
$$  \|u\|_{L^{r^*}(V_p)} \leq \|I_{1,p}\|_{L^{r^*}(V_p)} + \|I_{2,p}\|_{L^{r^*}(V_p)} \leq  C \sup_{q \in \mathcal{P}} \|f\|_{L^r(V_q)}.$$
We conclude the proof considering the supremum over all $p\in \mathcal{P}$.
\end{proof}

Using Lemmas \ref{uniquess_periodic_Br} and \ref{GNS_periodic_Br}, we are now able to establish the uniqueness of a solution $u$ to \eqref{equationref} such that $\nabla u \in \left(\mathcal{B}^r(\mathbb{R}^d)\right)^d$ for every $1<r<\infty$. 

\begin{lemme}
\label{uniquenesslemmaBr}
Let $r> 1$ and $u \in W^{1,r}_{loc}(\mathbb{R}^d)$ be a solution in the sense of distribution to : 
\begin{equation}
\label{equationhomogBr}
 -\operatorname{div}(a\nabla u) = 0 \quad\text{in } \mathbb{R}^d,  
\end{equation}
such that $\displaystyle \sup_{p \in \mathcal{P}}  \int_{\textit{V}_p}|\nabla u|^r < \infty$. Then $\nabla u$ = 0.
\end{lemme}

\begin{proof}
We first assume that $r>2$ and we remark that equation \eqref{equationhomogBr} is equivalent to : 
\begin{equation}
\label{equationmodifieeBP}
    - \operatorname{div}\left(a_{per} \nabla u\right) = \operatorname{div}(\Tilde{a}\nabla u) \quad \text{in } \mathbb{R}^d.
\end{equation}
We denote $f = \Tilde{a}\nabla u$. We know that $\Tilde{a}$ belongs to $L^{\infty}(\mathbb{R}^d)$ and we therefore obtain that $\displaystyle \sup_{p \in \mathcal{P}} \|f\|_{L^r(V_p)} < \infty$. The uniqueness result of Lemma \ref{uniquess_periodic_Br} ensures that $u$ is actually the solution defined in Lemma \ref{lemmeexistenceperiodiqueBP} (up to an additive constant). In addition, since $\displaystyle \sup_{p \in \mathcal{P}} \|\Tilde{a}\|_{L^r(V_p)}  <~\infty$, we use the Cauchy-Schwarz inequality and we obtain that 
$$\displaystyle \sup_{p \in \mathcal{P}} \|f\|_{L^{r/2}(V_p)} <\infty.$$
Lemma \ref{lemmeexistenceperiodiqueBP} therefore ensures that $\nabla u$ is such that $\displaystyle \sup_{p \in \mathcal{P}} \|\nabla u\|_{L^{r/2}(V_p)} <~\infty$. We can iterate this argument in order to obtain that $ \displaystyle \sup_{p \in \mathcal{P}} \|\nabla u\|_{L^{r_n}(V_p)} < \infty$ where $\displaystyle \frac{1}{r_n} = \frac{1}{r} + \frac{n}{r}$ for every $n\in \mathbb{N}
^*$ such that $\displaystyle \frac{r}{n}>1$. We have assumed $r >2$ and, it is therefore always possible to find $n\in \mathbb{N}$ such that $1\leq r_n \leq 2$. Thus, since $r_n \leq 2 \leq r$, we deduce that $ \displaystyle \sup_{p \in \mathcal{P}} \|\nabla u\|_{L^{2}(V_p)} < \infty$ by an interpolation result. We finally conclude that $\nabla u = 0$ using the result of Lemma \ref{lemme1} in the case $r=2$. 

We next assume that $r\leq 2$. We know that $u$ is a solution to \eqref{equationmodifieeBP}. Since $r\leq2<d$, $u$ is actually defined by \eqref{defuintegralBp} (up to an additive constant) where $f = \Tilde{a} \nabla u$. Lemma \ref{GNS_periodic_Br} ensures that $\displaystyle \sup_{p \in \mathcal{P}}  \int_{\textit{V}_p}|u|^{r^*} < \infty$. 
Now, since $r^*\geq r$, there exists a constant $C>0$ such that for every $x_0 \in \mathbb{R}^d$, we have : 
$$\|u\|_{L^r(B_2(x_0))} \leq C \|u\|_{L^{r^*}(B_2(x_0))},$$
and the properties of the cells $V_p$ ensure the existence of a constant $C_1>0$ independent of $x_0$ such that  
$$\|u\|_{L^{r^*}(B_2(x_0))} \leq C_1 \sup_{p \in \mathcal{P}}  \int_{\textit{V}_p}|u|^{r^*}.$$
We deduce the existence of $C_2>0$ independent of $x_0$ such that 
\begin{equation}
    \|u\|_{W^{1,r}(B_2(x_0))} \leq  C_2 \left(\sup_{p \in \mathcal{P}}  \left(\int_{\textit{V}_p}|u|^{r^*}\right)^{1/r^*} + \sup_{p \in \mathcal{P}}  \left(\int_{\textit{V}_p}|\nabla u|^{r}\right)^{1/r} \right).
\end{equation}
Therefore, since $u$ is solution to \eqref{equationhomogBr}, we repeat the method of the proof of uniqueness of Lemma \ref{uniquess_periodic_Br} using the result of \cite[Theorem 1]{brezis2008conjecture} and we obtain the existence of a constant $C>0$ such that for every $x_0$ in $\mathbb{R}
^d$, we have 
\begin{equation}
    \| u\|_{H^1(B_1(x_0))} \leq C \left(\sup_{p \in \mathcal{P}}  \left(\int_{\textit{V}_p}|u|^{r^*}\right)^{1/r^*} + \sup_{p \in \mathcal{P}}  \left(\int_{\textit{V}_p}|\nabla u|^{r}\right)^{1/r} \right).
\end{equation}
Using the De Giorgi-Nash inequality, we finally obtain that $u$ belongs to $L^{\infty}(\mathbb{R}^d)$ and, according to \cite[Section 6]{moser1961harnack}, $u$ is constant. Finally we have $\nabla u =0$.
\end{proof}

\subsection{Existence results in the general problem}
\label{section4Br}

In this section, we conclude the study of the equation in the general case (when $\Tilde{a}\neq0$) showing the existence of a solution $u$ to \eqref{equationref} such that $\nabla u \in \left(\mathcal{B}^r(\mathbb{R}^d)\right)^d$. To this end, we first prove in Lemma \ref{lemme2Br} the continuity of the operator $\nabla\left(-\operatorname{div}a\nabla\right)^{-1}\operatorname{div}$ from $\left(\mathcal{B}^r(\mathbb{R}^d)\right)^d$ to $\left(\mathcal{B}^r(\mathbb{R}^d)\right)^d$. We next use this property in order to generalize the existence result of the periodic context (that is when $\Tilde{a}$ = 0) applying the same argument as in the case $r=2$ used in Lemma \ref{lemmexistence}.  

\begin{lemme}
\label{lemme2Br}
There exists a constant $C>0$ such that for every $f \in  \left(\mathcal{B}^r(\mathbb{R}^d)\right)^d$ and $u$ solution in $\mathcal{D}'(\mathbb{R}^d)$ to \eqref{equationref} 
with $\nabla u \in \left(\mathcal{B}^r(\mathbb{R}^d)\right)^d$, we have the following estimate :
\begin{equation}
\label{aprioriestimateBr}
 \|\nabla u\|_{\mathcal{B}^r(\mathbb{R}^d)}  \leq C  \|f\|_{\mathcal{B}^r(\mathbb{R}^d)}. 
\end{equation}
\end{lemme}

\begin{proof}
We follow here the same pattern as the proof of Lemma \ref{lemme2}, that is we argue by contradiction using a compactness-concentration method. To this aim, we assume the existence of a sequence $f_n$ in $\left(\mathcal{B}^r(\mathbb{R}^d)\right)^d$ and an associated sequence of solutions $u_n$ such that $\nabla u_n$ is in $\left(\mathcal{B}^r(\mathbb{R}^d)\right)^d$ and : 
\begin{equation}
\label{9Bp}
-\operatorname{div}((a_{\textit{per}}+\Tilde{a})\nabla u_n) = \operatorname{div}(f_n),
\end{equation}
\begin{equation}
\label{10Bp}
\lim_{n \rightarrow \infty}\|f_n\|_{\mathcal{B}^r(\mathbb{R}^d)} =0,  
\end{equation}
\begin{equation}
\label{11Bp}
\forall n \in \mathbb{N}, \quad \|\nabla u_n\|_{\mathcal{B}^r(\mathbb{R}^d)} = 1.   
\end{equation}
We only give here the main steps of this proof and we refer the reader to the proof of Lemma \ref{lemme2} for details.

\textit{Step 1 : Compactness-concentration method.} Using a property of the supremum, we can find a sequence $x_n$ such that for every $n\in \mathbb{N}$, we have 
$$ \| \nabla u_n\|_{L^r_{\textit{unif}}} \geq \| \nabla u_n \|_{L^r(B_1(x_n))} \geq \| \nabla u_n\|_{L^r_{\textit{unif}}} - \dfrac{1}{n}.$$
In the spirit of the compactness-concentration method, we denote by $\Bar{u}_n = \tau_{x_n}u_n$, $\Bar{f}_n = \tau_{x_n}f_n$ and $\Bar{\Tilde{a}}_n = \tau_{x_n}\Tilde{a}$ and we have for every $n \in \mathbb{N}$ : 
\begin{equation}
\label{limunifBp}
\| \nabla u_n\|_{L^r_{\textit{unif}}} \geq \| \nabla \Bar{u}_n \|_{L^r(B_1)} \geq \| \nabla u_n\|_{L^r_{\textit{unif}}} - \dfrac{1}{n}.
\end{equation} 
In addition, for every $n \in \mathbb{N}$, the sequence $ \Bar{u}_n$ satisfies : 
\begin{equation}
\label{translated_equation_estimBp}
    - \operatorname{div}\left( \Bar{a}_n \nabla \Bar{u}_n \right) = \operatorname{div}\left(\Bar{f_n}\right) \quad \text{in } \mathbb{R}^d,
\end{equation}
and 
\begin{equation}
\label{translated_uniform_estimBp}
  \| \nabla \Bar{u}_n\|_{L^r_{\textit{unif}}} \leq 1.  
\end{equation}
The idea is now to study the behavior of $\nabla \Bar{u}_n$ on the compact subset $B_1$ in order to precise the behavior of $\nabla u_n$ in $L^r_{unif}$.

\textit{Step 2 : Study of the limit function when $n\rightarrow \infty$.} We next use estimate \eqref{translated_uniform_estimBp} and some elliptic regularity properties associated with equation \eqref{translated_equation_estimBp} to deduce the strong convergence of $\Bar{u}_n$ in $W^{1,r}(B_1)$ to a function $u$ solution to : 
$$- \operatorname{div}(A \nabla u) = 0 \quad \text{in } \mathbb{R}^d,$$
where : 
\begin{itemize}
    \item $A\in \left(L^r_{per} + \mathcal{B}^r(\mathbb{R}^d)\right)^{d \times d}$, $\nabla u \in \left(\mathcal{B}^r(\mathbb{R}^d)\right)^d$ if $x_n$ is bounded,
    \item $A \in \left(L^r_{per} + L^r(\mathbb{R}^d)\right)^{d\times d}$, $\nabla u \in \left(L^r(\mathbb{R}^d)\right)^d$ if $x_n$ is not bounded.
\end{itemize}
In both case, using Proposition \ref{uniquenesslemmaBr} or the results of \cite{blanc2018correctors} in the case of local defects in $L^r$, we deduce that $\nabla u = 0$. Using \eqref{limunifBp}, we finally obtain that the sequence $\nabla u_n$ converges to 0 in $L^r_{unif}$.

\textit{Step 3 : Convergence of $\nabla u_n$ to 0 in $\mathcal{B}^r(\mathbb{R}^d)$ and contradiction.}
We finally remark that for every $n$, $\nabla u_n$ is a solution to : 
$$ - \operatorname{div}(a_{per} \nabla u_n) = \operatorname{div}\left(\Tilde{a}\nabla u_n + f_n\right).$$
Therefore, using estimate \eqref{estimperBP} established in the periodic case, the uniform convergence to 0 of $\nabla u_n$ in $L^r_{unif}$ and the properties of $\Tilde{a}$, we finally deduce the convergence of $\nabla u_n$ to 0 in $\mathcal{B}^r(\mathbb{R}^d)$ and we obtain a contradiction.
\end{proof}

We next establish a result regarding the regularity of the solutions $u$ to \eqref{equationref} such that $\nabla u \in \left(\mathcal{B}^r(\mathbb{R}^d)\right)^d$. As we shall see in the sequel, this result allows to establish Lemma \ref{lemmexistenceBr} using a argument of density induced by Proposition \ref{densityBr}. Since the coefficient $a$ satisfies the property of regularity \eqref{hypothèses2}, it also ensures that the gradient of the corrector given by Theorem \ref{theoreme4} belongs to $\mathcal{C}^{0,\alpha}(\mathbb{R}^d)^d$.

\begin{lemme}
\label{regumarityLemmaBr}
Assume $r\geq 2$. Let $f$ be in $\left(\mathcal{B}^r(\mathbb{R}^d) \cap \mathcal{C}^{0,\alpha}(\mathbb{R}^d)\right)^d$ and $u \in L^1_{loc}(\mathbb{R}^d)$ be a solution to \eqref{equationref} such that $\nabla u \in \left(\mathcal{B}^r(\mathbb{R}^d)\right)^d$. Then $\nabla u$ belongs to $\left(\mathcal{C}^{0,\alpha}(\mathbb{R}^d)\right)^d$. 
\end{lemme}

\begin{proof}
First of all, if $\nabla u$ belongs to $\left(\mathcal{B}
^r(\mathbb{R}^d)\right)^d$, we have in particular that $\nabla u$ is in $\left(L^r_{unif}(\mathbb{R}^d)\right)^d$. In addition, since $r\geq 2$, using the Holder inequality, we have $L^2_{unif}\subset L^r_{unif}$ and the existence of a constant $C_1 = C_1(d,r) >0$ such that : 
$$\| \nabla u\|_{L^2_{unif}} \leq C_1 \| \nabla u \|_{L^r_{unif}}.$$
Since $a\in \left(\mathcal{C}^{0,\alpha}(\mathbb{R}^d)\right)^{d\times d}$, a direct consequence of a regularity result established in [15, Theorem 5.19 p.87] gives the existence of $C_2 >0$ such that for all $x \in \mathbb{R}^d$ :
\begin{align*}
\|\nabla u\|_{\mathcal{C}^{0,\alpha}(B_{1}(x))} &\leq C_2 \left(\| \nabla u\|_{L^2_{unif}(\mathbb{R}^d)} + \|f\|_{\mathcal{C}^{0,\alpha}(\mathbb{R}^d)}\right)\\
& \leq C_2(C_1+1) \left(\| \nabla u\|_{L^r_{unif}(\mathbb{R}^d)} + \|f\|_{\mathcal{C}^{0,\alpha}(\mathbb{R}^d)}\right).
\end{align*}
Since $C_1$ and $C_2$ are independent of $x$, we directly conclude that $\nabla u \in \left(\mathcal{C}^{0,\alpha}(\mathbb{R}^d)\right)^d$.
\end{proof}

We are now able to conclude the study of equation \eqref{equationref} in our particular case. The next Lemma establishes the existence and the uniqueness of a solution $u$ such that $\nabla u \in \left(\mathcal{B}^r(\mathbb{R}^d)\right)^d$ in the case $r\geq2$.  

\begin{lemme}
\label{lemmexistenceBr}
Let $f \in \left(\mathcal{B}^r(\mathbb{R}^d)\right)^d$ and assume $r\geq2$.  There exists $u\in L^1_{loc}(\mathbb{R}^d)$ solution to \eqref{equationref} such that $\nabla u \in \left(\mathcal{B}^r(\mathbb{R}^d)\right)^d$. 
\end{lemme}

\begin{proof}
The proof of this Lemma follows the same pattern as that of Lemma \ref{lemmexistence} in the case $r=2$. Therefore, in the sequel we only explain the main strategy to prove the existence result. For a generic coefficient $a$, we define $\mathbf{P}(a)$ the following assertion : "There exists a solution $u\in \mathcal{D}'(\mathbb{R}^d)$ to :
\begin{equation}
\label{eqconnectednessBr}
    -\operatorname{div}\left(a\nabla u\right) = \operatorname{div}(f) \quad \text{in } \mathbb{R}^d
\end{equation}
such that $\nabla u \in \left(\mathcal{B}^r(\mathbb{R}^d)\cap \mathcal{C}^{0,\alpha}(\mathbb{R}^d)\right)^d$." \\
For $t\in  [0,1]$, we denote $a_t = a_{per} + t\Tilde{a}$ and we define the following set : 
\begin{equation}
\mathcal{I} = \left\{t\in [0,1] \ \middle| \ \forall s \in [0,t], \text{$\mathbf{P}(a_s)$ is true}\right\}.
\end{equation}
Our aim here is to prove that $\mathcal{I}$ is non empty, closed and open for the topology of $[0,1]$ in order to use an argument of connexity adapted from \cite{blanc2018correctors}. First, $\mathcal{I}$ is obviously non empty according to the results of Section \ref{section2Br}. In order to prove that $\mathcal{I}$ is open and closed, we apply exactly the same method as in the proof of Lemma \ref{lemmexistence} using the continuity result of Lemma \ref{lemme2Br} and we conclude. 
\end{proof}

\begin{remark}
\label{Remark_r_inferieur}
Since the coefficient $\Tilde{a}$ and its associated limit $\Tilde{a}_{\infty}$ belong to $L^{\infty}(\mathbb{R}^d)$, we will show in the next section that if $\Tilde{a}$ is in $\mathcal{B}^r(\mathbb{R}^d)$ for $r<2$, then $\Tilde{a}$ also belongs to $\mathcal{B}^2(\mathbb{R}^d)$ and the result of Lemma \ref{lemmexistenceBr} is sufficient to establish Theorem \ref{theoreme4}. 
\end{remark}

\subsection{Existence of the corrector}
\label{section5Br}

In this section, we finally give a proof of Theorem \ref{theoreme4}. To this end, it is important to note that, for every $p \in \mathbb{R}^d$, corrector equation \eqref{correcteur} is equivalent to : 
\begin{equation}
\label{corrector_modifié_Br}
    - \operatorname{div}(a \nabla w_p) = \operatorname{div}(\Tilde{a}\left(\nabla w_{per,p} + p \right)) \quad \text{in } \mathbb{R}^d.
\end{equation}
The idea is therefore to use the results of the previous section showing that the function $\Tilde{a}\left(\nabla w_{per,p} + p \right)$ belongs to $\left(\mathcal{B}^r(\mathbb{R}^d)\right)^{d}$.

\begin{proof}[Proof of theorem \ref{theoreme4}]

For every $p \in \mathbb{R}^d$, our aim here is to find a function $\Tilde{w}_p$ where $\nabla \Tilde{w}_p \in \left(\mathcal{B}^r(\mathbb{R}^d)\right)^d$ and such that $w_{per,p} + \Tilde{w}_p$ is a solution to \eqref{correcteur}. First of all, since $w_{per,p}$ is the periodic corrector solution to \eqref{problemeperiodique}, we have remarked below that this problem is equivalent to finding a function $\Tilde{w}_p$ solution to \eqref{corrector_modifié_Br}.
In addition, under assumption \eqref{hypothèses2} of regularity, it is well known that the function $\nabla w_{per,p}$ belongs to $\left(\mathcal{C}^{0,\alpha}(\mathbb{R}^d)\right)^d$ and we can directly show that the periodicity of $\nabla w_{per,p}$ implies that the function $f = \Tilde{a} \left( \nabla w_{per,p} + p\right)$ belongs to $\left(\mathcal{B}^r(\mathbb{R}^d)\cap\mathcal{C}^{0,\alpha}(\mathbb{R}^d)\right)^d$, as soon as $\Tilde{a}$ belongs to $\left(\mathcal{B}^r(\mathbb{R}^d)\right)^{d\times d}$ and satisfies \eqref{hypothèses2}. In order to use the results established in Section \ref{section4Br}, two different cases can be distinguished depending on the value of $r$.

\textit{First case : $r\geq2$}

Since $f$ belongs to $\left(\mathcal{B}^r(\mathbb{R}^d) \cap \mathcal{C}^{0,\alpha}(\mathbb{R}^d)\right)^d$, the existence of a unique (up to an additive constant) solution $\Tilde{w}_p$ to \eqref{corrector_modifié_Br} such that $\nabla \Tilde{w}_p \in \left(\mathcal{B}^r(\mathbb{R}^d)\cap\mathcal{C}^{0,\alpha}(\mathbb{R}^d)\right)^d$ is a direct consequence of Lemmas \ref{regumarityLemmaBr} and \ref{lemmexistenceBr}. 

\textit{Second case : $r<2$}

In this case we remark that $\Tilde{a}$ actually belongs to $\left(\mathcal{B}^2(\mathbb{R}^d)\right)^{d\times d}$. Indeed, assumption \eqref{hypothèses2} ensures that both $\Tilde{a}$ and $\Tilde{a}_{\infty}$, its associated limit function in $L^r(\mathbb{R}^d)$, are bounded in $L^{\infty}(\mathbb{R}^d)$. Therefore, since $r<2$, we have $L^r\cap L^{\infty} \subset L^2\cap L^{\infty}$ and we obtain that $\Tilde{a}_{\infty}$ is in $\left(L^2(\mathbb{R}^d)\right)^{d \times d}$. Moreover, for every $q \in \mathcal{P}$, we have 
$$\lim_{|q| \rightarrow \infty}\|\Tilde{a} - \tau_{-q} \Tilde{a}_{\infty}\|_{L^2(V_q)}^2 \leq M \lim_{|q| \rightarrow \infty} \|\Tilde{a} - \tau_{-q} \Tilde{a}_{\infty}\|_{L^r(V_q)}^r = 0,$$
where we have denoted by $M = \left(\|\Tilde{a}\|_{L^{\infty}(\mathbb{R}^d)}+ \|\Tilde{a}_{\infty}\|_{L^{\infty}(\mathbb{R}^d)}\right)^{2-r}$.
We deduce that $\Tilde{a}$ belongs to $\left(\mathcal{B}^2(\mathbb{R}^d)\right)^{d \times d}$ and finally, that $f \in \left(\mathcal{B}^2(\mathbb{R}^d)\right)^d$. The existence result of the case $r=2$ implies there exists $\Tilde{w}_p$, solution to \eqref{corrector_modifié_Br}, such that $\nabla \Tilde{w}_p \in \left(\mathcal{B}^2(\mathbb{R}^d) \cap \mathcal{C}^{0, \alpha}(\mathbb{R}^d)\right)^d$. We want to show that $\nabla \Tilde{w}_p$ is actually in $\left(\mathcal{B}^r(\mathbb{R}^d)\right)^d$. First of all, for every $q \in \mathcal{P}$, considering a $2
^q$-translation of equation \eqref{corrector_modifié_Br} and using the periodicity of $a_{per}$, we obtain : 
$$ - \operatorname{div}((a_{per} + \tau_q \Tilde{a}) \tau_q (\nabla \Tilde{w}_p)) = \operatorname{div}(\tau_q f).$$
Letting $|q|$ go to the infinity in the above equation, we obtain : 
$$ - \operatorname{div}((a_{per} + \Tilde{a}_{\infty})\nabla \Tilde{w}_{p, \infty}) = \operatorname{div}(f_{\infty}).$$
Since both $\Tilde{a}_{\infty}$ and $f_{\infty}$ belong to $L^r(\mathbb{R}^d) \cap L^2(\mathbb{R}^d)$, a result of uniqueness established in \cite[Proposition 2.1]{blanc2018correctors} ensures that $\nabla \Tilde{w}_{p, \infty} \in \left(L^r(\mathbb{R}^d)\right)^d$. In addition, for every $q \in \mathcal{P}$, we have : 
\begin{align*}
    \Tilde{a} \nabla \Tilde{w}_p - \tau_{-q}\left( \Tilde{a}_{\infty} \nabla \Tilde{w}_{p,\infty}\right) & = \left(\Tilde{a} - \tau_{-q} \Tilde{a}_{\infty}\right) \nabla \Tilde{w}_{p} + \tau_{-q}\Tilde{a}_{\infty}\left(\nabla \Tilde{w}_p - \tau_{-q} \nabla \Tilde{w}_{p,\infty}\right) .\\
    & = I_q^1 + I_q^2.
\end{align*}
Since $\nabla \Tilde{w}_p$ is bounded in $L^{\infty}(\mathbb{R}^d)$ and $\Tilde{a} \in \left(\mathcal{B}^r(\mathbb{R}^d)\right)^{d\times d}$ we directly obtain : 
$$\lim_{|q| \rightarrow \infty} \| I_q^1\|_{L^r(V_q)} = 0.$$ 
Next, we consider $s>1$ such that $\displaystyle \frac{1}{r} = \frac{1}{2} + \frac{1}{s}$. Such a real $s$ always exists since $r<2$ and, in particular, we have $s>r$. Using the Holder inequality and the fact that $\Tilde{a}\in \left(L^r(\mathbb{R}^d) \cap L^{\infty}(\mathbb{R}^d)\right)^{d \times d}$, we obtain : 
\begin{align*}
    \lim_{|q| \rightarrow \infty} \| I_q^2\|_{L^r(V_q)} &\leq  \| \Tilde{a}_{\infty}\|_{L^s(V_q)} \lim_{|q| \rightarrow \infty} \| \nabla \Tilde{w}_p - \tau_{-q} \nabla \Tilde{w}_{p,\infty}\|_{L^2(V_q)}\\
    & \leq  \| \Tilde{a}_{\infty}\|_{L^s(\mathbb{R}^d)}\lim_{|q| \rightarrow \infty} \| \nabla \Tilde{w}_p - \tau_{-q} \nabla \Tilde{w}_{p,\infty}\|_{L^2(V_q)}\\
    & \leq C  \| \Tilde{a}_{\infty}\|_{L^r(\mathbb{R}^d)}\lim_{|q| \rightarrow \infty} \| \nabla \Tilde{w}_p - \tau_{-q} \nabla \Tilde{w}_{p,\infty}\|_{L^2(V_q)} = 0.
\end{align*}
In the last inequality, the constant $C$ only depends on $\|\Tilde{a}_{\infty}\|_{L^{\infty}(\mathbb{R}^d)}$. We finally obtain that 
$$  \lim_{|q| \rightarrow \infty} \| \Tilde{a} \nabla \Tilde{w}_p - \tau_{-q}\left( \Tilde{a}_{\infty} \nabla \Tilde{w}_{p,\infty}\right)\|_{L^r(V_q)}=0.$$
In addition, $\Tilde{a}_{\infty}$ is bounded and we decuce that $\Tilde{a}_{\infty} \nabla \Tilde{w}_{p,\infty} \in \left(L^r(\mathbb{R}^d)\right)^d$. We therefore deduce that $\Tilde{a} \nabla \Tilde{w}_p$ belongs to $\left(\mathcal{B}^r(\mathbb{R}^d)\right)^d$ with $\left(\Tilde{a} \nabla \Tilde{w}_p\right)_{\infty} = \Tilde{a}_{\infty} \nabla \Tilde{w}_{p,\infty}$. To conclude, we finally remark that equation \eqref{corrector_modifié_Br} is equivalent to : 
$$ - \operatorname{div}(a_{per} \nabla \Tilde{w}_p) = \operatorname{div}(\Tilde{a} \nabla \Tilde{w}_p + f).$$
The existence result of Lemma \ref{lemmeexistenceperiodiqueBP} and the uniqueness result of Lemma \ref{uniquess_periodic_Br} in the case $a = a_{per}$ therefore ensure that $\nabla \Tilde{w}_p $ belongs to $\left(\mathcal{B}^r(\mathbb{R}^d)\right)^d$ and we can conclude the proof. 
\end{proof}

\subsection{Homogenization results and convergence rates}
\label{section6Br}

In this section we generalize the method employed in Section \ref{Section5} in the case $r=2$ in order to establish the homogenization of equation \eqref{correcteur} and, in particular, the results of Theorem~\ref{theoreme5}. To this end, we use the corrector $w = \left(w_{e_i}\right)_{i \in \{1,...,d\}}$ given by Theorem \ref{theoreme4}. 

A first crucial step is to determine the limit of the sequence $\left(u^{\varepsilon}\right)_{\varepsilon>0}$ solutions to \eqref{equationepsilon}. First of all, Proposition \ref{moyenneBr} ensures that both the coefficient $\Tilde{a}$ and the gradient of the corrector $\Tilde{w_i}= \Tilde{w}_{e_i}$ (for every $i \in \{1,...,d\}$) have average value zero in the sense \eqref{zeroaverageBr} since they belong to $\mathcal{B}
^r(\mathbb{R}^d)$. In particular, for every $1<r<\infty$, the perturbations of $\mathcal{B}^r(\mathbb{R}^d)$ do not impact the periodic background on average and we can easily generalize the results of Proposition \ref{proplimithomogenization} to the case $r\neq2$. Therefore, we obtain that the the limit $u^*$ (weak-$H^1(\Omega)$ and strong-$L^2(\Omega)$) of the sequence $u^{\varepsilon}$  is actually a solution to \eqref{homog} where $a
^*$ is the same homogenized coefficient as in the periodic case. 

In the sequel, we study the behavior of the approximated sequence of solutions $u^{\varepsilon,1} = u^* + \varepsilon \sum_{i=1}^d \partial_{i}u^* w_{i}(./ \varepsilon)$. In particular, we established here the convergence rates stated in Theorem \ref{theoreme5} describing the convergence to zero of the sequence $R^{\varepsilon} = u^{\varepsilon} - u^{\varepsilon,1}$. We follow here the method employed in Section \ref{Section5} and we consider the divergence-free matrix defined by $M_k^i = a_{i,k}^* - \sum_{j=1}^d a_{i,j}(\delta_{j,k} + \partial_j w_k) \in L
^2_{per} + \mathcal{B}^r(\mathbb{R}^d)$. In our case, the existence of a potential in the form $B = B_{per} + \Tilde{B} \in L^2_{per}(\mathbb{R}^d) + \mathcal{B}^r(\mathbb{R}^d)$ solution to \eqref{equationdefB} is given by the following Lemma. 

\begin{lemme}
Let $\Tilde{M} = \big( \Tilde{M}^i_k\big)_{1 \leq i,k\leq d} \in \mathcal{B}^r(\mathbb{R}^d)^{d \times d}$ such that $\operatorname{div}(\Tilde{M}_k) = 0$, for every $k \in \{1, ...,d\}$. Then there exists a potential $\Tilde{B}^{i,j}_k \in W^{1,r}_{loc}(\mathbb{R}^d)$ such that $\nabla \Tilde{B} \in \mathcal{B}^r(\mathbb{R}^d)$ and for all $i,j,k \in \{1,...,d\}$ :
\begin{align}
\label{potentielperturbeBr}
- \Delta &\Tilde{B}^{i,j}_k = \partial_j\Tilde{M}_k^i - \partial_i \Tilde{M}_k^j,\\
\label{antisymetrieBr}
& \Tilde{B}_k^{i,j} = - \Tilde{B}_k^{j,i}, \\
\label{divergencepotBr}
& \sum_{i=1}^{d} \partial_i \Tilde{B}_k^{i,j} = \Tilde{M}_k^j.
\end{align}
In addition, there exists a constant $C_1>0$ which only depends  of the ambient dimension $d$ and such that : 
\begin{equation}
\label{estimpotentielBr}
\|\nabla \Tilde{B}\|_{\mathcal{B}^r(\mathbb{R}^d)} \leq C_1 \|\Tilde{M}\|_{\mathcal{B}^r(\mathbb{R}^d)}.   
\end{equation}
\end{lemme}

This result is actually a consequence of Lemma \ref{lemmeexistenceperiodiqueBP}. For the details of the proof, we refer the reader to associated Lemma \ref{lemme_existence_B} in the case $r=2$. Now that the existence of $B$ has been deal with, we need to study the behavior of the sequences $\varepsilon \Tilde{w}(./ \varepsilon)$  and $\varepsilon \Tilde{B}(./ \varepsilon)$ in $L^{\infty}(\mathbb{R}
^d)$ when $\varepsilon \rightarrow 0$. We shall see that the case $r>d$ is a consequence of estimate \eqref{souslineariteBr} established in Proposition \ref{propsouslineariteBr}. The case $r<d$ is studied in the next lemma.

\begin{lemme}
\label{lemmeborneBr}
Assume $r<d$. Then, the corrector $w=\left(w_i\right)_{i \in \{1,...,d\}}$ defined by Theorem \ref{theoreme4} and the potential $B$ solution to \eqref{equationdefB} are in $L^{\infty}(\mathbb{R}^d)$. 
\end{lemme}

\begin{proof}

For all $i \in \{1,...,d\}$, it is important to remark that $\Tilde{w}_i$ is also a solution to : 
$$-\operatorname{div} \left(a_{per} \nabla \Tilde{w}_i \right) = \operatorname{div}\left(\Tilde{a}\left(e_i + \nabla w_{per,i} + \nabla \Tilde{w}_i \right)\right).$$
We know the gradient of the corrector defined in Theorem \ref{theoreme4} is in $\mathcal{C}^{0,\alpha}(\mathbb{R}^d)$. Next, using Property \ref{stabilityBr}, we deduce that $f = \Tilde{a}\left(e_i + \nabla w_{per,i} + \nabla \Tilde{w}_i \right) \in L^{\infty}(\mathbb{R}^d) \cap \mathcal{B}^r(\mathbb{R}^d)$. Lemmas \ref{lemmeexistenceperiodiqueBP} and \ref{uniquess_periodic_Br} therefore ensure that for every $i \in \{1,...,d\}$, $\Tilde{w}_i$ is defined by : 
\begin{equation}
\Tilde{w}_i(x) = \int_{\mathbb{R}^d} \nabla_y G_{per}(x,y)f(y) dy.
\end{equation}
We want to prove that $\Tilde{w}_i$ actually belongs to $L^{\infty}(\mathbb{R}^d)$. 
In order to prove that the integral is bounded independently of $x$, we follow step by step the method used in the proof of Lemma \ref{lemmeborne}. We first fix $x\in \mathbb{R}^d$ and denote $p_x$ the unique element of $\mathcal{P}$ such that $x \in \textit{V}_{p_x}$. We define $W_{p_{x}} = W_{2^{p_x}}$ such as in Proposition \ref{openconstruction} and we split the integral in three parts  : 
\begin{align*}
\int_{\mathbb{R}^d} \nabla_y G_{per}(x,y)f(y) dy & = \int_{B_1(x)} \nabla_y G_{per}(x,y)f(y) dy +  \int_{W_{p_x} \setminus B_1(x)} \nabla_y G_{per}(x,y)f(y) dy \\
& +  \int_{ \mathbb{R}^d \setminus W_{p_x}} \nabla_y G_{per}(x,y)f(y) dy  = I_1(x) + I_2(x) + I_3(x).
\end{align*}
\\
Firstly, using estimate \eqref{estimgreen1} for the Green function, we obtain
\begin{align*}
\left| I_1(x) \right| &  \leq C \|f\|_{
L^{\infty}(\mathbb{R}^d)} \int_{B_1(x)} \dfrac{1}{\left|x-y\right|^{d-1}}dy \leq C \|f\|_{
L^{\infty}(\mathbb{R}^d)}.
\end{align*}
Where $C$ denotes a positive constant independent of $x$.

Next, we know there exists $C_1>0$ and $C_2>0$ independent of $x$ such that $W_{p_x} \subset B_{C_12^{p_x}(x)}$ and the number of $q \in \mathcal{P}$ such that $V_q \cap W_{p_x} \neq \emptyset$ is bounded by $C_2$ (as a consequence of Proposition \ref{openconstruction}). We therefore use the Holder inequality and we obtain :
\begin{align*}
|I_2(x)| & \leq \int_{W_{p_x}\setminus{B_1(x)}} \dfrac{1}{|x-y|^{d-1}} |f(y)|dy \\
& \leq C_2 \left(\int_{ B_{C_1 2^{p_x}(x)} \setminus{B_1(x)}} \dfrac{1}{|x-y|^{r'(d-1)}} dy\right)^{1/r'} \sup_{p \in \mathcal{P}} \|f\|_{L^r(V_q)},
\end{align*}
where we have denoted by $\displaystyle r' = \frac{r}{r-1}$.

In addition, since $r<d$, we have $\displaystyle (r'-1)(d-1) > 1$ and we deduce that : 
$$\int_{ B_{C_1 2^{p_x}(x)} \setminus{B_1(x)}} \frac{1}{\left|x-y\right|^{r'(d-1)}} dy = \int_{ B_{C_1 2^{p_x}(0)} \setminus{B_1(0)}} \dfrac{1}{\left|y\right|^{r'(d-1)}} dy \leq C\left(1 - \dfrac{1}{2^{|p_x|((r'-1)(d-1)-1)}}\right).$$
We finally obtain :
$$I_2(x) \leq C  \sup_{p \in \mathcal{P}} \|f\|_{L^r(V_q)} \left(1 - \dfrac{1}{2^{|p_x|((r'-1)(d-1)-1)}}\right)^{1/r'} \leq C  \sup_{p \in \mathcal{P}} \|f\|_{L^r(V_q)}.$$

Finally, to bound $I_3(x)$ we split the integral on each cell $V_q$ for $q \in \mathcal{P}$. Using the Holder inequality, we obtain : 
\begin{align*}
\left|I_3(x)\right| & \leq  \sum_{q \in \mathcal{P}} \int_{ V_q \setminus{W_{p_x}}} \left|\nabla_y G_{per}(x,y)f(y) \right| dy  \\
 & \leq \|f\|_{\mathcal{B}^r(\mathbb{R}^d)} \sum_{q \in \mathcal{P} } \left( \int_{V_q \setminus{W_{p_x}}} \left|\nabla_y G_{per}(x,y)\right| ^{r'} dy \right)^{1/r'}.
\end{align*}

Using again estimate \eqref{estimgreen1}, we have : 
\begin{align*}
\sum_{q \in \mathcal{P}} \left( \int_{V_q \setminus{W_{p_x}}} \left| \nabla_y G_{per}(x,y) \right|^{r'} dy \right)^{\frac{1}{r'}} & \leq C \sum_{q \in \mathcal{P}} \left( \int_{V_q \setminus{W_{p_x}}} \dfrac{1}{\left|x-y\right|^{r'(d-1)}} dy \right)^{1/r'}.
\end{align*}
Again, for every $q \in \mathcal{P_{C_0}}$ we have $|V_q| \leq C 2^{|q|}$, and the properties of $W_{p_x}$ (see proposition \ref{openconstruction} for details) ensure that $|x-y| \geq C 2^{|q|}$. We deduce that 
\begin{align*}
 \left|I_3(x)\right| & \leq C \|f\|_{\mathcal{B}^r(\mathbb{R}^d)} \sum_{q \in \mathcal{P}}  \frac{1}{2^{|q|\left(\frac{(r'-1)d}{r'}-1\right)}} < \infty. 
\end{align*}

Finally, we have bounded $\Tilde{w}_i(x)$ independently of $x$ and we deduce that $\Tilde{w}_i \in L^{\infty}(\mathbb{R}^d)$. With the same method we obtain the same result for $B = B_{per} + \Tilde{B}$ which allows us to conclude. 
\end{proof}

We next remind that for every $\varepsilon>0$, the function $R^{\varepsilon}$ is a solution to 
\begin{equation*}
-\operatorname{div}\left(a\left(\dfrac{x}{\varepsilon}\right)\nabla R^{\varepsilon}\right) = \operatorname{div}(H^{\varepsilon}) \quad \text{in } \Omega,
\end{equation*}
where $H^{\varepsilon}$ is defined by \eqref{defHeps}. We are now able to give a complete proof of Theorem \ref{theoreme5} using the results previously established in this study. 

\begin{proof}[Proof of Theorem \ref{theoreme5}]
Here we can exactly repeat the different steps of the proof of Theorem \ref{theoreme3} in the case $r=2$ and we obtain the following estimates : 
\begin{equation}
\label{estimH_Br}
     \|H^{\varepsilon}\|_{L^2(\Omega)} \leq 
C_1 \left(\| \varepsilon w(./\varepsilon) \|_{L^{\infty}(\Omega)} + \| \varepsilon B(./\varepsilon) \|_{L^{\infty}(\Omega)} \right) \| f\|_{L^2(\Omega)} ,
\end{equation}
\begin{equation}
\label{estimationRBr}
\| R^{\varepsilon}\|_{L^2(\Omega)} \leq C_2 \left( \left(\| \varepsilon w(./\varepsilon) \|_{L^{\infty}(\Omega)} + \| \varepsilon B(./\varepsilon) \|_{L^{\infty}(\Omega)} \right) \|f\|_{L^2(\Omega)}   + \|H^{\epsilon}\|_{L^2(\Omega)}\right),
\end{equation}
and for every $\Omega_1 \subset \subset \Omega$ : 
\begin{equation}
\label{estimationnablaRBr}
\|\nabla R^{\varepsilon} \|_{L^2(\Omega_1)} \leq C_3 \left( \| H^{\varepsilon} \|_{L^2(\Omega)} + \| R^{\varepsilon} \|_{L^2(\Omega)} \right),
\end{equation}
where $C_1>0$, $C_2>0$ and $C_3>0$ are independent of $\varepsilon$. We note here that we want to bound $\varepsilon w(./\varepsilon)$ and $\varepsilon B(./\varepsilon)$ in $L^{\infty}(\mathbb{R}^d)$ in order to establish estimates \eqref{estimate1Bp} and \eqref{estimate2Bp}. First, it is well known that $w_{per}$ and $B_{per}$ are in $L^{\infty}$. Secondly, since both $\nabla \Tilde{w}$ and $\nabla \Tilde{B}$ belong to $\left(\mathcal{B}^r(\mathbb{R}^d)\cap L^{\infty}(\mathbb{R}
^d)\right)^d$, Proposition \ref{propsouslineariteBr} (for the case $r>d$) and Lemma \ref{lemmeborneBr} (for the case $r<d$) ensure the existence of a constant $C>0$ independent of $\varepsilon$ such that : 
\begin{align}
\label{estimLinfiniwBr}
    \|\varepsilon \Tilde{w}(./\varepsilon)\|_{L^{\infty}(\mathbb{R}^d)} & \leq C \left(\log |\varepsilon| \right)^{\mu_r} \varepsilon^{\nu_r}, \\
    \label{estimLinifiniBBr}
    \|\varepsilon \Tilde{B}(./\varepsilon)\|_{L^{\infty}(\mathbb{R}^d)} & \leq C \left(\log |\varepsilon| \right)^{\mu_r} \varepsilon^{\nu_r},
\end{align}
where $\nu_r$ and $\mu_r$ are defined by \eqref{def_nu_r} and \eqref{def_mu_r}. To conclude, we finally use \eqref{estimH_Br}, \eqref{estimationRBr}, \eqref{estimationnablaRBr}, \eqref{estimLinfiniwBr} and \eqref{estimLinifiniBBr} and we obtain : 
\begin{equation*}
\|R^{\varepsilon}\|_{L^2(\Omega)} \leq C \left(\log |\varepsilon| \right)^{\mu_r} \varepsilon^{\nu_r} \|f\|_{L^2(\Omega)},
\end{equation*}
and
\begin{equation*}
\|\nabla R^{\varepsilon}\|_{L^2(\Omega_1)} \leq \Tilde{C} \left(\log |\varepsilon| \right)^{\mu_r} \varepsilon^{\nu_r} \|f\|_{L^2(\Omega)},
\end{equation*}
where $C$ and $\Tilde{C}$ are independent of $\varepsilon$. We have proved Theorem \ref{theoreme5}.

\end{proof}
\end{spacing}

\appendix


\begin{thebibliography}{99}

\bibitem{allaire1992homogenization} G. Allaire,  \textit{Homogenization and two-scale convergence}, SIAM Journal on Mathematical Analysis 23, no.6, pp 1482-1518, 1992.

\bibitem{anantharaman2011asymptotic} A. Anantharaman, X. Blanc, F.  Legoll, \textit{Asymptotic behaviour of Green functions of divergence form operators with periodic coefficients}, Applied Mathematics Research Express, vol. 2013 (1), 79-101, 2013.

\bibitem{avellaneda1987compactness} M. Avellaneda and F.H. Lin, \textit{Compactness methods in the theory of homogenization},  Communications on
Pure and Applied Mathematics 40, no. 6, pp 803 - 847, 1987.  

\bibitem{avellaneda1989compactness} M. Avellaneda and F.H. Lin, \textit{Compactness methods in the theory of homogenization II: Equations in non-divergence form},  Communications on
Pure and Applied Mathematics 42, no. 2, pp 139 - 172, 1989.  


\bibitem{avellaneda1991lp} M. Avellaneda and F.H. Lin, \textit{$L^p$ bounds on singular integrals in homogenization},  Communications on pure and applied mathematics 44, no.8-9, pp 897 - 910, 1991.


\bibitem{bensoussan2011asymptotic} A. Bensoussan, J. L. Lions, G. Papanicolaou,  \textit{Asymptotic analysis for periodic structures}, Studies in Mathematics and its Applications, 5. North-Holland Publishing Co., Amsterdam-New York, 1978.

\bibitem{blanc2018precised} X. Blanc, M. Josien, C. Le Bris,  \textit{Precised approximations in elliptic homogenization beyond the periodic setting}, Asymptotic Analysis, 116(2), 93–137, 2020.

\bibitem{blanc2019correctors} X. Blanc, C. Le Bris, P-L. Lions,  \textit{On correctors for linear elliptic homogenization in the presence of local defects: The case of advection-diffusion}, Journal de Math{\'e}matiques Pures et Appliqu{\'e}es 124, pp 106-122, 2019.

\bibitem{blanc2018correctors} X. Blanc, C. Le Bris, P-L. Lions,  \textit{On correctors for linear elliptic homogenization in the presence of local defects}, Communications in Partial Differential Equations 43, no.6, pp 965-997, 2018.

\bibitem{blanc2015local} X. Blanc, C. Le Bris, P-L. Lions,  \textit{Local profiles for elliptic problems at different scales: defects in, and interfaces between periodic structures}, Communications in Partial Differential Equations 40, no.12, pp 2173-2236, 2015.

\bibitem{blanc2012possible} X. Blanc, C. Le Bris, P-L. Lions,  \textit{A possible homogenization approach for the numerical simulation of periodic microstructures with defects}, Milan Journal of Mathematics 80, no.2, pp 351-367, 2012.


\bibitem{brezis2008conjecture} H. Brezis, \textit{On a conjecture of J. Serrin}, Atti della Accademia Nazionale dei Lincei, Classe di Scienze Fisiche, Matematiche e Naturali, Rendiconti Lincei Matematica e Applicazioni 19, no.4, pp 335–338, 2008. 



\bibitem{deny1954espaces} J. Deny, J.L. Lions,  \textit{Les espaces du type de Beppo Levi}, Annales de l'institut Fourier 5, pp 305-370, 1954.


\bibitem{evans10} L.C. Evans, \textit{Partial Differential Equations}, Graduate Studies in Mathematics, 19. American Mathematical Society, Providence, RI, 1998.


\bibitem{giaquinta1983multiple} M. Giaquinta, \textit{Multiple integrals in the calculus of variations and nonlinear elliptic systems}, Princeton University Press, 1983.



\bibitem{MR3099262} M. Giaquinta, L. Martinazzi, \textit{An introduction to the regularity theory for elliptic systems, harmonic maps and minimal graphs}, Lecture Notes Scuola Normale Superiore di Pisa (New Series), Volume 11, Edizioni della Normale, Pisa, Second edition, 2012.

\bibitem{goudey} R. Goudey,  PhD thesis, in preparation.

\bibitem{gruter1982green} M. Gruter, K.O. Widman,  \textit{The Green function for uniformly elliptic equations}, Manuscripta Mathematica 37, no.3, pp 303-342, 1982.

\bibitem{jikov2012homogenization} V.V. Jikov, S.M Kozlov, O.A. Oleinik,  \textit{Homogenization of differential operators and integral functionals}, Springer Science \& Business Media, 2012.

\bibitem{josien2018etude} M. Josien, \textit{Etude math\'ematique et num\'erique de quelques mod\`eles multi-\'echelles issus de la m\'ecanique des mat\'eriaux}, thesis, 2018.

\bibitem{moser1961harnack} J. Moser,  \textit{On Harnack's theorem for elliptic differential equations}, Communications on Pure and Applied Mathematics 14, no.3, pp 577-591, 1961.


\bibitem{AIHPC_1984__1_4_223_0} P.-L. Lions, \textit{The concentration-compactness principle in the calculus of variations. The locally compact case, Parts 1 \& 2}, Ann. Inst. H. Poincar\'e 1, 109-145 and 223-283, 1984.


\bibitem{tartar2009general} L. Tartar,  \textit{The general theory of homogenization: a personalized introduction}, volume 7, Springer Science \& Business Media, 2009.




\end{thebibliography}
\end{document}